\documentclass{amsart}
\usepackage{amssymb}

\usepackage{amscd}
\input{tcilatex}
\input tcilatex

\begin{document}
\author{Marek Kara\'{s}}
\title{Multidegrees of tame automorphisms of $\Bbb{C}^{n}$}
\date{}
\maketitle

\begin{abstract}
Let $F=\left( F_{1},\ldots ,F_{n}\right) :\Bbb{C}^{n}\rightarrow \Bbb{C}^{n}$
be a polynomial mapping. By the multidegree of $F$ we mean $\limfunc{mdeg}%
F=\left( \deg F_{1},\ldots ,\deg F_{n}\right) \in \Bbb{N}^{n}.$

The aim of this paper is to study the following problem (especially for $%
n=3) $: \textit{for which sequence }$\left( d_{1},\ldots ,d_{n}\right) \in 
\Bbb{N}^{n}$\textit{\ is there a tame automorphism }$F$ \textit{of }$\Bbb{C}%
^{n}$\textit{\ such that }$\limfunc{mdeg}F=\left( d_{1},\ldots ,d_{n}\right)
?$ In other words we investigate the set $\limfunc{mdeg}\left( \limfunc{Tame}%
\left( \Bbb{C}^{n}\right) \right) ,$ where $\limfunc{Tame}\left( \Bbb{C}%
^{n}\right) $ denotes the group of tame automorphisms of $\Bbb{C}^{n}.$

Since $\allowbreak \allowbreak \limfunc{mdeg}\left( \limfunc{Tame}\left( 
\Bbb{C}^{n}\right) \right) $ is invariant under permutations of coordinates,
we may focus on the set $\left\{ \left( d_{1},\ldots ,d_{n}\right)
:d_{1}\leq \ldots \leq d_{n}\right\} \cap \limfunc{mdeg}\left( \limfunc{Tame}%
\left( \Bbb{C}^{n}\right) \right) .$

Obviously, we have $\left\{ \left( 1,d_{2},d_{3}\right) :1\leq d_{2}\leq
d_{3}\right\} \cap \allowbreak \limfunc{mdeg}\left( \limfunc{Tame}\left( 
\Bbb{C}^{3}\right) \right) =$\newline
$\left\{ \left( 1,d_{2},d_{3}\right) :\allowbreak 1\leq d_{2}\leq
d_{3}\right\} .$ Not obvious, but still easy to prove is the equality $%
\left\{ \left( 2,d_{2},d_{3}\right) :2\leq d_{2}\leq d_{3}\right\} \cap
\allowbreak \limfunc{mdeg}\left( \limfunc{Tame}\left( \Bbb{C}^{3}\right)
\right) =\allowbreak \left\{ \left( 2,d_{2},d_{3}\right) :2\leq d_{2}\leq
d_{3}\right\} .$

In the paper, among other things, we give a complete description of the sets 
$\left\{ \left( 3,d_{2},d_{3}\right) :3\leq d_{2}\leq d_{3}\right\} \cap
\allowbreak \limfunc{mdeg}\left( \limfunc{Tame}\left( \Bbb{C}^{3}\right)
\right) $ and $\left\{ \left( 5,d_{2},d_{3}\right) :5\leq d_{2}\leq
d_{3}\right\} \cap \allowbreak \limfunc{mdeg}\left( \limfunc{Tame}\left( 
\Bbb{C}^{3}\right) \right) .$ In the examination of the last set the most
difficult part is to prove that $\left( 5,6,9\right) \notin \limfunc{mdeg}%
\left( \limfunc{Tame}\left( \Bbb{C}^{3}\right) \right) .$ To do this, we use
the two dimensional Jacobian Conjecture (which is true for low degrees) and
the Jung-van der Kulk Theorem.

As a surprising consequence of the method used in proving that $\left(
5,6,9\right) \notin \limfunc{mdeg}\left( \limfunc{Tame}\left( \Bbb{C}%
^{3}\right) \right) ,$ we show that the existence of a tame automorphism $F$
of $\Bbb{C}^{3}$ with $\limfunc{mdeg}F=\left( 37,70,105\right) $ implies
that the two dimensional Jacobian Conjecture is not true.

Also, we give a complete description of the following sets:\newline
$\left\{ \left( p_{1},p_{2},d_{3}\right) :2<p_{1}<p_{2}\leq d_{3},\text{ }%
p_{1},p_{2}\text{ prime numbers}\right\} \cap \allowbreak \limfunc{mdeg}%
\left( \limfunc{Tame}\left( \Bbb{C}^{3}\right) \right) ,$ $\left\{ \left(
d_{1},d_{2},d_{3}\right) :d_{1}\leq d_{2}\leq d_{3},\text{ }d_{1},d_{2}\in 2%
\Bbb{N}+1,\text{ }\gcd \left( d_{1},d_{2}\right) =1\right\} \cap \allowbreak 
\limfunc{mdeg}\left( \limfunc{Tame}\left( \Bbb{C}^{3}\right) \right)
.\allowbreak $ Using the description of the last set we show that $\limfunc{%
mdeg}\left( \limfunc{Aut}\left( \Bbb{C}^{3}\right) \right) \backslash $%
\linebreak $\limfunc{mdeg}\left( \limfunc{Tame}\left( \Bbb{C}^{3}\right)
\right) $ is infinite.

We also obtain a (still incomplete) description of the set\linebreak $%
\left\{ \left( 4,d_{2},d_{3}\right) :4\leq d_{2}\leq d_{3}\right\} \cap
\allowbreak \limfunc{mdeg}\left( \limfunc{Tame}\left( \Bbb{C}^{3}\right)
\right) $ and we give complete information about $\limfunc{mdeg}F^{-1}$ for $%
F\in \limfunc{Aut}\left( \Bbb{C}^{2}\right) .$
\end{abstract}

\setcounter{section}{-1}

\section{Introduction}

The object of principal interest in this paper is the multidegree (i.e. the
sequence of the degrees of the coordinate functions) of a polynomial
automorphism of the vector space $\Bbb{C}^{n}.$ Let us mention that in the
Scotish Book (\cite{scotishbook}, Problem 79) Mazur and Orlicz posed the
following question: ``If $F=\left( F_{1},\ldots ,F_{n}\right) :\Bbb{C}%
^{n}\rightarrow \Bbb{C}^{n}$ is a one-to-one polynomial map whose inverse is
also a polynomial map, is each $F_{i}$ of degree one?'' In other words, they
asked whether every polynomial automorphism of $\Bbb{C}^{n}$ has multidegree 
$\left( 1,\ldots ,1\right) .$ The answer to this question is obviously
``no'', and in the Scotish Book itself one can find the following example:
let $1\leq i\leq n$ and $a=a\left( X_{1},\ldots ,X_{i-1},X_{i+1},\ldots
,X_{n}\right) \in \Bbb{C}\left[ X_{1},\ldots ,X_{i-1},X_{i+1},\ldots
,X_{n}\right] .$ Then 
\begin{equation*}
E:\Bbb{C}^{n}\ni \left( x_{1},\ldots ,x_{n}\right) \,\mapsto \left(
x_{1},\ldots ,x_{i-1},x_{i}+a,x_{i+1},\ldots ,x_{n}\right) \in \Bbb{C}^{n}
\end{equation*}
is a polynomial automorphism with multidegree $\left( 1,\ldots ,1,\deg
a,1,\ldots ,1\right) .$ A map as above is called an \textit{elementary}
polynomial map. Taking finite compositions of such elementary maps and
elements of the affine subgroup $\limfunc{Aff}\left( \Bbb{C}^{n}\right) ,$
i.e. the group of polynomial automorphisms $F=\left( F_{1},\ldots
,F_{n}\right) :\Bbb{C}^{n}\rightarrow \Bbb{C}^{n}$ such that $\deg F_{i}=1$
for all $i,$ we get automorphisms called \textit{tame.}

In 1942 Jung \cite{Jung} proved that each polynomial automorphism of $k^{2},$
where $k$ is a field of characteristic zero, is tame. Later, in 1953, van
der Kulk extended Jung's result to fields of arbitrary characteristic. Since
then several authors have given other proofs of that result: Gutwirth \cite
{Gutwirth} in 1961, Shafarevich \cite{Shafarevich} in 1966, Rentschler \cite
{Rentschler} in 1968, Makar-Limanov \cite{Makar-Lim} in 1970, Nagata \cite
{Nagata} in 1972, Abhyankar and Moh \cite{Abhyankar-Moh} in 1975, Dicks \cite
{Dicks} in 1983, McKay and Wang \cite{MacKayWng2} in 1988. The stronger
statement, also called the Shafarevich-Nagata-Kombayashi theorem, saying
that the group of all polynomial automorphisms of $k^{2}$ is the amalgamated
product of the affine subgroup and the subgroup of de Jonqui\`{e}res
automorphisms over their intersection, can be found in \cite{Kulk}, \cite
{Kambayashi}, \cite{Nagata}, \cite{Dicks}, \cite{Alperin} and without proof
in \cite{Shafarevich}.

From the result of Jung and van der Kulk it also follows that if $\left(
d_{1},d_{2}\right) $ is the multidegree of an automorphism of $\Bbb{C}^{2},$
then $d_{1}|d_{2}$ or $d_{2}|d_{1}$ (see subsection \ref
{subsection)jung_kulk}).

Tame automorphisms are closely related to the problem of embedding of affine
algebraic varieties. For example, in the proof of the famous
Abhyankar-Moh-Suzuki theorem, saying that every embedding of a line in $\Bbb{%
C}^{2}$ is rectifiable (i.e. a composition of the standard embedding $\Bbb{C}%
\ni x\mapsto \left( x,0\right) \in \Bbb{C}^{2}$ and an automorphism of $\Bbb{%
C}^{2}$), tame automorphisms play a prominent role. This result, formulated
in algebraic terms as follows: if $f(T),g(T)\in k\left[ T\right] $ and $%
k\left[ f(T),g(T)\right] =k\left[ T\right] ,$ then either $\deg f(T)|\deg
g(T)$ or $\deg g(T)|\deg f(T),$ was used by Segree \cite{Segree} to
``prove'' the Jacobian Conjecture. The problem of embeddings of affine
algebraic varieties was also considered by Jelonek \cite{jel1,jel2,jel3},
Kaliman \cite{kalim}, Srinivas \cite{srinivas} and Craighero \cite{Craighero}%
.

Since Jung and van der Kulk proved their theorem, many authors have tried to
prove or disprove the similar result for dimension $n\geq 3,$ but without
any results. The most famous candidate for a so-called wild automorphism
(i.e. one that is not tame) was proposed by Nagata in 1972. It took more
than thirty years to prove that the Nagata automorphism 
\begin{equation*}
\sigma :\Bbb{C}^{3}\ni (x,y,z)\mapsto
(x+2y(y^{2}+zx)-z(y^{2}+zx)^{2},y-z(y^{2}+zx),z)\in \Bbb{C}^{3}
\end{equation*}
is indeed wild. This remarkable result was obtained by Shestakov and
Umirbaev \cite{sh umb1}. The two main ingredients in the proof of the above
result are recalled as Theorems \ref{tw_deg_g_fg} and \ref{tw_type_1-4} (see
subsections \ref{subsection_Poisson_bracket} and \ref
{subsection_reductions_I_IV}). These two theorems are also basic tools in
our considerations concerning multidegrees of tame automorphisms of $\Bbb{C}%
^{3}.$

The paper is organized as follows. In section \ref{section_notations_and} we
fix notation, recall basic definitions, and discuss the multidegree of
polynomial automorphisms of $\Bbb{C}^{2}$ (see subsection \ref
{subsection)jung_kulk}). The discussion is based on the Jung-van der Kulk
result. In section \ref{section_main_tools} we recall the notion of a
Poisson bracket of two polynomials, and two theorems due to Shestakov and
Umirbaev (Theorems \ref{tw_deg_g_fg} and \ref{tw_type_1-4}). They are the
main tools used in the paper. We also prove that the degree of the Poisson
bracket is an invariant of a linear change of coordinates (Lemma \ref
{lem_degree_linear_change}). This is a new result. In this section we also
explain in detail that an example of a polynomial automorphism (Example \ref
{exmple_red_type_I}) due to Shestakov and Umirbaev does not admit an
elementary reduction, and recall a theorem from number theory (Theorem \ref
{tw_sywester}) that will be useful in some parts of the paper.

In section \ref{section_useful_results} we collect some general results
about multidegrees. Some of them were already published by the author:
Proposition \ref{prop_m<n}, Proposition \ref{prop_sum_d_i} and Corollary \ref
{cor_dim2_tame_d1_d2} \cite{Karas}. The other results in that section
(except Theorem \ref{tw_Kuroda} due to Kuroda) are new. The most important
results of that section are Proposition \ref{prop_sum_d_i}, Theorem \ref
{tw_reduc_type_4} and Lemma \ref{lem_deg_poiss_2}.

In section \ref{section_ppd_and_general} we discuss tame automorphisms of $%
\Bbb{C}^{3}$ with multidegree of the form $(p_{1},p_{2},d_{3}),$ $%
2<p_{1}<p_{2}\leq d_{3},$ where $p_{1}$ and $p_{2}$ are prime numbers, and
more generaly, coprime odd numbers. In both cases we give a necessary and
sufficient numerical condition on $(p_{1},p_{2},d_{3})$ to be the
multidegree of tame automorphism of $\Bbb{C}^{3}.$ The results of that
section were already published by the author \cite{Karas2}, and by the
author and J. Zygad\l o \cite{Karas Zygad}.

Section \ref{section_3dd} presents results due to the author \cite{Karas3}.
They concern tame automorphisms with multidegeree $(3,d_{2},d_{3}),$ $3\leq
d_{2}\leq d_{3}.$

The results of sections \ref{section_4dd} and \ref{section_pdd_5dd} are new
and concern tame automorphisms with multidegree $(4,d_{2},d_{3}),$ $4\leq
d_{2}\leq d_{3}$ (section \ref{section_4dd}), and $(p,d_{2},d_{3}),$ $5\leq
p\leq d_{2}\leq d_{3},$ where $p$ is a prime (section \ref{section_pdd_5dd}%
). It is of interest that in showing that there is no tame automorphism of $%
\Bbb{C}^{3}$ with multidegree $(5,6,9),$ we use the Jacobian Conjecture
(actually the Moh theorem). On the other hand, it is very surprising that
the existence of a tame automorphism of $\Bbb{C}^{3}$ with multidegree $%
(37,70,105)$ implies that the two-dimensional Jacobian Conjecture is false
(this is proved in section \ref{section_pdd_5dd}).

In section \ref{section_ab_skonczone} we present a result due to J. Zygad\l
o \cite{Zygadlo}, and in the last section we give new results on the
multidegree of the inverse of a polynomial automorphism of $\Bbb{C}^{2}.$

\section{Notations, basic definitions and two-dimensional case\label%
{section_notations_and}}

\subsection{Notation}

We assume that $0\in \Bbb{N},$ and we denote by $\Bbb{N}^{*},\Bbb{Z}^{*},%
\Bbb{C}^{*},$ respectively, $\Bbb{N}\backslash \left\{ 0\right\} ,\Bbb{Z}%
\backslash \left\{ 0\right\} ,\Bbb{C}\backslash \left\{ 0\right\} .$ By $%
\Bbb{C}\left[ X_{1},\ldots ,X_{n}\right] $ we denote the polynomial ring in $%
n$ variables over $\Bbb{C}.$ In particular, $X_{1},\ldots ,X_{n}$ denote
variables, and $x_{1},\ldots ,x_{n}$ denote coordinates in $\Bbb{C}^{n}.$ We
will work over the complex field $\Bbb{C},$ but all results remain valid
over any algebraically closed field of characteristic zero.

For any $f\in \Bbb{C}\left[ X_{1},\ldots ,X_{n}\right] ,$ $\deg f$ denotes
the usual total degree of $f.$ We say that $f$ is homogeneous if $f$ is a
sum of monomials of the same degree. We denote by $\overline{f}$ the leading
form of $f,$ i.e. the homogeneous part of $f$ of the maximal degree. Of
course, $\deg f=\deg \overline{f}.$

Moreover, $\gcd (d_{1},\ldots ,d_{n})$ and $\func{lcm}(d_{1},\ldots ,d_{n})$
denote the greatest common divisor of $d_{1},\ldots ,d_{n}$ and least commom
multiply of $d_{1},\ldots ,d_{n},$ respectively.

\subsection{Examples of polynomial automorphisms}

First of all, recall that a \textit{polynomial mapping} $F:\Bbb{C}%
^{n}\rightarrow \Bbb{C}^{n}$ is a mapping whose coordinate functions $F_{i},$
where $F=(F_{1},\ldots ,F_{n}),$ are polynomials. By a \textit{polynomial
automorphism} of $\Bbb{C}^{n}$ (later, just \textit{automorphism}) we mean a
polynomial mapping $F:\Bbb{C}^{n}\rightarrow \Bbb{C}^{n}$ such that there
exists a polynomial mapping $G:\Bbb{C}^{n}\rightarrow \Bbb{C}^{n}$ with $%
F\circ G=G\circ F=\limfunc{id}_{\Bbb{C}^{n}}.$ We then also say that $F$ is
invertible. The group of all polynomial automorphisms of $\Bbb{C}^{n}$ is
denoted by $\limfunc{Aut}\left( \Bbb{C}^{n}\right) .$

Polynomial automorphisms play a prominent role in affine algebraic geometry 
\cite{scotishbook, smale}. Typical problems are the Jacobian Problem \cite
{Applegate, BabaNakai, Jung, Kulk, Nagata, Nagata1, Nagata2, Nagata3,
Pinchuk}, existence of wild automorphisms \cite{EssMak-Lim, sh umb1, sh
umb2, sh umb3}, the inverse formula \cite{MacKayWng1, MacKayWng2,
MacKayWng3, McKayWngMoh} or stable tameness \cite{smith}.

There are some special kinds of polynomial automorphisms of $\Bbb{C}^{n}$:

\begin{itemize}
\item  Affine polynomial automorphisms, i.e. polynomial automorphisms $%
F=\left( F_{1},\ldots ,F_{n}\right) $ such that $\deg F_{i}=1$ for $%
i=1,\ldots ,n.$ The set of all such automorphisms will be denoted $\limfunc{%
Aff}\left( \Bbb{C}^{n}\right) ;$ it is a subgroup of $\limfunc{Aut}\left( 
\Bbb{C}^{n}\right) .$

\item  Linear automorphisms, i.e. affine automorphisms $F:\Bbb{C}%
^{n}\rightarrow \Bbb{C}^{n}$ such that $F\left( 0,\ldots ,0\right) =\left(
0,\ldots ,0\right) .$\newline
This is of course the same as the general linear group, denoted $%
GL_{n}\left( \Bbb{C}\right) .$

\item  Elementary automorphisms, i.e. maps of the form 
\begin{equation*}
F:\Bbb{C}^{n}\ni \left( x_{1},\ldots ,x_{n}\right) \mapsto \left(
x_{1},\ldots ,x_{i}+f(x_{1},\ldots ,x_{i-1},x_{i+1},\ldots ,x_{n})\ldots
,x_{n}\right) \in \Bbb{C}^{n}
\end{equation*}
for some $i\in \left\{ 1,\ldots ,n\right\} $ and $f\in \Bbb{C}\left[
X_{1},\ldots ,X_{i-1},X_{i+1},\ldots ,X_{n}\right] .$\newline
One can easily see that 
\begin{equation*}
F^{-1}\left( x_{1},\ldots ,x_{n}\right) =\left( x_{1},\ldots
,x_{i}-f(x_{1},\ldots ,x_{i-1},x_{i+1},\ldots ,x_{n}),\ldots ,x_{n}\right) .
\end{equation*}

\item  Triangular automorphisms, i.e. maps of the form 
\begin{equation}
F:\Bbb{C}^{n}\ni \left( x_{1},\ldots ,x_{n}\right) \mapsto \left(
x_{1},x_{2}+f_{1}(x_{1}),\ldots ,x_{n}+f_{n-1}(x_{1},\ldots ,x_{n-1})\right)
\in \Bbb{C}^{n},  \label{introd_def_tame}
\end{equation}
where $f_{1}\in \Bbb{C}\left[ X_{1}\right] ,f_{2}\in \Bbb{C}\left[
X_{1},X_{2}\right] ,\ldots ,f_{n-1}\in \Bbb{C}\left[ X_{1},\ldots
,X_{n-1}\right] .$\newline
One can check that $F$ is invertible and 
\begin{equation*}
F^{-1}\left( \left\{ 
\begin{array}{l}
x_{1} \\ 
x_{2} \\ 
x_{3} \\ 
\vdots
\end{array}
\right\} \right) =\left\{ 
\begin{array}{l}
x_{1} \\ 
x_{2}-f_{1}\left( x_{1}\right) \\ 
x_{3}-f_{2}\left( x_{1},x_{2}-f_{1}\left( x_{1}\right) \right) \\ 
\vdots
\end{array}
\right\} .
\end{equation*}
We will also say that $F$ is triangular if $F$ is of the form (\ref
{introd_def_tame}) after some permutation of variables.

\item  De Jonqui\`{e}res automorphisms, i.e. mappings of the form 
\begin{equation}
F:\Bbb{C}^{n}\ni \left\{ 
\begin{array}{l}
x_{1} \\ 
x_{2} \\ 
\vdots \\ 
x_{n}
\end{array}
\right\} \mapsto \left\{ 
\begin{array}{l}
a_{1}x_{1}+f_{1}\left( x_{2},\ldots ,x_{n}\right) \\ 
a_{2}x_{2}+f_{2}\left( x_{3},\ldots ,x_{n}\right) \\ 
\vdots \\ 
a_{n}x_{n}+f_{n}
\end{array}
\right\} \in \Bbb{C}^{n},  \label{row_de_Jonquieres}
\end{equation}
where $a_{i}\in \Bbb{C}^{*},$ $f_{i}\in \Bbb{C}\left[ X_{i+1},\ldots
,X_{n}\right] $ for all $1\leq i\leq n-1$ and $f_{n}\in \Bbb{C}.$ We then
write $F\in J\left( \Bbb{C}^{n}\right) .$\newline
As for triangular mappings, one can check that if $F\in J\left( \Bbb{C}%
^{n}\right) ,$ then $F$ is invertible. Also, one can verify that $J\left( 
\Bbb{C}^{n}\right) $ is a subgroup of $\limfunc{Aut}\left( \Bbb{C}%
^{n}\right) .$

\item  Tame automorphisms, i.e. compositions of a finite number of affine
and triangular automorphisms. Sometimes a tame automorphism is defined as a
composition of a finite number of affine and elementary automorphisms, or as
a composition of a finite number of affine and de Jonqui\`{e}res
automorphisms. One can check that all these definitions are equivalent.
\end{itemize}

To end this section, recall that for any polynomial mapping $F:\Bbb{C}%
^{n}\rightarrow \Bbb{C}^{n}$ we have the $\Bbb{C}$-homomorphism $F^{*}:\Bbb{C%
}\left[ X_{1},\ldots ,X_{n}\right] \rightarrow \Bbb{C}\left[ X_{1},\ldots
,X_{n}\right] $ defined by 
\begin{equation*}
F^{*}:\Bbb{C}\left[ X_{1},\ldots ,X_{n}\right] \ni h\mapsto h\circ F\in \Bbb{%
C}\left[ X_{1},\ldots ,X_{n}\right] ,
\end{equation*}
and for any $\Bbb{C}$-homomorphism $\Phi :\Bbb{C}\left[ X_{1},\ldots
,X_{n}\right] \rightarrow \Bbb{C}\left[ X_{1},\ldots ,X_{n}\right] $ we have
the polynomial mapping $\Phi _{*}:\Bbb{C}^{n}\rightarrow \Bbb{C}^{n}$
defined as 
\begin{equation*}
\Phi _{*}:\Bbb{C}^{n}\ni \left( x_{1},\ldots ,x_{n}\right) \mapsto \left(
F_{1}\left( x_{1},\ldots ,x_{n}\right) ,\ldots ,F_{n}\left( x_{1},\ldots
,x_{n}\right) \right) \in \Bbb{C}^{n},
\end{equation*}
where $F_{i}=\Phi \left( X_{i}\right) .$ Moreover, recall that $\left(
F^{*}\right) _{*}=F,\left( \Phi _{*}\right) ^{*}=\Phi ,$ and $F\,$is an
automorphism if and only if $F^{*}$ is a $\Bbb{C}$-automorphism of $\Bbb{C}%
\left[ X_{1},\ldots ,X_{n}\right] .$ Thus one can translate the notions of
affine, linear, elementary, triangular and tame automorphisms of $\Bbb{C}%
^{n} $ into the language of $\Bbb{C}$-automorphisms of $\Bbb{C}\left[
X_{1},\ldots ,X_{n}\right] .$

\subsection{Degree, bidegree and multidegree}

Let $F=\left( F_{1},\ldots ,F_{n}\right) :\Bbb{C}^{n}\rightarrow \Bbb{C}^{n}$
be any polynomial map. By the \textit{degree} of $F,$ denoted $\deg F,$ we
mean the number 
\begin{equation*}
\deg F=\max \left\{ \deg F_{1},\ldots ,\deg F_{n}\right\} ,
\end{equation*}
and by the \textit{multidegree} of $F,$ denoted $\limfunc{mdeg}F,$ we mean
the sequence of natural numbers 
\begin{equation*}
\limfunc{mdeg}F=\left( \deg F_{1},\ldots ,\deg F_{n}\right) .
\end{equation*}
For $n=2$ the multidegree is called bidegree, and denoted $\limfunc{bideg}.$
(see e.g. \cite{van den Essen}).

For a fixed $n\in \Bbb{N},$ we will also consider the mappings 
\begin{equation*}
\deg :\limfunc{End}\left( \Bbb{C}^{n}\right) \ni F\mapsto \deg F\in \Bbb{N}
\end{equation*}
and 
\begin{equation*}
\limfunc{mdeg}:\limfunc{End}\left( \Bbb{C}^{n}\right) \ni F\mapsto \limfunc{%
mdeg}F\in \Bbb{N}^{n},
\end{equation*}
where $\limfunc{End}\left( \Bbb{C}^{n}\right) $ denotes the set of all
polynomial mappings $\Bbb{C}^{n}\rightarrow \Bbb{C}^{n}.$

One of the main goals of this paper is to obtain a description of the sets 
\begin{equation*}
\limfunc{mdeg}\left( \limfunc{Aut}\left( \Bbb{C}^{n}\right) \right) ,%
\limfunc{mdeg}\left( \limfunc{Tame}\left( \Bbb{C}^{n}\right) \right) \subset 
\Bbb{N}^{n}.
\end{equation*}
If $n=1$ the answer is 
\begin{equation*}
\limfunc{mdeg}\left( \limfunc{Aut}\left( \Bbb{C}^{1}\right) \right) =%
\limfunc{mdeg}\left( \limfunc{Tame}\left( \Bbb{C}^{1}\right) \right)
=\left\{ 1\right\} .
\end{equation*}
The description for $n=2,$ based on a theorem of Jung and van der Kulk, will
be given in the next subsection. The answer for $n\geq 3$ is much more
complicated, and will be investigated in the rest of the paper. The very
first result in this direction says that $\left( 3,4,5\right) \notin 
\limfunc{mdeg}\left( \limfunc{Tame}\left( \Bbb{C}^{3}\right) \right) $ \cite
{Karas}. The next results obtained by the author \cite{Karas2, Karas3, Karas
Zygad} are also included.

Since for any $\left( F_{1},\ldots ,F_{n}\right) \in \limfunc{Aut}\left( 
\Bbb{C}^{n}\right) $ we have $\deg F_{i}\ge 1,$ $i=1,\ldots ,n,$ and since
for any permutation $\sigma $ of $\left\{ 1,\ldots ,n\right\} $ and any
sequence $\left( d_{1},\ldots ,d_{n}\right) \in \Bbb{N}^{n}$ we have 
\begin{equation*}
\left( d_{1},\ldots ,d_{n}\right) \in \limfunc{mdeg}\left( \limfunc{Tame}%
\left( \Bbb{C}^{n}\right) \right) \Longleftrightarrow \left( d_{\sigma
(1)},\ldots ,d_{\sigma (n)}\right) \in \limfunc{mdeg}\left( \limfunc{Tame}%
\left( \Bbb{C}^{n}\right) \right)
\end{equation*}
and 
\begin{equation*}
\left( d_{1},\ldots ,d_{n}\right) \in \limfunc{mdeg}\left( \limfunc{Aut}%
\left( \Bbb{C}^{n}\right) \right) \Longleftrightarrow \left( d_{\sigma
(1)},\ldots ,d_{\sigma (n)}\right) \in \limfunc{mdeg}\left( \limfunc{Aut}%
\left( \Bbb{C}^{n}\right) \right) ,
\end{equation*}
in our considerations we can always assume that $1\leq d_{1}\leq \ldots \leq
d_{n}.$ In other words, we will consider the sets 
\begin{equation*}
\limfunc{mdeg}\left( \limfunc{Tame}\left( \Bbb{C}^{n}\right) \right) \cap
\left\{ \left( d_{1},\ldots ,d_{n}\right) :1\leq d_{1}\leq \ldots \leq
d_{n}\right\} \subset \Bbb{N}^{n}
\end{equation*}
and 
\begin{equation*}
\limfunc{mdeg}\left( \limfunc{Aut}\left( \Bbb{C}^{n}\right) \right) \cap
\left\{ \left( d_{1},\ldots ,d_{n}\right) :1\leq d_{1}\leq \ldots \leq
d_{n}\right\} \subset \Bbb{N}^{n}.
\end{equation*}

\subsection{Jung and van der Kulk result\label{subsection)jung_kulk}}

Before giving a description of the set $\limfunc{mdeg}\left( \limfunc{Tame}%
\left( \Bbb{C}^{2}\right) \right) ,$ we recall the following two classical
results.

\begin{proposition}
\textit{(\cite{van den Essen}, Corollary 5.1.3) }$\limfunc{Tame}\left( \Bbb{C%
}^{2}\right) $ is the amalgamated product of $\limfunc{Aff}\left( \Bbb{C}%
^{2}\right) $ and $J\left( \Bbb{C}^{2}\right) $ over their intersection,
i.e. $\limfunc{Tame}\left( \Bbb{C}^{2}\right) $ is generated by these two
groups and if $\tau _{i}\in J\left( \Bbb{C}^{2}\right) \backslash \limfunc{%
Aff}\left( \Bbb{C}^{2}\right) $ and $\lambda _{i}\in \limfunc{Aff}\left( 
\Bbb{C}^{2}\right) \backslash J\left( \Bbb{C}^{2}\right) ,$ then $\tau
_{1}\circ \lambda _{1}\circ \cdots \circ \tau _{n}\circ \lambda _{n}\circ
\tau _{n+1}$ does not belong to $\limfunc{Aff}\left( \Bbb{C}^{2}\right) .$
\end{proposition}

Let us here recall the definition of an amalgamated product, following \cite
{Rusek}.

\begin{definition}
Let $G$ be a group and let $A,B$ be two subgroups with $C=A\cap B.$ We
denote by $\Phi $ (resp. $\Psi $) a complete set of representatives of the
left coset space $A/C$ (resp. $B/C$) subject only to the restriction that
the representative of $C$ itself is the neutral element of $G.$ We say that $%
G$ is an amalgamated product of $A$ and $B\,$over $C$ if every element $g\in
G$ can be written uniquely as $g=\varphi _{0}\psi _{1}\varphi _{1}\psi
_{2}\cdots \varphi _{n-1}\psi _{n}\varphi _{n}\gamma $ for suitable $n\in 
\Bbb{N},$ $\varphi _{0},\ldots ,\varphi _{n}\in \Phi ,$ $\psi _{1},\ldots
,\psi _{n}\in \Psi ,$ $\gamma \in C,$ where only $\varphi _{0},\varphi _{n}$
and $\gamma $ may be the neutral element.
\end{definition}

The second result is the following

\begin{corollary}
\label{cor_dim2_tame_d1_d2}\textit{(\cite{van den Essen}, Corollary 5.1.6) }%
Let $F=\left( F_{1},F_{2}\right) \in \limfunc{Tame}\left( \Bbb{C}^{2}\right) 
$ with $\limfunc{bideg}F=\left( d_{1},d_{2}\right) .$ Let $h_{i}$ denote the
homogeneous component of $F_{i}$ of degree $d_{i}.$ Then:\newline
a) $d_{1}|d_{2}$ or $d_{2}|d_{1}.$\newline
b) If $\deg F>1,$ then we have:\newline
\qquad i) if $d_{1}<d_{2},$ then $h_{2}=ch_{1}^{\frac{d_{2}}{d_{1}}}$ for
some $c\in \Bbb{C},$\newline
\qquad ii) if $d_{2}<d_{1},$ then $h_{1}=ch_{2}^{\frac{d_{1}}{d_{2}}}$ for
some $c\in \Bbb{C},$\newline
\qquad iii) if $d_{1}=d_{2},$ then there exists $\lambda \in \limfunc{Aff}(%
\Bbb{C}^{2})$ such that $\deg \widetilde{F}_{1}>\deg \widetilde{F}_{2},$
where $\widetilde{F}=\left( \widetilde{F}_{1},\widetilde{F}_{2}\right)
=\lambda \circ F.$
\end{corollary}

From the above corollary we obtain 
\begin{equation*}
\limfunc{mdeg}\left( \limfunc{Tame}\left( \Bbb{C}^{2}\right) \right) \cap
\left\{ \left( d_{1},d_{2}\right) :1\leq d_{1}\leq d_{2}\right\} \subset
\left\{ \left( d_{1},d_{2}\right) \in \left( \Bbb{N}^{*}\right)
^{2}:d_{1}|d_{2}\right\} .
\end{equation*}
Since, for $d_{1}|d_{2},$ and 
\begin{eqnarray*}
F_{1} &:&\Bbb{C}^{2}\ni \left( x,y\right) \mapsto \left(
x+y^{d_{1}},y\right) \in \Bbb{C}^{2}, \\
F_{2} &:&\Bbb{C}^{2}\ni \left( u,v\right) \mapsto \left( u,v+u^{\frac{d_{2}}{%
d_{1}}}\right) \in \Bbb{C}^{2},
\end{eqnarray*}
$F_{2}\circ F_{1}$ is a tame automorphism of $\Bbb{C}^{2}$ with $\limfunc{%
mdeg}\left( F_{2}\circ F_{1}\right) =\left( d_{1},d_{2}\right) ,$ we see
that 
\begin{equation*}
\limfunc{mdeg}\left( \limfunc{Tame}\left( \Bbb{C}^{2}\right) \right) \cap
\left\{ \left( d_{1},d_{2}\right) :1\leq d_{1}\leq d_{2}\right\} =\left\{
\left( d_{1},d_{2}\right) \in \left( \Bbb{N}^{*}\right)
^{2}:d_{1}|d_{2}\right\} .
\end{equation*}

To obtain a description of the set $\limfunc{mdeg}\left( \limfunc{Aut}\left( 
\Bbb{C}^{2}\right) \right) ,$ we also need the following result due to Jung 
\cite{Jung} and van der Kulk \cite{Kulk}.

\begin{theorem}
\label{tw_jung_van_der_Kulk}\textit{(Jung-van der Kulk, see e.g. \cite{van
den Essen}, Theorem 5.1.11) }We have $\limfunc{Aut}\left( \Bbb{C}^{2}\right)
=\limfunc{Tame}\left( \Bbb{C}^{2}\right) .$ More precisely, $\limfunc{Aut}%
\left( \Bbb{C}^{2}\right) $ is the amalgamated product of $\limfunc{Aff}%
\left( \Bbb{C}^{2}\right) $ and $J\left( \Bbb{C}^{2}\right) $ over their
intersection.
\end{theorem}

Using Theorem \ref{tw_jung_van_der_Kulk}, we of course obtain 
\begin{equation*}
\limfunc{mdeg}\left( \limfunc{Aut}\left( \Bbb{C}^{2}\right) \right) =%
\limfunc{mdeg}\left( \limfunc{Tame}\left( \Bbb{C}^{2}\right) \right) ,
\end{equation*}
and so 
\begin{equation*}
\limfunc{mdeg}\left( \limfunc{Aut}\left( \Bbb{C}^{2}\right) \right) \cap
\left\{ \left( d_{1},d_{2}\right) :1\leq d_{1}\leq d_{2}\right\} =\left\{
\left( d_{1},d_{2}\right) \in \left( \Bbb{N}^{*}\right)
^{2}:d_{1}|d_{2}\right\} .
\end{equation*}

A crucial result, used in the proof of the Jung-van der Kulk result, is the
following lemma and the notion of elementary reduction.

\begin{lemma}
\label{lem_f_g_alg_zalez}\textit{(see e.g. \cite{van den Essen}, Lemma
10.2.4) }Let $f,g\in \Bbb{C}\left[ X,Y\right] $ be homogeneous polynomials
such that $Jac\left( f,g\right) =0.$ Then there exists a homogeneous
polynomial $h$ such that:\newline
i) $f=c_{1}h^{n_{1}}$ and $g=c_{2}h^{n_{2}}$ for some integers $%
n_{1},n_{2}\geq 0$ and $c_{1},c_{2}\in \Bbb{C}^{*}.$\newline
ii) $h$ is not of the form $ch_{0}^{s}$ for any $c\in k^{*},$ any $h_{0}\in
k\left[ x,y\right] $ and any integer $s>1.$
\end{lemma}

Recall that an automorphism $F=\left( F_{1},\ldots ,F_{n}\right) $ admits an 
\textit{elementary reduction} if there exists an elementary automorphism $%
\tau :\Bbb{C}^{n}\rightarrow \Bbb{C}^{n}$ such that for $G=\left(
G_{1},\ldots ,G_{n}\right) =\tau \circ F$ we have 
\begin{equation*}
\limfunc{mdeg}G<\limfunc{mdeg}F,
\end{equation*}
i.e. 
\begin{equation*}
\deg G_{i}\leq \deg F_{i}\qquad \text{for all }i=1,\ldots ,n
\end{equation*}
and 
\begin{equation*}
\deg G_{i}<\deg F_{i}\qquad \text{for some }i\in \left\{ 1,\ldots ,n\right\}
.
\end{equation*}
We then say that $G$ is an elementary reduction of $F.$ One can easily
notice that $F$ admits an elementary reduction if there exists $i\in \left\{
1,\ldots ,n\right\} $ and a polynomial $g\in \Bbb{C}\left[ Y_{1},\ldots
,Y_{n-1}\right] $ such that 
\begin{equation*}
\deg \left( F_{i}-g\left( F_{1},\ldots ,F_{i-1},F_{i+1},\ldots ,F_{n}\right)
\right) <\deg F_{i}.
\end{equation*}

We will also need the following generalization of the above lemma.

\begin{proposition}
\label{prop_fg_alg_zalez}Let $f,g\in \Bbb{C}\left[ X_{1},\ldots
,X_{n}\right] $ be homogeneous, algebraically dependent polynomials. Then
there exists a homogeneous polynomial $h\in \Bbb{C}\left[ X_{1},\ldots
,X_{n}\right] $ such that:\newline
i) $f=c_{1}h^{n_{1}}$ and $g=c_{2}h^{n_{2}}$ for some integers $%
n_{1},n_{2}\geq 0$ and $c_{1},c_{2}\in \Bbb{C}^{*}.$\newline
ii) $h$ is not of the form $ch_{0}^{s}$ for any $c\in \Bbb{C}^{*},$ any $%
h_{0}\in \Bbb{C}\left[ X_{1},\ldots ,X_{n}\right] $ and any integer $s>1.$
\end{proposition}

One can obtain the above result using Lemma 2 in \cite{Umirbaev Yu}.

\section{Main tools\label{section_main_tools}}

\subsection{Poisson bracket and degree of polynomials\label%
{subsection_Poisson_bracket}}

In this section we present the first main tool which we will use in our
considerations: the Poisson bracket of two polynomials and a theorem that
estimates from below the degree of a polynomial of the form $h\left(
f,g\right) ,$ where $f,g\in \Bbb{C}[X_{1},\ldots ,X_{n}]$ and $h\in \Bbb{C}%
[X,Y].$

We start with the definition of a *-reduced pair.

\begin{definition}
\textit{(\cite{sh umb1}, Definition 1) }\label{def_*-red}A pair $f,g\in \Bbb{%
C}[X_{1},\ldots ,X_{n}]$ is called *-reduced if\newline
(i) $f,g$ are algebraically independent;\newline
(ii) $\overline{f},\overline{g}$ are algebraically dependent;\newline
(iii) $\overline{f}\notin \Bbb{C}[\overline{g}]$ and $\overline{g}\notin 
\Bbb{C}[\overline{f}].$\newline
Moreover, we say that $f,g$ is a $p$-reduced pair$\ $if $f,g$ is a *-reduced
pair with $\deg f<\deg g$ and $p=\frac{\deg f}{\gcd (\deg f,\deg g)}.$
\end{definition}

One may ask whether $p$ can be equal to $1$ for a $p$-reduced pair $f,g.$
The answer is given by the following

\begin{proposition}
If $f,g$ is a $p$-reduced pair, then $p>1.$
\end{proposition}

\begin{proof}
If $f,g$ is $p$-reduced, then $\overline{f}$ and $\overline{g}$ are
algebraically dependent. This means, by Proposition \ref{prop_fg_alg_zalez},
that there is a homogeneous polynomial $h$ such that 
\begin{equation*}
\overline{f}=\alpha h^{l}\qquad \text{and\qquad }\overline{g}=\beta h^{m}
\end{equation*}
for some $\alpha ,\beta \in \Bbb{C}^{*}$ and $l,m\in \Bbb{N}.$ Assume that $%
p=\frac{\deg f}{\gcd \left( \deg f,\deg g\right) }=1.$ Then $l|m,$ and so $%
\overline{g}=\gamma \overline{f}^{r}$ for $r=\frac{m}{l}$ and $\gamma \in 
\Bbb{C}^{*}.$ This contradicts condition (iii) of Definition \ref{def_*-red}.
\end{proof}

For any $f,g\in \Bbb{C}[X_{1},\ldots ,X_{n}]$ we denote by $\left[
f,g\right] $ the Poisson bracket of $f$ and $g$, i.e. the following formal
sum: 
\begin{equation*}
\sum_{1\leq i<j\leq n}\left( \frac{\partial f}{\partial X_{i}}\frac{\partial
g}{\partial X_{j}}-\frac{\partial f}{\partial X_{j}}\frac{\partial g}{%
\partial X_{i}}\right) \left[ X_{i},X_{j}\right] ,
\end{equation*}
where $[X_{i},X_{j}]$ are formal objects satisfying the condition 
\begin{equation*}
\lbrack X_{i},X_{j}]=-[X_{j},X_{i}]\qquad \text{for all }i,j.
\end{equation*}
We also define 
\begin{equation*}
\deg \left[ X_{i},X_{j}\right] =2\qquad \text{for all }i\neq j,
\end{equation*}
$\deg 0=-\infty $ and 
\begin{equation*}
\deg \left[ f,g\right] =\max_{1\leq i<j\leq n}\deg \left\{ \left( \frac{%
\partial f}{\partial X_{i}}\frac{\partial g}{\partial X_{j}}-\frac{\partial f%
}{\partial X_{j}}\frac{\partial g}{\partial X_{i}}\right) \left[
X_{i},X_{j}\right] \right\} .
\end{equation*}
Since $2-\infty =-\infty ,$ we have 
\begin{equation*}
\deg [f,g]=2+\underset{1\leq i<j\leq n}{\max }\deg \left( \frac{\partial f}{%
\partial X_{i}}\frac{\partial g}{\partial X_{j}}-\frac{\partial f}{\partial
X_{j}}\frac{\partial g}{\partial X_{i}}\right) ,
\end{equation*}
and hence 
\begin{equation}
\deg \left[ f,g\right] \leq \deg f+\deg g.
\label{deg_Poisson_bracket_deg_f_g}
\end{equation}

Another inequality involving the degree of a Poisson bracket will be a
consequence of Proposition \ref{cor_rank_trdeg} below, in which $\frac{%
\partial \left( F_{1},\ldots ,F_{r}\right) }{\partial \left( X_{1},\ldots
,X_{n}\right) }$ means the Jacobian matrix (not necessarily quadratic) of
the mapping $\left( F_{1},\ldots ,F_{r}\right) :\Bbb{C}^{n}\rightarrow \Bbb{C%
}^{r}.$

\begin{proposition}
\label{cor_rank_trdeg}If $F_{1},\ldots ,F_{r}\in \Bbb{C}\left[ X_{1},\ldots
,X_{n}\right] ,$ then 
\begin{equation*}
\limfunc{rank}\frac{\partial \left( F_{1},\ldots ,F_{r}\right) }{\partial
\left( X_{1},\ldots ,X_{n}\right) }=\limfunc{trdeg}\nolimits_{\Bbb{C}}\Bbb{C}%
\left( F_{1},\ldots ,F_{r}\right) .
\end{equation*}
\end{proposition}

One can deduce the above result from \cite[Chap. X, Prop. 10]{S.Lang}. The
version for $r=n$ can also be found in \cite[Prop. 1.2.9]{van den Essen}.

By Proposition \ref{cor_rank_trdeg} and the definition of the degree of a
Poisson bracket we obtain the following remark.

\begin{remark}
$f,g\in \Bbb{C}[X_{1},\ldots ,X_{n}]$ are algebraically independent if and
only if $\deg \left[ f,g\right] \geq 2.$
\end{remark}

We also have the following

\begin{remark}
For any $f,g\in \Bbb{C}[X_{1},\ldots ,X_{n}]$ the following conditions are
equivalent:\newline
(1) $\deg \left[ f,g\right] =\deg f+\deg g,$\newline
(2) $\overline{f},\overline{g}$ are algebraically independent.
\end{remark}

\begin{proof}
Let 
\begin{equation*}
f=f_{0}+\cdots +f_{d},\qquad g=g_{0}+\cdots +g_{m}
\end{equation*}
be the homogeneous decompositions of $f$ and $g.$ Since 
\begin{equation*}
\left[ f,g\right] =\sum_{i,j}\left[ f_{i},g_{j}\right] =\left[
f_{d},g_{m}\right] +\sum_{i<d\text{ or }j<m}\left[ f_{i},g_{j}\right]
\end{equation*}
and 
\begin{equation*}
\deg \left[ f_{i},g_{j}\right] \leq \deg f_{i}+\deg g_{j}=i+j<d+m,
\end{equation*}
for $i<d$ or $j<m,$ it follows that 
\begin{equation*}
\deg \left[ f,g\right] =d+m\Longleftrightarrow \deg \left[
f_{d},g_{m}\right] =d+m.
\end{equation*}
But, since $f_{d}$ and $g_{m}$ are homogeneous polynomials of degrees $d$
and $m,$ respectively, by the definition of Poisson bracket we have 
\begin{equation*}
\deg \left[ f_{d},g_{m}\right] =d+m\Longleftrightarrow \left[
f_{d},g_{m}\right] \neq 0.
\end{equation*}
The last condition, by Proposition \ref{cor_rank_trdeg}, is equivalent to $%
f_{d},g_{m}$ being algebraically independent.
\end{proof}

Recall the following theorem due to Shestakov and Umirbaev.

\begin{theorem}
\textit{(\cite{sh umb1}, Theorem 2)}\label{tw_deg_g_fg} Let $f,g\in \Bbb{C}%
[X_{1},\ldots ,X_{n}]$ be a $p$-reduced pair, and let $G(X,Y)\in k[X,Y]$
with $\deg _{Y}G(X,Y)=pq+r,0\leq r<p.$ Then 
\begin{equation*}
\deg G(f,g)\geq q\left( p\deg g-\deg g-\deg f+\deg [f,g]\right) +r\deg g.
\end{equation*}
\end{theorem}

Notice that the estimate from Theorem \ref{tw_deg_g_fg} is true even if the
condition (ii) of Definition \ref{def_*-red} is not satisfied. Indeed, if $%
G=\sum_{i,j}a_{i,j}X^{i}Y^{j},$ we then have, by the algebraic independence
of $\overline{f}$ and $\overline{g},$ 
\begin{eqnarray*}
\deg G(f,g) &=&\underset{i,j}{\max }\deg (a_{i,j}f^{i}g^{j})\geq \deg
_{Y}G(X,Y)\cdot \deg g \\
&=&(qp+r)\deg g\geq q(p\deg g-\deg f-\deg g+\deg [f,g])+r\deg g.
\end{eqnarray*}
The last inequality is a consequence of the fact that $\deg [f,g]\leq \deg
f+\deg g.$

Notice that the above calculations are also valid for $p=1$ (when the pair $%
f,g$ does not satisfy the condition (ii) of Definition \ref{def_*-red}, $p$
may be equal to one).

Thus we have the following proposition.

\begin{proposition}
\label{prop_deg_g_fg}Let $f,g\in \Bbb{C}[X_{1},\ldots ,X_{n}]$ satisfy
conditions (i) and (iii) of Definition \ref{def_*-red}. Assume that $\deg
f<\deg g,$ put 
\begin{equation*}
p=\frac{\deg f}{\gcd \left( \deg f,\deg g\right) },
\end{equation*}
and let $G(X,Y)\in \Bbb{C}[X,Y]$ with $\deg _{Y}G(X,Y)=pq+r,0\leq r<p.$ Then 
\begin{equation*}
\deg G(f,g)\geq q\left( p\deg g-\deg g-\deg f+\deg [f,g]\right) +r\deg g.
\end{equation*}
\end{proposition}

\subsection{Degree of a Poisson bracket and a linear change of coordinates}

This section is devoted to showing the following lemma saying that the
degree of a Poisson bracket is invariant under a linear change of
coordinates.

\begin{lemma}
\label{lem_degree_linear_change}If $f,g\in \Bbb{C}[X_{1},\ldots ,X_{n}]$ and 
$L\in GL_{n}(\Bbb{C}),\,$then 
\begin{equation*}
\deg [L^{*}(f),L^{*}(g)]=\deg [f,g],
\end{equation*}
where $L^{*}(h)=h\circ L$ for any $h\in \Bbb{C}[X_{1},\ldots ,X_{n}].$
\end{lemma}

We first show

\begin{proposition}
\label{prop_degree_linear_change}If $f,g\in \Bbb{C}[X_{1},\ldots ,X_{n}]$
and $L:\Bbb{C}^{n}\rightarrow \Bbb{C}^{n}$ is any linear map, $\,$then 
\begin{equation*}
\deg [L^{*}(f),L^{*}(g)]\leq \deg [f,g].
\end{equation*}
\end{proposition}

\begin{proof}
It is easy to see that for every $h\in \Bbb{C}[X_{1},\ldots ,X_{n}]$ we have
(here we allow $L^{*}(h_{d})=0$ even if $h_{d}\neq 0$) 
\begin{equation*}
\left[ L^{*}(h)\right] _{d}=L^{*}(h_{d}),
\end{equation*}
where the subscript $d$ denotes the homogeneous part of degree $d.$ We also
have 
\begin{equation*}
\left[ Jac^{ij}(f,g)\right] _{d}=\sum_{k+l=d+2}Jac^{ij}(f_{k},g_{l}),
\end{equation*}
where 
\begin{equation*}
Jac^{ij}(f,g)=Jac^{X_{i}X_{j}}(f,g)=\det \left[ 
\begin{array}{ll}
\frac{\partial f}{\partial X_{i}} & \frac{\partial f}{\partial X_{j}} \\ 
\frac{\partial g}{\partial X_{i}} & \frac{\partial g}{\partial X_{j}}
\end{array}
\right] .
\end{equation*}
By the above equalities we have 
\begin{eqnarray}
\left[ Jac^{ij}\left( L^{*}(f),L^{*}(g)\right) \right] _{d}
&=&\sum_{k+l=d+2}Jac^{ij}(L^{*}(f)_{k},L^{*}(g)_{l})
\label{row_deg_bracket_lienar_change_1} \\
&=&\sum_{k+l=d+2}Jac^{ij}(L^{*}(f_{k}),L^{*}(g_{l})).  \notag
\end{eqnarray}
Since for any $h\in \Bbb{C}[X_{1},\ldots ,X_{n}]$ and $r\in \{1,\ldots ,n\}$
we have 
\begin{equation*}
\frac{\partial L^{*}(h)}{\partial X_{r}}=\frac{\partial (h\circ L)}{\partial
X_{r}}=\sum_{s=1}^{n}\frac{\partial h}{\partial X_{s}}\left( L\right) \cdot
a_{sr},
\end{equation*}
where $\left( a_{ij}\right) $ is the matrix of the mapping $L,$ it follows
that 
\begin{gather}
Jac^{ij}(L^{*}(f_{k}),L^{*}(g_{l}))=\det \left[ 
\begin{array}{ll}
\sum_{r=1}^{n}\frac{\partial f_{k}}{\partial X_{r}}\left( L\right) \cdot
a_{ri} & \sum_{r=1}^{n}\frac{\partial f_{k}}{\partial X_{r}}\left( L\right)
\cdot a_{rj} \\ 
\sum_{s=1}^{n}\frac{\partial g_{l}}{\partial X_{s}}\left( L\right) \cdot
a_{si} & \sum_{s=1}^{n}\frac{\partial g_{l}}{\partial Xs}\left( L\right)
\cdot a_{sj}
\end{array}
\right]  \label{row_deg_bracket_lienar_change_2} \\
=\sum_{r,s=1}^{n}\frac{\partial f_{k}}{\partial X_{r}}\left( L\right) \cdot
a_{ri}\cdot \frac{\partial g_{l}}{\partial X_{s}}\left( L\right) \cdot
a_{sj}-\sum_{r,s=1}^{n}\frac{\partial f_{k}}{\partial X_{r}}\left( L\right)
\cdot a_{rj}\cdot \frac{\partial g_{l}}{\partial X_{s}}\left( L\right) \cdot
a_{si}  \notag \\
=\sum_{r,s=1}^{n}\left[ \frac{\partial f_{k}}{\partial X_{r}}\left( L\right)
\cdot a_{ri}\cdot \frac{\partial g_{l}}{\partial X_{s}}\left( L\right) \cdot
a_{sj}-\frac{\partial f_{k}}{\partial X_{s}}\left( L\right) \cdot
a_{sj}\cdot \frac{\partial g_{l}}{\partial X_{r}}\left( L\right) \cdot
a_{ri}\right]  \notag \\
=\sum_{r,s=1}^{n}Jac^{rs}(f_{k},g_{l})\left( L\right) \cdot a_{ri}a_{sj} 
\notag \\
=\sum_{1\leq r<s\leq n}Jac^{rs}(f_{k},g_{l})\left( L\right) \cdot
a_{ri}a_{sj}+\sum_{1\leq s<r\leq n}Jac^{rs}(f_{k},g_{l})\left( L\right)
\cdot a_{ri}a_{sj}  \notag \\
=\sum_{1\leq r<s\leq n}Jac^{rs}(f_{k},g_{l})\left( L\right) \cdot
a_{ri}a_{sj}-\sum_{1\leq r<s\leq n}Jac^{rs}(f_{k},g_{l})\left( L\right)
\cdot a_{si}a_{rj}  \notag \\
=\sum_{1\leq r<s\leq n}Jac^{rs}(f_{k},g_{l})\left( L\right) \det \left[ 
\begin{array}{ll}
a_{ri} & a_{rj} \\ 
a_{si} & a_{sj}
\end{array}
\right] .  \notag
\end{gather}
Now, by (\ref{row_deg_bracket_lienar_change_1}) and (\ref
{row_deg_bracket_lienar_change_2}), we have 
\begin{gather}
\left[ Jac^{ij}\left( L^{*}(f),L^{*}(g)\right) \right] _{d}
\label{row_deg_bracket_lienar_change_3} \\
=\sum_{k+l=d+2}\ \sum_{1\leq r<s\leq n}Jac^{rs}(f_{k},g_{l})\left( L\right)
\det \left[ 
\begin{array}{ll}
a_{ri} & a_{rj} \\ 
a_{si} & a_{sj}
\end{array}
\right]  \notag \\
=\sum_{1\leq r<s\leq n}\left( \sum_{k+l=d+2}Jac^{rs}(f_{k},g_{l})\right)
\left( L\right) \det \left[ 
\begin{array}{ll}
a_{ri} & a_{rj} \\ 
a_{si} & a_{sj}
\end{array}
\right] .  \notag
\end{gather}

Take any $d>\deg \left[ f,g\right] .$ Then 
\begin{equation}
\sum_{k+l=d+2}Jac^{rs}(f_{k},g_{l})=0
\label{row_deg_bracket_lienar_change_4}
\end{equation}
for all pairs $r,s$ satisfying $1\leq r<s\leq n.$ Thus, by (\ref
{row_deg_bracket_lienar_change_3}) and (\ref{row_deg_bracket_lienar_change_4}%
), we obtain 
\begin{equation}
\left[ Jac^{ij}\left( L^{*}(f),L^{*}(g)\right) \right] _{d}=0
\label{row_deg_bracket_lienar_change_5}
\end{equation}
for all $i,j.$ The above equalities (for all $i,j$) mean that $\deg
[L^{*}(f),L^{*}(g)]<d.$ Since we can take $d=\deg [f,g]+1,\deg
[f,g]+2,\ldots $ we obtain 
\begin{equation}
\deg [L^{*}(f),L^{*}(g)]\leq \deg [f,g].
\label{row_deg_bracket_lienar_change_6}
\end{equation}
\end{proof}

Now, we can prove Lemma \ref{lem_degree_linear_change}.

\begin{proof}
By the above proposition we only need to show that $\deg
[L^{*}(f),L^{*}(g)]\geq \deg [f,g].$ But $f=\left( L^{-1}\right) ^{*}\left(
L^{*}\left( f\right) \right) $ and $g=\left( L^{-1}\right) ^{*}\left(
L^{*}\left( g\right) \right) .$ So applying Proposition \ref
{prop_degree_linear_change} to the polynomials $L^{*}\left( f\right)
,L^{*}\left( g\right) $ and the mapping $L^{-1}$ we obtain 
\begin{equation*}
\deg [f,g]=\deg [\left( L^{-1}\right) ^{*}\left( L^{*}\left( f\right)
\right) ,\left( L^{-1}\right) ^{*}\left( L^{*}\left( g\right) \right) ]\leq
\deg [L^{*}(f),L^{*}(g)].
\end{equation*}
\end{proof}

\subsection{Shestakov-Umirbaev reductions\label{subsection_reductions_I_IV}}

In this section we present the most remarkable result of Shestakov and
Umirbaev, Theorem \ref{tw_deg_g_fg}. The notions of reductions of types I-IV
are crucial in this theorem. Thus we start with the following definitions
(see \cite{sh umb1} or \cite{sh umb2}).

\begin{definition}
\label{def_type_I}Let $\Theta =\left( f_{1},f_{2},f_{3}\right) \,$be an
automorphism of $A=\Bbb{C}[X,Y,Z]$ such that (for some $n\in \Bbb{N}^{*}$) $%
\deg f_{1}=2n,\deg f_{2}=ns,$ where $s\geq 3$ is an odd number, $2n<\deg
f_{3}\leq ns$ and $\overline{f}_{3}\notin \Bbb{C}\left[ \overline{f}_{1},%
\overline{f}_{2}\right] .$ Suppose that there exists $\alpha \in \Bbb{C}^{*}$
such that the elements $g_{1}=f_{1},$ $g_{2}=f_{2}-\alpha f_{3}$ satisfy the
following conditions:\newline
(i) $g_{1},g_{2}$ is a 2-reduced pair and $\deg g_{1}=\deg f_{1},\deg
g_{2}=\deg f_{2};$\newline
(ii) the automorphism $\left( g_{1},g_{2},f_{3}\right) $ admits an
elementary reduction $\left( g_{1},g_{2},g_{3}\right) $ with $\deg \left[
g_{1},g_{3}\right] <\deg g_{2}+\deg \left[ g_{1},g_{2}\right] .$\newline
Then we will say that $\Theta $ admits a reduction $\left(
g_{1},g_{2},g_{3}\right) $ of type I.\newline
We will also say that a polynomial automorphism $F=\left(
F_{1},F_{2},F_{3}\right) $ admits a reduction of type I if for some
permutation $\sigma $ of $\left\{ 1,2,3\right\} ,$ the automorphism $\Theta
=\left( F_{\sigma (1)},F_{\sigma (2)},F_{\sigma (3)}\right) $ admits a
reduction of type I.
\end{definition}

Before proposing next definitions we present an example due to Shestakov and
Umirbaev of a tame automorphism of $\Bbb{C}^{3}$ which does not admit an
elementary reduction but admits a reduction of type I.

\begin{example}
\label{exmple_red_type_I}Let 
\begin{eqnarray*}
T_{1}(x_{1},x_{2},x_{3})
&=&(x_{1},x_{2}+x_{1}^{2},x_{3}+2x_{1}x_{2}+x_{1}^{3}), \\
T_{2}(y_{1},y_{2},y_{3})
&=&(6y_{1}+6y_{2}y_{3}+y_{3}^{3},4y_{2}+y_{3}^{2},y_{3}), \\
T_{3}(z_{1},z_{2},z_{3}) &=&(z_{1},z_{2},z_{3}+z_{1}^{2}-z_{2}^{3}), \\
L(u_{1},u_{2},u_{3}) &=&(u_{1}+u_{3},u_{2},u_{3})
\end{eqnarray*}
and 
\begin{eqnarray*}
G &=&T_{3}\circ T_{2}\circ T_{1}, \\
F &=&L\circ G.
\end{eqnarray*}
\end{example}

It is easy to see that 
\begin{equation*}
\limfunc{mdeg}\left( T_{2}\circ T_{1}\right) =\left( 9,6,3\right) ,
\end{equation*}
and because 
\begin{equation*}
\left( 6y_{1}+6y_{2}y_{3}+y_{3}^{3}\right) ^{2}-\left(
4y_{2}+y_{3}^{2}\right)
^{3}=36y_{1}^{2}+72y_{1}y_{2}y_{3}+12y_{1}y_{3}^{3}-12y_{2}^{2}y_{3}^{2}-%
\allowbreak 64y_{2}^{3}
\end{equation*}
and (provided that $y_{1}=x_{1},y_{2}=x_{2}+x_{1}^{2}$ and $%
y_{3}=x_{3}+2x_{1}x_{2}+x_{1}^{3}$) 
\begin{gather*}
12y_{1}y_{3}^{3}-12y_{2}^{2}y_{3}^{2} \\
=12x_{1}\left( x_{3}+2x_{1}x_{2}+x_{1}^{3}\right) ^{3}-12\left(
x_{2}+x_{1}^{2}\right) ^{2}\left( x_{3}+2x_{1}x_{2}+x_{1}^{3}\right) ^{2} \\
=12x_{3}x_{1}^{7}-12x_{1}^{6}x_{2}^{2}+\text{lower degree monomials,}
\end{gather*}
we have 
\begin{equation*}
\limfunc{mdeg}\left( T_{3}\circ T_{2}\circ T_{1}\right) =\left( 9,6,8\right)
\end{equation*}
and so 
\begin{equation*}
\limfunc{mdeg}F=\limfunc{mdeg}\left( L\circ G\right) =\left( 9,6,8\right) .
\end{equation*}

From the construction of $F$ it is clear that $F$ is a tame automorphism.
Moreover, it does not admit an elementary reduction. Indeed, if we put $%
F=\left( F_{1},F_{2},F_{3}\right) $ and assume that $\left( F_{1}-g\left(
F_{2},F_{3}\right) ,F_{2},F_{3}\right) ,$ for some $g\in \Bbb{C}[X,Y],$ is
an elementary reduction of $\left( F_{1},F_{2},F_{3}\right) $ then we must
have\label{row_exmpl_1} 
\begin{equation}
\deg g\left( F_{2},F_{3}\right) =9.  \label{row_exmpl_red_typ_I_1}
\end{equation}
But by Proposition \ref{prop_deg_g_fg}, we have 
\begin{equation}
\deg g\left( F_{2},F_{3}\right) \geq q\left( p\cdot 8-6-8+\deg
[F_{2},F_{3}]\right) +8r,  \label{row_exmpl_red_typ_I_2}
\end{equation}
where $\deg _{Y}g(X,Y)=qp+r,0\leq r<p,p=\frac{6}{\gcd \left( 6,8\right) }=3.$
Thus by (\ref{row_exmpl_red_typ_I_1}) and (\ref{row_exmpl_red_typ_I_2}) and
because $p\cdot 8-6-8+\deg [F_{2},F_{3}]=10+\deg [F_{2},F_{3}]\geq 12>9,$ we
must have $q=0$ and $r\leq 1.$ Thus $g$ must be of the form 
\begin{equation}
g(X,Y)=g_{0}(X)+g_{1}(X)Y.  \label{row_exmpl_red_typ_I_3}
\end{equation}
Since $8\Bbb{N\cap }\left( 6+8\Bbb{N}\right) =\emptyset ,$ from (\ref
{row_exmpl_red_typ_I_1}) and (\ref{row_exmpl_red_typ_I_3}) we obtain $9=\deg
g\left( F_{2},F_{3}\right) \in 8\Bbb{N\cup }\left( 6+8\Bbb{N}\right) ,$ a
contradiction.

Next, if we assume that $\left( F_{1},F_{2}-g\left( F_{3},F_{1}\right)
,F_{3}\right) ,$ for some $g\in \Bbb{C}[X,Y],$ is an elementary reduction of 
$\left( F_{1},F_{2},F_{3}\right) $ then we must have 
\begin{equation}
\deg g\left( F_{3},F_{1}\right) =6.  \label{row_exmpl_red_typ_I_4}
\end{equation}
But by Proposition \ref{prop_deg_g_fg}, 
\begin{equation}
\deg g\left( F_{3},F_{1}\right) \geq q\left( p\cdot 9-9-8+\deg
[F_{3},F_{1}]\right) +9r,  \label{row_exmpl_red_typ_I_5}
\end{equation}
where $\deg _{Y}g(X,Y)=qp+r,0\leq r<p,p=\frac{8}{\gcd \left( 8,9\right) }=8.$
Because $p\cdot 9-9-8+\deg [F_{3},F_{1}]=55+\deg [F_{3},F_{1}]\geq 57>8,$
from (\ref{row_exmpl_red_typ_I_4}) and (\ref{row_exmpl_red_typ_I_5}) we
obtain $q=r=0.$ This means that $g(X,Y)=g(X)$ and $\deg g\left(
F_{3},F_{1}\right) =\deg g(F_{3})\in 8\Bbb{N}.$ However, $6\notin 8\Bbb{N}.$

Finally, if we assume that $\left( F_{1},F_{2},F_{3}-g\left(
F_{2},F_{1}\right) \right) ,$ for some $g\in \Bbb{C}[X,Y],$ is an elementary
reduction of $\left( F_{1},F_{2},F_{3}\right) $ then 
\begin{equation}
\deg g\left( F_{2},F_{1}\right) =8.  \label{row_exmpl_red_typ_I_6}
\end{equation}
As before, by Proposition \ref{prop_deg_g_fg}, 
\begin{equation}
\deg g\left( F_{2},F_{1}\right) \geq q\left( p\cdot 9-9-6+\deg
[F_{2},F_{1}]\right) +9r,  \label{row_exmpl_red_typ_I_7}
\end{equation}
where $\deg _{Y}g(X,Y)=qp+r,0\leq r<p,p=\frac{6}{\gcd \left( 6,9\right) }=2.$
In this case $p\cdot 9-9-6=3$ is not large enough for our purpose but $\deg
[F_{2},F_{1}]$ is. Indeed, 
\begin{eqnarray*}
\frac{\partial F_{1}}{\partial x_{i}} &=&\frac{\partial u_{1}}{\partial x_{i}%
}+\frac{\partial u_{3}}{\partial x_{i}} \\
&=&\frac{\partial z_{1}}{\partial x_{i}}+\frac{\partial z_{3}}{\partial x_{i}%
}+2z_{1}\frac{\partial z_{1}}{\partial x_{i}}-3z_{2}^{2}\frac{\partial z_{2}%
}{\partial x_{i}}
\end{eqnarray*}
and 
\begin{equation*}
\frac{\partial F_{2}}{\partial x_{i}}=\frac{\partial u_{2}}{\partial x_{i}}=%
\frac{\partial z_{2}}{\partial x_{i}}.
\end{equation*}
Thus, for $1\leq i<j\leq 3,$%
\begin{eqnarray}
\frac{\partial F_{1}}{\partial x_{i}}\frac{\partial F_{2}}{\partial x_{j}}-%
\frac{\partial F_{1}}{\partial x_{j}}\frac{\partial F_{2}}{\partial x_{i}}
&=&\left( \frac{\partial z_{1}}{\partial x_{i}}+\frac{\partial z_{3}}{%
\partial x_{i}}+2z_{1}\frac{\partial z_{1}}{\partial x_{i}}-3z_{2}^{2}\frac{%
\partial z_{2}}{\partial x_{i}}\right) \frac{\partial z_{2}}{\partial x_{j}}
\notag \\
&&-\left( \frac{\partial z_{1}}{\partial x_{j}}+\frac{\partial z_{3}}{%
\partial x_{j}}+2z_{1}\frac{\partial z_{1}}{\partial x_{j}}-3z_{2}^{2}\frac{%
\partial z_{2}}{\partial x_{j}}\right) \frac{\partial z_{2}}{\partial x_{i}}
\notag \\
&=&\left( \frac{\partial z_{1}}{\partial x_{i}}\frac{\partial z_{2}}{%
\partial x_{j}}-\frac{\partial z_{1}}{\partial x_{j}}\frac{\partial z_{2}}{%
\partial x_{i}}\right) +\left( \frac{\partial z_{3}}{\partial x_{i}}\frac{%
\partial z_{2}}{\partial x_{j}}-\frac{\partial z_{3}}{\partial x_{j}}\frac{%
\partial z_{2}}{\partial x_{i}}\right)  \label{row_exmpl_red_typ_I_8} \\
&&+2z_{1}\left( \frac{\partial z_{1}}{\partial x_{i}}\frac{\partial z_{2}}{%
\partial x_{j}}-\frac{\partial z_{1}}{\partial x_{j}}\frac{\partial z_{2}}{%
\partial x_{i}}\right) .  \notag
\end{eqnarray}
Since $z_{1},z_{2},z_{3}$ are algebraically independent, by Corollary \ref
{cor_rank_trdeg} for at least one pair $i,j,1\leq i<j\leq 3,$ we have 
\begin{equation*}
\frac{\partial z_{1}}{\partial x_{i}}\frac{\partial z_{2}}{\partial x_{j}}-%
\frac{\partial z_{1}}{\partial x_{j}}\frac{\partial z_{2}}{\partial x_{i}}%
\neq 0.
\end{equation*}
And since $\deg z_{1}=9,$ for that pair $i,j$ we have 
\begin{equation}
\deg 2z_{1}\left( \frac{\partial z_{1}}{\partial x_{i}}\frac{\partial z_{2}}{%
\partial x_{j}}-\frac{\partial z_{1}}{\partial x_{j}}\frac{\partial z_{2}}{%
\partial x_{i}}\right) \geq 9.  \label{row_exmpl_red_typ_I_9}
\end{equation}
Of course we also have 
\begin{equation}
\deg 2z_{1}\left( \frac{\partial z_{1}}{\partial x_{i}}\frac{\partial z_{2}}{%
\partial x_{j}}-\frac{\partial z_{1}}{\partial x_{j}}\frac{\partial z_{2}}{%
\partial x_{i}}\right) >\deg \left( \frac{\partial z_{1}}{\partial x_{i}}%
\frac{\partial z_{2}}{\partial x_{j}}-\frac{\partial z_{1}}{\partial x_{j}}%
\frac{\partial z_{2}}{\partial x_{i}}\right) .
\label{row_exmpl_red_typ_I_10}
\end{equation}
Since moreover 
\begin{eqnarray*}
\frac{\partial z_{2}}{\partial x_{i}} &=&4\frac{\partial y_{2}}{\partial
x_{i}}+2y_{3}\frac{\partial y_{3}}{\partial x_{i}}, \\
\frac{\partial z_{3}}{\partial x_{i}} &=&\frac{\partial y_{3}}{\partial x_{i}%
}
\end{eqnarray*}
and 
\begin{eqnarray*}
\deg y_{2} &=&\deg \left( x_{2}+x_{1}^{2}\right) =2, \\
\deg y_{3} &=&\deg \left( x_{3}+2x_{1}x_{2}+x_{1}^{3}\right) =3,
\end{eqnarray*}
it follows that 
\begin{eqnarray*}
\frac{\partial z_{2}}{\partial x_{i}}\frac{\partial z_{3}}{\partial x_{j}}-%
\frac{\partial z_{2}}{\partial x_{j}}\frac{\partial z_{3}}{\partial x_{i}}
&=&\left( 4\frac{\partial y_{2}}{\partial x_{i}}+2y_{3}\frac{\partial y_{3}}{%
\partial x_{i}}\right) \frac{\partial y_{3}}{\partial x_{j}}-\left( 4\frac{%
\partial y_{2}}{\partial x_{j}}+2y_{3}\frac{\partial y_{3}}{\partial x_{j}}%
\right) \frac{\partial y_{3}}{\partial x_{i}} \\
&=&4\left( \frac{\partial y_{2}}{\partial x_{i}}\frac{\partial y_{3}}{%
\partial x_{j}}-\frac{\partial y_{2}}{\partial x_{j}}\frac{\partial y_{3}}{%
\partial x_{i}}\right) ,
\end{eqnarray*}
and so 
\begin{equation}
\deg \left( \frac{\partial z_{2}}{\partial x_{i}}\frac{\partial z_{3}}{%
\partial x_{j}}-\frac{\partial z_{2}}{\partial x_{j}}\frac{\partial z_{3}}{%
\partial x_{i}}\right) =\deg \left( \frac{\partial y_{2}}{\partial x_{i}}%
\frac{\partial y_{3}}{\partial x_{j}}-\frac{\partial y_{2}}{\partial x_{j}}%
\frac{\partial y_{3}}{\partial x_{i}}\right) \leq 3.
\label{row_exmpl_red_typ_I_11}
\end{equation}
Finally, by (\ref{row_exmpl_red_typ_I_8}) - (\ref{row_exmpl_red_typ_I_11}), 
\begin{equation}
\deg \left[ F_{1},F_{2}\right] \geq 11.  \label{row_exmpl_red_typ_I_12}
\end{equation}

Now, using (\ref{row_exmpl_red_typ_I_12}) and (\ref{row_exmpl_red_typ_I_7})
we find that 
\begin{equation}
\deg g\left( F_{2},F_{1}\right) \geq q\cdot 14+9r.
\label{row_exmpl_red_typ_I_13}
\end{equation}
Thus, by (\ref{row_exmpl_red_typ_I_13}) and (\ref{row_exmpl_red_typ_I_6}),
we have $q=r=0.$ This means that $g\left( X,Y\right) =g\left( X\right) $ and 
$\deg g\left( F_{2},F_{1}\right) =\deg g\left( F_{2}\right) \in 6\Bbb{N},$
contrary to $8\notin 6\Bbb{N}.$

For more information about polynomial automorphisms which admit reductions
of type I see \cite{Kuroda2}.

\begin{definition}
\label{def_type_II}Let $\Theta =\left( f_{1},f_{2},f_{3}\right) \,$be an
automorphism of $A=\Bbb{C}[X,Y,Z]$ such that (for some $n\in \Bbb{N}^{*}$) $%
\deg f_{1}=2n,$ $\deg f_{2}=3n,$ $\frac{3}{2}n<\deg f_{3}\leq 2n$ and $%
\overline{f}_{1},\overline{f}_{3}$ are linearly independent. Suppose that
there exist $\alpha ,\beta \in \Bbb{C}$ with $\left( \alpha ,\beta \right)
\neq \left( 0,0\right) $ such that the elements $g_{1}=f_{1}-\alpha f_{3},$ $%
g_{2}=f_{2}-\beta f_{3}$ satisfy the following conditions:\newline
(i) $g_{1},g_{2}$ is a 2-reduced pair and $\deg g_{1}=\deg f_{1},$ $\deg
g_{2}=\deg f_{2};$\newline
(ii) the automorphism $\left( g_{1},g_{2},f_{3}\right) $ admits an
elementary reduction $\left( g_{1},g_{2},g_{3}\right) $ with $\deg \left[
g_{1},g_{3}\right] <\deg g_{2}+\deg \left[ g_{1},g_{2}\right] .$\newline
Then we will say that $\Theta $ admits a reduction $\left(
g_{1},g_{2},g_{3}\right) $ of type II.\newline
We will also say that a polynomial automorphism $F=\left(
F_{1},F_{2},F_{3}\right) $ admits a reduction of type II if for some
permutation $\sigma $ of $\left\{ 1,2,3\right\} ,$ the automorphism $\Theta
=\left( F_{\sigma (1)},F_{\sigma (2)},F_{\sigma (3)}\right) $ admits a
reduction of type II.
\end{definition}

\begin{definition}
\label{def_type_III_IV}Let $\Theta =\left( f_{1},f_{2},f_{3}\right) \,$be an
automorphism of $A=\Bbb{C}[X,Y,Z]$ such that (for some $n\in \Bbb{N}^{*}$) $%
\deg f_{1}=2n,$ and either 
\begin{equation*}
\deg f_{2}=3n,\qquad n<\deg f_{3}\leq \frac{3}{2}n,
\end{equation*}
or 
\begin{equation*}
\frac{5}{2}n<\deg f_{2}\leq 3n,\qquad \deg f_{3}=\frac{3}{2}n.
\end{equation*}
Suppose that there exist $\alpha ,\beta ,\gamma \in \Bbb{C}$ such that the
elements $g_{1}=f_{1}-\beta f_{3},$ $g_{2}=f_{2}-\gamma f_{3}-\alpha
f_{3}^{2}$ satisfy the following conditions:\newline
(i) $g_{1},g_{2}$ is a 2-reduced pair and $\deg g_{1}=2n,$ $\deg g_{2}=3n;$%
\newline
(ii) there exists $g_{3}$ of the form $g_{3}=\sigma f_{3}+g,$ where $\sigma
\in \Bbb{C}^{*},$ $g\in \Bbb{C}\left[ g_{1},g_{2}\right] ,$ such that $\deg
g_{3}\leq \frac{3}{2}n,$ $\deg \left[ g_{1},g_{3}\right] <3n+\deg \left[
g_{1},g_{2}\right] .$ \newline
If $\left( \alpha ,\beta ,\gamma \right) \neq \left( 0,0,0\right) $ and $%
\deg g_{3}<n+\deg \left[ g_{1},g_{2}\right] ,$ then we will say that $\Theta 
$ admits a reduction $\left( g_{1},g_{2},g_{3}\right) $ of type III. On the
other hand, if there exists $\mu \in \Bbb{C}^{*}$ such that $\deg \left(
g_{2}-\mu g_{3}^{2}\right) \leq 2n,$ then we will say that $\Theta $ admits
a reduction $\left( g_{1},g_{2}-\mu g_{3}^{2},g_{3}\right) $ of type IV.%
\newline
We will also say that a polynomial automorphism $F=\left(
F_{1},F_{2},F_{3}\right) $ admits a reduction of type III (type IV) if for
some permutation $\sigma $ of $\left\{ 1,2,3\right\} ,$ the automorphism $%
\Theta =\left( F_{\sigma (1)},F_{\sigma (2)},F_{\sigma (3)}\right) $ admits
a reduction of type III (type IV).
\end{definition}

Now, we can present the above mentioned theorem.

\begin{theorem}
\label{tw_type_1-4}\textit{(\cite{sh umb1}, Theorem 3) }Let $%
F=(F_{1},F_{2},F_{3})\,$be a tame automorphism of $\Bbb{C}^{3}.$ If $\deg
F_{1}+\deg F_{2}+\deg F_{3}>3$ (in other words, if $F$ is not an affine
automorphism), then $F$ admits either an elementary reduction or a reduction
of one of types I-IV.
\end{theorem}

\subsection{Some number theory}

We will use the following result from number theory, connected with the
so-called coin problem or Frobenius problem.

\begin{theorem}
\label{tw_sywester}\textit{(see e.g. \cite{Brauer}) }If $d_{1},d_{2}$ are
positive integers such that $\gcd (d_{1},d_{2})=1,$ then for every integer $%
k\geq (d_{1}-1)(d_{2}-1)$ there are $k_{1},k_{2}\in \Bbb{N}$ such that 
\begin{equation*}
k=k_{1}d_{1}+k_{2}d_{2}.
\end{equation*}
Moreover $(d_{1}-1)(d_{2}-1)-1\notin d_{1}\Bbb{N}+d_{2}\Bbb{N}.$
\end{theorem}

The proof of the above theorem can be found in the number theory literature,
but for the convenience of the reader we give it here. In the proof we will
write $M\left( d_{1},d_{2}\right) $ for the minimal $s\in \Bbb{N}$ such that 
$\left\{ s,s+1,\ldots \right\} \subset d_{1}\Bbb{N}+d_{2}\Bbb{N}.$ Let us
mention that the so-called Frobenius number (the maximal $s\in \Bbb{N}$ such
that $s\notin d_{1}\Bbb{N}+d_{2}\Bbb{N}$) is equal to $M\left(
d_{1},d_{2}\right) -1.$

\begin{proof}
Without loss of generality we can assume that $1<d_{1}\leq d_{2}.$ Indeed,
if $d_{1}=1,$ then $d_{1}\Bbb{N+}d_{2}\Bbb{N=N}$ and $(d_{1}-1)(d_{2}-1)=0.$
Thus for any $r=1,\ldots ,d_{1}-1$ there are integers $k_{1,r},k_{2,r}\in 
\Bbb{Z}$ such that 
\begin{equation*}
k_{1,r}d_{1}+k_{2,r}d_{2}=r.
\end{equation*}
Since $d_{1},d_{2},r>0$ and $r<d_{1}\leq d_{2},$ we have $k_{1,r}k_{2,r}<0.$
Moreover, since $%
(k_{1,r}-d_{2})d_{1}+(k_{2,r}+d_{1})d_{2}=k_{1,r}d_{1}+k_{2,r}d_{2}=r,$ we
can assume that $k_{2,r}>0.$ Notice that we can assume even more, namely
that $k_{2,r}>0$ and $k_{1,r}\geq d_{2}-1.\,$ Indeed, let $%
k_{1,r},k_{2,r}\in \Bbb{Z}$ be such that $k_{1,r}d_{1}+k_{2,r}d_{2}=r,$ $%
k_{2,r}>0$ and there are no $k_{1,r}^{\prime },k_{2,r}^{\prime }\in \Bbb{Z}$
such that $k_{1,r}^{\prime }d_{1}+k_{2,r}^{\prime }d_{2}=r,$ $%
k_{2,r}^{\prime }>0$ and $k_{2,r}^{\prime }<k_{2,r}.$ Then, since $%
(k_{1,r}+d_{2})d_{1}+(k_{2,r}-d_{1})d_{2}=k_{1,r}d_{1}+k_{2,r}d_{2}=r,$ we
have $k_{2,r}-d_{1}\leq 0$ (since $r<d_{1}\leq d_{2}$ we actually have $%
k_{2,r}-d_{1}<0$). Thus $k_{1,r}+d_{2}>0,$ and so $k_{1,r}\geq d_{2}-1.$

It is easy to see that to show that any natural number $k\geq
(d_{1}-1)(d_{2}-1)$ is in $d_{1}\Bbb{N+}d_{2}\Bbb{N},$ we only need to show
that 
\begin{equation*}
(d_{1}-1)(d_{2}-1),(d_{1}-1)(d_{2}-1)+1,\ldots
,(d_{1}-1)(d_{2}-1)+d_{1}-1\in d_{1}\Bbb{N+}d_{2}\Bbb{N}.
\end{equation*}
First, 
\begin{eqnarray*}
(d_{1}-1)(d_{2}-1)
&=&(d_{2}-1)d_{1}-d_{2}+1=(d_{2}-1)d_{1}-d_{2}+k_{1,1}d_{1}+k_{2,1}d_{2} \\
&=&(d_{2}-1+k_{1,1})d_{1}+(k_{2,1}-1)d_{2}\in d_{1}\Bbb{N+}d_{2}\Bbb{N},
\end{eqnarray*}
because $k_{1,1}\geq d_{2}-1$ and $k_{2,1}>0.$ Similarly, we show that $%
(d_{1}-1)(d_{2}-1)+1=(d_{2}-1)d_{1}-d_{2}+2,\ldots
,(d_{1}-1)(d_{2}-1)+d_{1}-2=(d_{2}-1)d_{1}-d_{2}+(d_{1}-1)\in d_{1}\Bbb{N+}%
d_{2}\Bbb{N}.$ To see that $(d_{1}-1)(d_{2}-1)+d_{1}-1\in d_{1}\Bbb{N+}d_{2}%
\Bbb{N}$ we write 
\begin{equation*}
(d_{1}-1)(d_{2}-1)+d_{1}-1=d_{1}d_{2}-d_{1}-d_{2}+1+d_{1}-1=(d_{1}-1)d_{2}.
\end{equation*}
Thus we have shown that $M(d_{1},d_{2})\leq (d_{1}-1)(d_{2}-1).$

To prove that $M(d_{1},d_{2})=\left( d_{1}-1\right) \left( d_{2}-1\right) $
it is enough to show that $(d_{1}-1)(d_{2}-1)-1\notin d_{1}\Bbb{N+}d_{2}\Bbb{%
N}.\,$Since $(d_{2}-1)d_{1}-d_{2}=(d_{1}-1)(d_{2}-1)-1$ and $\limfunc{lcm}%
(d_{1},d_{2})=d_{1}d_{2},$ it follows that 
\begin{eqnarray*}
\{(k_{1},k_{2}) &\in &\Bbb{Z}^{2}\ |\
k_{1}d_{1}+k_{2}d_{2}=(d_{1}-1)(d_{2}-1)-1\} \\
&=&\{(d_{2}-1-ld_{2},ld_{1}-1)\ |\ l\in \Bbb{Z}\}.
\end{eqnarray*}
But $\{(d_{2}-1-ld_{2},ld_{1}-1)|\ l\in \Bbb{Z}\}\cap \Bbb{N}^{2}=\emptyset
. $ This ends the proof.
\end{proof}

\section{Some useful results\label{section_useful_results}}

\subsection{Some simple remarks}

In this section we make some simple but useful remarks about existence of
automorphisms and tame automorphisms with given multidegree.

\begin{proposition}
\textit{(}\cite{Karas}\label{prop_m<n}, Prop. 2.1) If for $1\leq d_{1}\leq
\ldots \leq d_{n}$ there is a sequence of integers $1\leq i_{1}<\ldots
<i_{m}\leq n,$ such that there exists an automorphism $G$ of $\Bbb{C}^{m}$
with $\limfunc{mdeg}G=(d_{i_{1}},\ldots ,d_{i_{m}}),$ then there exists an
automorphism $F$ of $\Bbb{C}^{n}$ with $\limfunc{mdeg}F=(d_{1},\ldots
,d_{n}).$ Moreover, if $G$ is tame, then $F$ can also be found tame.
\end{proposition}

\begin{proof}
Without loss of generality we can assume that $m<n.$ Let $1\leq j_{1}<\ldots
<j_{n-m}\leq n$ be such that $\{i_{1},\ldots ,i_{m}\}\cup \{j_{1},\ldots
,j_{n-m}\}=\{1,\ldots ,n\}.$ Then, of course, $\{i_{1},\ldots ,i_{m}\}\cap
\{j_{1},\ldots ,j_{n-m}\}=\emptyset .$ Consider the mapping $h=(h_{1},\ldots
,h_{n}):\Bbb{C}^{n}\rightarrow \Bbb{C}^{n}$ given by 
\begin{equation*}
h_{k}(x_{1},\ldots ,x_{n})=\left\{ 
\begin{array}{ll}
x_{k} & \text{for }k\in \{i_{1},\ldots ,i_{m}\}, \\ 
x_{k}+(x_{i_{1}})^{d_{k}}\quad & \text{for }k\in \{j_{1},\ldots ,j_{n-m}\}.
\end{array}
\right.
\end{equation*}
Of course $h$ is an automorphism of $\Bbb{C}^{n}$ and $\deg h_{k}=d_{k}$ for 
$k\in \{j_{1},\ldots ,j_{n-m}\}.$

Consider also the mapping $g=(g_{1},\ldots ,g_{n}):\Bbb{C}^{n}\rightarrow 
\Bbb{C}^{n}$ given by 
\begin{equation*}
g_{k}(u_{1},\ldots ,u_{n})=\left\{ 
\begin{array}{ll}
G_{l}(u_{i_{1}},\ldots ,u_{i_{m}})\quad & \text{for }k=i_{l}, \\ 
u_{k}\quad & \text{for }k\in \{j_{1},\ldots ,j_{n-m}\}.
\end{array}
\right.
\end{equation*}
Then $g$ is an automorphism of $\Bbb{C}^{n}$ and $\deg g_{k}=d_{k}$ for $%
k\in \{i_{1},\ldots ,i_{m}\}.$

Now $F=g\circ h$ is an automorphism of $\Bbb{C}^{n}$ (tame when $G$ is tame)
with $\limfunc{mdeg}F=\left( d_{1},\ldots ,d_{n}\right) .$
\end{proof}

\begin{proposition}
\label{prop_sum_d_i} \textit{(}\cite{Karas}, Prop. 2.2) If for a sequence of
integers $1\leq d_{1}\leq \ldots \leq d_{n}$ there is $i\in \{1,\ldots ,n\}$
such that 
\begin{equation*}
d_{i}=\sum_{j=1}^{i-1}k_{j}d_{j}\qquad \text{with }k_{j}\in \Bbb{N},
\end{equation*}
then there exists a tame automorphism $F$ of $\Bbb{C}^{n}$ with $\limfunc{%
mdeg}F=(d_{1},\ldots ,d_{n}).$
\end{proposition}

\begin{proof}
Define $h=(h_{1},\ldots ,h_{n}):\Bbb{C}^{n}\rightarrow \Bbb{C}^{n}$ and $%
g=(g_{1},\ldots ,g_{n}):\Bbb{C}^{n}\rightarrow \Bbb{C}^{n}$ by 
\begin{equation*}
h_{k}(x_{1},\ldots ,x_{n})=\left\{ 
\begin{array}{ll}
x_{k}\quad & \text{for }k=i, \\ 
x_{k}+x_{i}^{d_{k}}\quad & \text{for }k\neq i,
\end{array}
\right.
\end{equation*}
and 
\begin{equation*}
g_{k}(u_{1},\ldots ,u_{n})=\left\{ 
\begin{array}{ll}
u_{k}+u_{1}^{k_{1}}\cdots u_{i-1}^{k_{i-1}}\quad & \text{for }k=i, \\ 
u_{k}\quad & \text{for }k\neq i.
\end{array}
\right.
\end{equation*}
It is easy to see that $F=g\circ h$ is a tame automorphism with $\limfunc{%
mdeg}F=(d_{1},\ldots ,d_{n}).$
\end{proof}

The above proposition implies the following result.

\begin{corollary}
\textit{(}\cite{Karas}, Cor. 2.3) \label{cor_male_d1}If $1\leq d_{1}\leq
\ldots \leq d_{n}$ is a sequence of integers with $d_{1}\leq n-1,$ then
there exists a tame automorphism $F$ of $\Bbb{C}^{n}$ with $\limfunc{mdeg}%
F=(d_{1},\ldots ,d_{n}).$
\end{corollary}

\begin{proof}
Let $r_{i}\in \{0,1,\ldots ,d_{1}-1\},$ for $i=2,\ldots ,n,$ be such that $%
d_{i}\equiv r_{i}(\func{mod}d_{1}).$ If there is an $i\in \{2,\ldots ,n\}$
such that $r_{i}=0,$ then $d_{i}=kd_{1}$ for some $k\in \Bbb{N}^{*}$ and by
Proposition \ref{prop_sum_d_i}, there exists a tame automorphism $F$ of $%
\Bbb{C}^{n}$ with the desired properties.

Thus assume that $r_{i}\neq 0$ for all $i=2,\ldots ,n.$ Since $d_{1}-1<n-1,$
there are $i,j\in \{2,\ldots ,n\},$ $i\neq j,$ such that $r_{i}=r_{j}.$
Without loss of generality we can assume that $i<j.$ Then $%
d_{j}=d_{i}+kd_{1} $ for some $k\in \Bbb{N},$ and by Proposition \ref
{prop_sum_d_i} there exists a tame automorphism $F$ of $\Bbb{C}^{n}$ with
the desired properties.
\end{proof}

The above corollary can be improved as follows.

\begin{theorem}
\label{tw_gcd_male} If $1\leq d_{1}\leq \ldots \leq d_{n}$ is a sequence of
integers with 
\begin{equation*}
\frac{d_{1}}{\gcd \left( d_{1},\ldots ,d_{n}\right) }\leq n-1,
\end{equation*}
then there exists a tame automorphism $F$ of $\Bbb{C}^{n}$ with $\limfunc{%
mdeg}F=(d_{1},\ldots ,d_{n}).$
\end{theorem}

\begin{proof}
Let $d=\gcd (d_{1},\ldots ,d_{n}).$ Then the numbers $r_{2},\ldots ,r_{n}$
defined as in the proof of Corollary \ref{cor_male_d1} satisfy $r_{i}\in
\{0,d,2d,\ldots ,d_{1}-d\}$ for $i=2,\ldots ,n.$ Since the number of
elements of the set $\{0,d,2d,\ldots ,d_{1}-d\}$ is equal to 
\begin{equation*}
\frac{d_{1}}{\gcd \left( d_{1},\ldots ,d_{n}\right) }\leq n-1,
\end{equation*}
we can use the same arguments as in the proof of Corollary \ref{cor_male_d1}.
\end{proof}

Combining Theorem \ref{tw_gcd_male} and Proposition \ref{prop_m<n} we obtain
the following result.

\begin{corollary}
If for $1\leq d_{1}\leq \ldots \leq d_{n}$ there is a sequence of integers $%
1\leq i_{1}<\ldots <i_{m}\leq n,$ such that 
\begin{equation*}
\frac{d_{i_{1}}}{\gcd \left( d_{i_{1}},\ldots ,d_{i_{m}}\right) }\leq m-1,
\end{equation*}
then there exists a tame automorphism $F$ of $\Bbb{C}^{n}$ with $\limfunc{%
mdeg}F=(d_{1},\ldots ,d_{n}).$
\end{corollary}

\subsection{Reducibility of type I and II}

Now we will show that in our considerations we do not need to pay attention
to reducibility of type I and II.

\begin{lemma}
\label{lem_red_I_II}Let $\left( d_{1},d_{2},d_{3}\right) \neq \left(
1,1,1\right) ,$ $d_{1}\leq d_{2}\leq d_{3}$ be a sequence of positive
integers. If there is an automorphism (a tame automorphism) $F:\Bbb{C}%
^{3}\rightarrow \Bbb{C}^{3}$ such that $F$ admits a reduction of type I or
II and $\limfunc{mdeg}F=\left( d_{1},d_{2},d_{3}\right) ,\,$then there is
also an automorphism (a tame automorphism) $\widetilde{F}:\Bbb{C}%
^{3}\rightarrow \Bbb{C}^{3}$ such that $\widetilde{F}$ admits an elementary
reduction and $\limfunc{mdeg}\widetilde{F}=\left( d_{1},d_{2},d_{3}\right) .$
Moreover, if $F\left( 0,0,0\right) =\left( 0,0,0\right) ,$ then $\widetilde{F%
}$ can also be found such that $\widetilde{F}\left( 0,0,0\right) =\left(
0,0,0\right) .$
\end{lemma}

\begin{proof}
Assume that $F=\left( F_{1},F_{2},F_{3}\right) $ admits a reduction of type
I. By Definition \ref{def_type_I} there is a permutation $\sigma \,$of $%
\left\{ 1,2,3\right\} $ and $\alpha \in \Bbb{C}^{*}$ such that the elements $%
g_{1}=F_{\sigma (1)},g_{2}=F_{\sigma (2)}-\alpha F_{\sigma (3)}$ satisfy the
following conditions:\newline
(i) $g_{1},g_{2}$ is a 2-reduced pair and $\deg g_{1}=\deg F_{\sigma (1)},$ $%
\deg g_{2}=\deg F_{\sigma (2)};$\newline
(ii) the automorphism $\left( g_{1},g_{2},F_{\sigma (3)}\right) $ admits an
elementary reduction of the form $\left( g_{1},g_{2},g_{3}\right) .$

For simplicity of notation (and without loss of generality) we assume that $%
\sigma =\limfunc{id}_{\left\{ 1,2,3\right\} }.$ Thus we can take $\widetilde{%
F}=\left( g_{1},g_{2},F_{3}\right) .$

If $F$ admits a reduction of type II we can use a similar construction to
obtain an automorphism $\widetilde{F}.$

Since $\widetilde{F}=G\circ F,$ where 
\begin{equation*}
G:\Bbb{C}^{3}\ni \left\{ 
\begin{array}{l}
x \\ 
y \\ 
z
\end{array}
\right\} \mapsto \left\{ 
\begin{array}{l}
x \\ 
y-\alpha z \\ 
z
\end{array}
\right\} \in \Bbb{C}^{3}\qquad \text{(for type I)}
\end{equation*}
or 
\begin{equation*}
G:\Bbb{C}^{3}\ni \left\{ 
\begin{array}{l}
x \\ 
y \\ 
z
\end{array}
\right\} \mapsto \left\{ 
\begin{array}{l}
x-\alpha z \\ 
y-\beta z \\ 
z
\end{array}
\right\} \in \Bbb{C}^{3}\qquad \text{(for type II)}
\end{equation*}
$\widetilde{F}$ is tame if and only if $F$ is tame. It is also clear that $%
\widetilde{F}\left( 0,0,0\right) =\left( 0,0,0\right) $ when $F\left(
0,0,0\right) =\left( 0,0,0\right) .$
\end{proof}

The above lemma also implies the following

\begin{proposition}
\label{prop_reduc_type_1_2}Let $\left( d_{1},d_{2},d_{3}\right) \neq \left(
1,1,1\right) ,$ $d_{1}\leq d_{2}\leq d_{3},$ be a sequence of positive
integers. If there is a tame automorphism $F:\Bbb{C}^{3}\rightarrow \Bbb{C}%
^{3}$ with $\limfunc{mdeg}F=\left( d_{1},d_{2},d_{3}\right) ,\,$then there
is also a tame automorphism $\widetilde{F}:\Bbb{C}^{3}\rightarrow \Bbb{C}%
^{3} $ such that $\limfunc{mdeg}\widetilde{F}=\left(
d_{1},d_{2},d_{3}\right) $ and $\widetilde{F}$ admits either an elementary
reduction or a reduction of type III or IV. Moreover we can require that $%
\widetilde{F}\left( 0,0,0\right) =\left( 0,0,0\right) .$
\end{proposition}

\begin{proof}
Let $F=\left( F_{1},F_{2},F_{3}\right) :\Bbb{C}^{3}\rightarrow \Bbb{C}^{3}$
be any tame automorphism with $\limfunc{mdeg}F=\left(
d_{1},d_{2},d_{3}\right) $ and let $T:\Bbb{C}^{3}\rightarrow \Bbb{C}^{3}$ be
the translation given by 
\begin{equation*}
T:\Bbb{C}^{3}\ni \left( x,y,z\right) \mapsto \left( x-F_{1}\left( 0\right)
,y-F_{2}\left( 0\right) ,z-F_{3}\left( 0\right) \right) \in \Bbb{C}^{3}.
\end{equation*}
Then obviously $T\circ F$ is a tame automorphism of $\Bbb{C}^{3}$ such that $%
\limfunc{mdeg}\left( T\circ F\right) =\limfunc{mdeg}F=\left(
d_{1},d_{2},d_{3}\right) $ and $\left( T\circ F\right) \left( 0,0,0\right)
=\left( 0,0,0\right) .$ If $T\circ F$ admits either an elementary reduction
or a reduction of type III or IV, then we take $\widetilde{F}=T\circ F.$ And
if $T\circ F$ admits a reduction of type I or II, then we can use Lemma \ref
{lem_red_I_II}.
\end{proof}

In particular Proposition \ref{prop_reduc_type_1_2} says that reductions of
type I and II are irrelevant for our considerations. To be precise we
formulate the following

\begin{theorem}
\label{tw_reduc+type_1_2}Let $\left( d_{1},d_{2},d_{3}\right) \neq \left(
1,1,1\right) ,$ $d_{1}\leq d_{2}\leq d_{3}$ be a sequence of positive
integers. To prove that there is no tame automorphism of $\Bbb{C}^{3}$ with
multidegree $\left( d_{1},d_{2},d_{3}\right) $ it is enough to show that a
(hypothetical) automorphism $F$ of $\Bbb{C}^{3}$ with $\limfunc{mdeg}%
F=\left( d_{1},d_{2},d_{3}\right) $ admits neither an elementary reduction
nor a reduction of type III or IV. Moreover, we can restrict our attention
to automorphisms $F$ with $F\left( 0,0,0\right) =\left( 0,0,0\right) .$
\end{theorem}

To end this section, let us look again at Example \ref{exmple_red_type_I}.
If $F$ is the automorphism from that example, then $\limfunc{mdeg}F=(9,6,8)$
or $(6,8,9)$ after permutation of coordinates. This automorphism does not
admit an elementary reduction and admits a reduction of type I. One can
easily see that (in the notation of Example \ref{exmple_red_type_I}) 
\begin{equation*}
T_{2}\circ T_{1}=T_{3}^{-1}\circ L^{-1}\circ F
\end{equation*}
is a reduction of type I of $F.$ Moreover for $\widetilde{F}=L^{-1}\circ F$
we have 
\begin{equation*}
\limfunc{mdeg}\widetilde{F}=\limfunc{mdeg}F
\end{equation*}
and $T_{3}^{-1}\circ \widetilde{F}$ is an elementary reduction of $%
\widetilde{F}.$

\subsection{Reducibility of type III}

First of all notice that if $1\leq d_{1}\leq d_{2}\leq d_{3}$ are such that $%
\limfunc{mdeg}F=\left( d_{1},d_{2},d_{3}\right) $\thinspace for some
automorphism $F$ that admits a reduction of type III, then by Definition \ref
{def_type_III_IV} there is $n\in \Bbb{N}^{*}$ such that 
\begin{equation*}
d_{\sigma (1)}=2n
\end{equation*}
and either 
\begin{equation*}
d_{\sigma (2)}=3n,\qquad n<d_{\sigma (3)}\leq \frac{3}{2}n,
\end{equation*}
or 
\begin{equation*}
\frac{5}{2}n<d_{\sigma (2)}\leq 3n,\qquad d_{\sigma (3)}=\frac{3}{2}n
\end{equation*}
for some permutation $\sigma \,$of $\left\{ 1,2,3\right\} .$ Since $\frac{3}{%
2}n<2n<\min \left\{ \frac{5}{2}n,3n\right\} ,$ we must actually have 
\begin{equation*}
d_{2}=2n
\end{equation*}
and either 
\begin{equation*}
d_{3}=3n,\qquad n<d_{1}\leq \frac{3}{2}n,
\end{equation*}
or 
\begin{equation*}
\frac{5}{2}n<d_{3}\leq 3n,\qquad d_{1}=\frac{3}{2}n.
\end{equation*}

Thus we have the following remark.

\begin{remark}
\label{remakrk_red_type_3}If an automorphism $F$ of $\Bbb{C}^{3}$ with $%
\limfunc{mdeg}F=\left( d_{1},d_{2},d_{3}\right) ,$ $1\leq d_{1}\leq
d_{2}\leq d_{3},$ admits a reduction of type III, then\newline
(1) $2|d_{2},$\newline
(2) $3|d_{1}$ or $\frac{d_{3}}{d_{2}}=\frac{3}{2}.$
\end{remark}

Because of the remark above it is natural to consider the situation of the
following lemma.

\begin{lemma}
\label{lem_type_3}Let $\left( d_{1},d_{2},d_{3}\right) \neq \left(
1,1,1\right) ,$ $d_{1}\leq d_{2}\leq d_{3}$ be a sequence of positive
integers such that $\frac{d_{3}}{d_{2}}=\frac{3}{2}.$ If there is an
automorphism (a tame automorphism) $F:\Bbb{C}^{3}\rightarrow \Bbb{C}^{3}$
such that $F$ admits a reduction of type III and $\limfunc{mdeg}F=\left(
d_{1},d_{2},d_{3}\right) ,\,$then there is also an automorphism (a tame
automorphism) $\widetilde{F}:\Bbb{C}^{3}\rightarrow \Bbb{C}^{3}$ such that $%
\widetilde{F}$ admits an elementary reduction and $\limfunc{mdeg}\widetilde{F%
}=\left( d_{1},d_{2},d_{3}\right) .$ Moreover, if $F\left( 0,0,0\right)
=\left( 0,0,0\right) ,$ then $\widetilde{F}$ can also be found such that $%
\widetilde{F}\left( 0,0,0\right) =\left( 0,0,0\right) .$
\end{lemma}

In the proof of this lemma we will use the following result.

\begin{lemma}
(\textit{\cite{sh umb2}, Corollary 4}) If an automorphism $\left(
g_{1},g_{2},g_{3}\right) $ is a reduction of type III of an automorphism $%
\left( f_{1},f_{2},f_{3}\right) ,$ then 
\begin{equation*}
\deg g_{1}+\deg g_{2}+\deg g_{3}<\deg f_{1}+\deg f_{2}+\deg f_{3}.
\end{equation*}
\end{lemma}

\begin{proof}
\textit{of Lemma \ref{lem_type_3} }Assume that $F=\left(
F_{1},F_{2},F_{3}\right) $ admits a reduction of type III. By the above
considerations, the conditions of Definition \ref{def_type_III_IV} must be
satisfied for the automorphism $\theta =\left( f_{1},f_{2},f_{3}\right)
=\left( F_{2},F_{3},F_{1}\right) .$ Also by Definition \ref{def_type_III_IV}
there are $n\in \Bbb{N}^{*}$ and $\alpha ,\beta ,\gamma \in \Bbb{C},\left(
\alpha ,\beta ,\gamma \right) \neq \left( 0,0,0\right) ,$ such that the
elements $g_{1}=f_{1}-\beta f_{3},$ $g_{2}=f_{2}-\gamma f_{3}-\alpha
f_{3}^{2}$ satisfy the following conditions:\newline
(i) $g_{1},g_{2}$ is a 2-reduced pair and $\deg g_{1}=2n,\deg g_{2}=3n;$%
\newline
(ii) there exists $g_{3}$ of the form $g_{3}=\sigma f_{3}+g,$ where $\sigma
\in \Bbb{C}^{*},$ $g\in \Bbb{C}\left[ g_{1},g_{2}\right] ,$ such that $\deg
g_{3}\leq \frac{3}{2}n,\deg \left[ g_{1},g_{3}\right] <3n+\deg \left[
g_{1},g_{2}\right] ;$ \newline
(iii) $\deg g_{3}<n+\deg \left[ g_{1},g_{2}\right] .$

Let us notice that apart from $g_{3}=\sigma f_{3}+g\,$we can also take $%
\widetilde{g}_{3}=f_{3}+\frac{1}{\sigma }g=f_{3}+\widetilde{g},$ with $%
\widetilde{g}=\frac{1}{\sigma }g\in \Bbb{C}\left[ g_{1},g_{2}\right] .$

Since in our situation, i.e. $\frac{d_{3}}{d_{2}}=\frac{3}{2},$ we have $%
d_{2}=2n,$ $d_{3}=3n$ and hence $\deg F_{2}=\deg f_{1}=2n=\deg g_{1}$ and $%
\deg F_{3}=\deg f_{2}=3n=\deg g_{2},$ the lemma above yields $\deg
g_{3}<\deg f_{3}=\deg F_{1}=d_{1}.$ This means that the automorphism $\left(
g_{1},g_{2},f_{3}\right) ,$ and hence $\widetilde{F}=\left(
F_{1},g_{1},g_{2}\right) ,$ admits an elementary reduction. Of course $%
\limfunc{mdeg}\left( F_{1},g_{1},g_{2}\right) =\limfunc{mdeg}$ $\left(
F_{1},F_{2},F_{3}\right) .$

Since $\widetilde{F}=T_{2}\circ T_{1}\circ F,$ where the mappings 
\begin{equation*}
T_{1}:\Bbb{C}^{3}\ni \left\{ 
\begin{array}{l}
x \\ 
y \\ 
z
\end{array}
\right\} \mapsto \left\{ 
\begin{array}{l}
x \\ 
y-\beta x \\ 
z-\gamma x-\alpha x^{2}
\end{array}
\right\} \in \Bbb{C}^{3}
\end{equation*}
and 
\begin{equation*}
T_{2}:\Bbb{C}^{3}\ni \left\{ 
\begin{array}{l}
x \\ 
y \\ 
z
\end{array}
\right\} \mapsto \left\{ 
\begin{array}{l}
x+\widetilde{g}\left( y,z\right) \\ 
y \\ 
z
\end{array}
\right\} \in \Bbb{C}^{3}
\end{equation*}
are triangular automorphisms, $\widetilde{F}$ is tame if and only if $F$ is
tame.

Since $\deg F_{1}>0,$ also $\deg \widetilde{g}>0,$ and hence $\overline{%
\widetilde{g}}=\overline{\widetilde{g}-a}$ for all $a\in \Bbb{C}.$ Thus we
can assume that $\widetilde{g}\left( 0,0\right) =0.$ Then $\widetilde{F}%
\left( 0,0,0\right) =\left( 0,0,0\right) $ when $F\left( 0,0,0\right)
=\left( 0,0,0\right) .$
\end{proof}

By Lemma \ref{lem_type_3} we also have the following result.

\begin{proposition}
\label{prop_reduc_type_3}Let $\left( d_{1},d_{2},d_{3}\right) \neq \left(
1,1,1\right) ,$ $d_{1}\leq d_{2}\leq d_{3},$ be a sequence of positive
integers such that $\frac{d_{3}}{d_{2}}=\frac{3}{2}.$ If there is a tame
automorphism $F:\Bbb{C}^{3}\rightarrow \Bbb{C}^{3}$ such that $\limfunc{mdeg}%
F=\left( d_{1},d_{2},d_{3}\right) ,\,$then there is also a tame automorphism 
$\widetilde{F}:\Bbb{C}^{3}\rightarrow \Bbb{C}^{3}$ such that $\widetilde{F}$
admits either a reduction of type IV or an elementary reduction and $%
\limfunc{mdeg}\widetilde{F}=\left( d_{1},d_{2},d_{3}\right) .$ Moreover we
can require that $\widetilde{F}\left( 0,0,0\right) =\left( 0,0,0\right) .$
\end{proposition}

\begin{proof}
As in the proof of Proposition \ref{prop_reduc_type_1_2}, we consider the
automorphism $T\circ F.$ Then we have three cases: (I)$\;T\circ F$ admits a
reduction of type IV or an elementary reduction; (II) $T\circ F$ admits
reduction of type III; (III) $T\circ F$ admits a reduction of type I or II.
In the first case we put $\widetilde{F}=T\circ F,$ in the second case we use
Lemma \ref{lem_type_3} and in the third case we use Lemma \ref{lem_red_I_II}.
\end{proof}

The above proposition means that whenever $\frac{d_{3}}{d_{2}}=\frac{3}{2},$
reductions of type I, II and III are irrelevant for our considerations. More
precisely, we have the following

\begin{theorem}
\label{tw_reduc_type_3}Let $\left( d_{1},d_{2},d_{3}\right) \neq \left(
1,1,1\right) ,$ $d_{1}\leq d_{2}\leq d_{3},$ be a sequence of positive
integers such that $\frac{d_{3}}{d_{2}}=\frac{3}{2}$ or $3\nmid d_{1}.$ To
prove that there is no tame automorphism of $\Bbb{C}^{3}$ with multidegree $%
\left( d_{1},d_{2},d_{3}\right) $ it is enough to show that a (hypothetical)
automorphism $F$ of $\Bbb{C}^{3}$ with $\limfunc{mdeg}F=\left(
d_{1},d_{2},d_{3}\right) $ admits neither a reduction of type IV nor an
elementary reduction. Moreover, we can restrict our attention to
automorphisms $F:\Bbb{C}^{3}\rightarrow \Bbb{C}^{3}$ such that $F\left(
0,0,0\right) =\left( 0,0,0\right) .$
\end{theorem}

\begin{proof}
Take any $\widetilde{F}\in \limfunc{Tame}(\Bbb{C}^{3})$ with $\limfunc{mdeg}%
\widetilde{F}=(d_{1},d_{2},d_{3}).$ By Theorem \ref{tw_reduc+type_1_2} we
can assume that $\widetilde{F}$ admits either an elementary reduction or a
reduction of type III or IV.

If $\widetilde{F}$ admits a reduction of type III, then by Remark \ref
{remakrk_red_type_3} and by the assumptions we have $\frac{d_{3}}{d_{2}}=%
\frac{3}{2}.$ Thus we can use Proposition \ref{prop_reduc_type_3}.
\end{proof}

\subsection{Reducibility of type IV and Kuroda's result}

In the previous sections we have proved that from our point of view
reductions of type I and II are irrelevant. The same is true for reductions
of type III under an additional assumption (see Theorem \ref{tw_reduc_type_3}%
).

The following result due to Kuroda says that reduction of type IV is also
irrelevant for our aim.

\begin{theorem}
\label{tw_Kuroda}\textit{(\cite{Kuroda3}, Thm. 7.1)} No tame automorphism of 
$\Bbb{C}^{3}$ admits a reduction of type IV.
\end{theorem}

Thus we have the following

\begin{theorem}
\label{tw_reduc_type_4}Let $\left( d_{1},d_{2},d_{3}\right) \neq \left(
1,1,1\right) ,$ $d_{1}\leq d_{2}\leq d_{3},$ be a sequence of positive
integers. To prove that there is no tame automorphism $F$ of $\Bbb{C}^{3}$
with $\limfunc{mdeg}F=\left( d_{1},d_{2},d_{3}\right) $ it is enough to show
that a (hypothetical) automorphism $F$ of $\Bbb{C}^{3}$ with $\limfunc{mdeg}%
F=\left( d_{1},d_{2},d_{3}\right) $ admits neither a reduction of type III
nor an elementary reduction. Moreover, if we additionally assume that $\frac{%
d_{3}}{d_{2}}=\frac{3}{2}$ or $3\nmid d_{1},\,$then it is enough to show
that no (hypothetical) automorphism of $\Bbb{C}^{3}$ with multidegree $%
\left( d_{1},d_{2},d_{3}\right) $ admits an elementary reduction. In both
cases we can restrict our attention to automorphisms $F:\Bbb{C}%
^{3}\rightarrow \Bbb{C}^{3}$ such that $F\left( 0,0,0\right) =\left(
0,0,0\right) .$
\end{theorem}

\begin{proof}
The proof is similar to the proof of Theorem \ref{tw_reduc_type_3}.
\end{proof}

\subsection{Reducibility and linear change of coordinates}

Now we make some remarks that will be useful in considerations of some
special cases. The main result of this section says that we can restrict our
attention to the automorphisms whose linear part is the identity map.

\begin{lemma}
If an automorphism $\left( F_{1},F_{2},F_{3}\right) $ admits an elementary
reduction, then so does $\left( F_{1},F_{2},F_{3}\right) \circ L$ for every $%
L\in GL_{3}\left( \Bbb{C}\right) $.
\end{lemma}

\begin{proof}
Without loss of generality we can assume that $\left(
F_{1},F_{2},F_{3}\right) $ admits an elementary reduction of the form $%
\left( F_{1}-G\left( F_{2},F_{3}\right) ,F_{2},F_{3}\right) .$ It is easy to
see that $\left( F_{1}\circ L-G\left( F_{2}\circ L,F_{3}\circ L\right)
,F_{2}\circ L,F_{3}\circ L\right) =\left( F_{1}-G\left( F_{2},F_{3}\right)
,F_{2},F_{3}\right) \circ L$ is an elementary reduction of $\left(
F_{1},F_{2},F_{3}\right) \circ L\allowbreak =\left( F_{1}\circ L,F_{2}\circ
L,F_{3}\circ L\right) .$
\end{proof}

We also have the following obvious lemma.

\begin{lemma}
\label{lem_mdeg_linear_change}For every mapping $F:\Bbb{C}^{n}\rightarrow 
\Bbb{C}^{n}$ and every $L\in GL_{n}\left( \Bbb{C}\right) $ we have 
\begin{equation*}
\limfunc{mdeg}\left( F\circ L\right) =\limfunc{mdeg}F.
\end{equation*}
\end{lemma}

Combining the above two lemmas we obtain the following result.

\begin{theorem}
For every sequence of positive integers $\left( d_{1},\ldots ,d_{n}\right)
\neq \left( 1,\ldots ,1\right) ,$ if there is a tame automorphism $F:\Bbb{C}%
^{n}\rightarrow \Bbb{C}^{n}$ such that $F$ admits an elementary reduction, $%
F(0,\ldots ,0)=(0,\ldots ,0)$ and $\limfunc{mdeg}F=\left( d_{1},\ldots
,d_{n}\right) ,\,$then there is also a tame automorphism $\widetilde{F}:\Bbb{%
C}^{n}\rightarrow \Bbb{C}^{n}$ such that $\widetilde{F}$ admits an
elementary reduction, $\limfunc{mdeg}\widetilde{F}=\left( d_{1},\ldots
,d_{n}\right) $, $\widetilde{F}\left( 0,\ldots ,0\right) =\left( 0,\ldots
,0\right) $ and the linear part of $\widetilde{F}\,$is equal to $\limfunc{id}%
_{\Bbb{C}^{n}}.$
\end{theorem}

\begin{proof}
Let $L$ be the linear part of $F.$ Since $F\in \limfunc{Aut}(\Bbb{C}^{n}),$
we have $L\in GL_{n}(\Bbb{C}).$ The linear part of $F\circ L^{-1}$ is equal
to $\limfunc{id}_{\Bbb{C}^{n}}.$ We also have $\left( F\circ L^{-1}\right)
(0,\ldots ,0)=F(0,\ldots ,0)=(0,\ldots ,0).$
\end{proof}

\subsection{Relationship between the degree of the Poisson bracket and the
number of variables}

The main result of this section is Lemma \ref{lem_deg_fg_2_x1x2} below. We
start with the following

\begin{lemma}
\label{lem_fg_f_involve}\label{lem_deg_poiss_1}Let $f,g\in \Bbb{C}%
[X_{1},\ldots ,X_{n}]$ be such that 
\begin{equation*}
f=X_{1}+f_{2}+\cdots +f_{l},\qquad g=X_{2}+g_{2}+\cdots g_{m},
\end{equation*}
where $f_{i},g_{i}$ are homogeneous forms of degree $i.$ If $\deg \left[
f,g\right] =2$ and $f$ does not involve $X_{i},$ where $i>2,$ then $g$ does
not involve $X_{i}$ either.
\end{lemma}

\begin{proof}
The assumption $\deg \left[ f,g\right] =2$ implies that for all $1\leq
k<l\leq n$ we have 
\begin{equation*}
\deg Jac^{X_{k}X_{l}}\left( f,g\right) \leq 0.
\end{equation*}
In particular, 
\begin{equation*}
\deg Jac^{X_{1}X_{i}}\left( f,g\right) \leq 0,
\end{equation*}
but 
\begin{equation*}
Jac^{X_{1}X_{i}}\left( f,g\right) =\frac{\partial f}{\partial X_{1}}\frac{%
\partial g}{\partial X_{i}}-\frac{\partial f}{\partial X_{i}}\frac{\partial g%
}{\partial X_{1}}=\frac{\partial f}{\partial X_{1}}\frac{\partial g}{%
\partial X_{i}}.
\end{equation*}
Thus $\deg \frac{\partial g}{\partial X_{i}}\leq 0.$ In other words if $g$
involves $X_{i}$ then $X_{i}$ occurs in the linear part of $g.$ But this
contradicts the assumptions.
\end{proof}

Now we are in a position to prove the following lemma that is one of the
main ingredients in proving, for instance, that $\left( 5,6,9\right) \notin 
\limfunc{mdeg}\left( \limfunc{Tame}\left( \Bbb{C}^{3}\right) \right) .$

\begin{lemma}
\label{lem_deg_poiss_2}\label{lem_deg_fg_2_x1x2}Let $f,g\in \Bbb{C}%
[X_{1},\ldots ,X_{n}]$ be such that 
\begin{equation*}
f=X_{1}+f_{2}+\cdots +f_{l},\qquad g=X_{2}+g_{2}+\cdots +g_{m},
\end{equation*}
where $f_{i},g_{i}$ are homogeneous forms of degree $i.$ If $\deg \left[
f,g\right] =2,$ then $f,g\in \Bbb{C}[X_{1},X_{2}].$
\end{lemma}

\begin{proof}
Without loss of generality we can assume that $l\leq m.$ Let $i>2$ be
arbitrary. Let us notice that 
\begin{equation*}
\left[ Jac^{X_{1}X_{i}}\left( f,g\right) \right] _{1}=Jac^{X_{1}X_{i}}\left(
X_{1},g_{2}\right) +Jac^{X_{1}X_{i}}\left( f_{2},X_{2}\right) =\frac{%
\partial g_{2}}{\partial X_{i}}
\end{equation*}
and 
\begin{equation*}
\left[ Jac^{X_{2}X_{i}}\left( f,g\right) \right] _{1}=Jac^{X_{2}X_{i}}\left(
X_{1},g_{2}\right) +Jac^{X_{2}X_{i}}\left( f_{2},X_{2}\right) =-\frac{%
\partial f_{2}}{\partial X_{i}},
\end{equation*}
where $\left[ Jac^{X_{k}X_{l}}\left( f,g\right) \right] _{d}$ is the
homogeneous part of degree $d$ of $Jac^{X_{k}X_{l}}\left( f,g\right) .$ But
the assumption $\deg \left[ f,g\right] =2$ means in particular that $\left[
Jac^{X_{1}X_{i}}\left( f,g\right) \right] _{1}=0$ and $\left[
Jac^{X_{2}X_{i}}\left( f,g\right) \right] _{1}=0.$ Thus we obtain 
\begin{equation*}
\frac{\partial g_{2}}{\partial X_{i}}=0,\qquad \frac{\partial f_{2}}{%
\partial X_{i}}=0,
\end{equation*}
and so $f_{2},g_{2}$ do not involve $X_{i}.$ It follows that 
\begin{eqnarray*}
\left[ Jac^{X_{1}X_{i}}\left( f,g\right) \right] _{2}
&=&Jac^{X_{1}X_{i}}\left( X_{1},g_{3}\right) +Jac^{X_{1}X_{i}}\left(
f_{2},g_{2}\right) +Jac^{X_{1}X_{i}}\left( f_{3},X_{2}\right) \\
&=&Jac^{X_{1}X_{i}}\left( X_{1},g_{3}\right) =\frac{\partial g_{3}}{\partial
X_{i}}
\end{eqnarray*}
and 
\begin{eqnarray*}
\left[ Jac^{X_{2}X_{i}}\left( f,g\right) \right] _{2}
&=&Jac^{X_{2}X_{i}}\left( X_{1},g_{3}\right) +Jac^{X_{2}x_{i}}\left(
f_{2},g_{2}\right) +Jac^{X_{2}X_{i}}\left( f_{3},X_{2}\right) \\
&=&Jac^{X_{2}X_{i}}\left( f_{3},X_{2}\right) =-\frac{\partial f_{3}}{%
\partial X_{i}}.
\end{eqnarray*}
Since $\deg \left[ f,g\right] =2$ implies $\left[ Jac^{x_{1}x_{i}}\left(
f,g\right) \right] _{2}=0$ and $\left[ Jac^{x_{2}x_{i}}\left( f,g\right)
\right] _{2}=0,$ we see that 
\begin{equation*}
\frac{\partial g_{3}}{\partial X_{i}}=0,\qquad \frac{\partial f_{3}}{%
\partial X_{i}}=0,
\end{equation*}
and so $f_{3},g_{3}$ do not involve $X_{i}.$

Proceeding inductively, when we know that $f_{2},\ldots
,f_{l-1},g_{2},\ldots ,g_{l-1}$ do not involve $X_{i},$ we obtain 
\begin{eqnarray*}
\left[ Jac^{X_{1}X_{i}}\left( f,g\right) \right] _{n-1}
&=&Jac^{X_{1}X_{i}}\left( X_{1},g_{n}\right) +\cdots +Jac^{X_{1}X_{i}}\left(
f_{n},X_{2}\right) \\
&=&Jac^{X_{1}X_{i}}\left( X_{1},g_{n}\right) =\frac{\partial g_{n}}{\partial
X_{i}}
\end{eqnarray*}
and 
\begin{eqnarray*}
\left[ Jac^{X_{2}X_{i}}\left( f,g\right) \right] _{n-1}
&=&Jac^{X_{2}X_{i}}\left( X_{1},g_{n}\right) +\cdots +Jac^{X_{2}X_{i}}\left(
f_{n},X_{2}\right) \\
&=&Jac^{X_{2}X_{i}}\left( f_{n},X_{2}\right) =-\frac{\partial f_{n}}{%
\partial X_{i}}.
\end{eqnarray*}
By the assumption $\deg \left[ f,g\right] =2,$ as before we find that $f_{n}$
and $g_{n}$ do not involve $X_{i}.$ Therefore $f$ does not involve $X_{i}.$
To deduce that $g$ does not involve $X_{i}$ either, we can use Lemma \ref
{lem_fg_f_involve}.
\end{proof}

By similar arguments one can prove the following

\begin{theorem}
\label{tw_fg_f_involve_x1x2}Let $f,g\in \Bbb{C}[X_{1},\ldots ,X_{n}]$ be
such that 
\begin{equation*}
f=X_{1}+f_{2}+\cdots +f_{l},\qquad g=X_{2}+g_{2}+\cdots +g_{m},
\end{equation*}
where $f_{i},g_{i}$ are homogeneous forms of degree $i.$ If $\deg \left[
f,g\right] =d\leq \min \left\{ k,m\right\} ,d\geq 2,$ and $f_{i},g_{i}$ for $%
i=1,\ldots ,d-1$ do not involve $X_{l},$ where $l>2,$ then $f$ and $g$ do
not involve $X_{l}.$
\end{theorem}

The results of Lemma \ref{lem_deg_fg_2_x1x2} and Theorem \ref
{tw_fg_f_involve_x1x2} can be generalized as follows.

\begin{theorem}
Let $f,g\in \Bbb{C}[X_{1},\ldots ,X_{n}]$ be such that 
\begin{equation*}
f=f_{1}+f_{2}+\cdots +f_{l},\qquad g=g_{1}+g_{2}+\cdots g_{m},
\end{equation*}
where $f_{i},g_{i}$ are homogeneous forms of degree $i.$ If $f_{1},g_{1}$
are linearly independent, $\deg \left[ f,g\right] =d\leq \min \left\{
l,m\right\} ,d\geq 2,$ and $f_{i},g_{i},$ for $i=1,\ldots ,d-1,$ do not
involve $X_{r},$ where $r>2,$ then $f$ and $g$ do not involve $X_{r}.$
\end{theorem}

\begin{proof}
Let $l_{3},\ldots ,l_{n-1}\in \Bbb{C}[X_{1},\ldots ,X_{r-1},X_{r+1},\ldots
,X_{n}]$ be linear forms such that $f_{1},g_{1},l_{3},\ldots ,l_{n-1}$ are
linearly independent. Then $f_{1},g_{1},l_{3},\ldots ,l_{n-1},X_{r}$ are
also linearly independent. Let $L=(f_{1},g_{1},l_{3},\ldots ,l_{n-1},X_{r}):%
\Bbb{C}^{n}\rightarrow \Bbb{C}^{n}.$ Of course $L,L^{-1}\in GL_{n}(\Bbb{C}),$
and by Lemma \ref{lem_degree_linear_change}, $\deg [f\circ L^{-1},g\circ
L^{-1}]=\deg [f,g]=d.$ One can also check that $\left( f\circ L^{-1}\right)
_{1}=X_{1},\left( g\circ L^{-1}\right) _{1}=X_{2}$ and that $\left( f\circ
L^{-1}\right) _{i},\left( g\circ L^{-1}\right) _{i},$ for $i=1,\ldots ,d-1,$
do not involve $X_{l}.$ Thus by Theorem \ref{tw_fg_f_involve_x1x2}, $f\circ
L^{-1},g\circ L^{-1}$ do not involve $X_{l}$ either. And one can easily
check that the same is true for $f=\left( f\circ L^{-1}\right) \circ L$ and $%
g=\left( g\circ L^{-1}\right) \circ L.$
\end{proof}

\newpage 

\section{The case $\left( p_{1},p_{2},d_{3}\right) $ and its generalization%
\label{section_ppd_and_general}}

\subsection{The case $\left( p_{1},p_{2},d_{3}\right) $}

Here we investigate the set 
\begin{equation*}
\left\{ \left( p_{1},p_{2},d_{3}\right) :3\leq p_{1}<p_{2}\leq d_{3},\text{ }%
p_{1},p_{2}\text{ prime numbers }\right\} \cap \limfunc{mdeg}\left( \limfunc{%
Tame}\left( \Bbb{C}^{3}\right) \right) .
\end{equation*}
The complete description of this set is given in the following theorem.

\begin{theorem}
\label{tw_p1_p2_d3} \textit{(\cite{Karas2}, Thm. 1.1) }Let $d_{3}\geq
p_{2}>p_{1}\geq 3$ be positive integers. If $p_{1}$ and $p_{2}$ are prime
numbers, then $(p_{1},p_{2},d_{3})\in \limfunc{mdeg}(\limfunc{Tame}(\Bbb{C}%
^{3}))$ if and only if $d_{3}\in p_{1}\Bbb{N}+p_{2}\Bbb{N}.$
\end{theorem}

\begin{proof}
If $d_{3}\in p_{1}\Bbb{N}+p_{2}\Bbb{N},$ then by Proposition \ref
{prop_sum_d_i}, there exists a tame automorphism $F\in \limfunc{Tame}\left( 
\Bbb{C}^{3}\right) $ such that $\limfunc{mdeg}F=\left(
p_{1},p_{2},d_{3}\right) .$ Thus in order to prove the theorem we must only
show that if $d_{3}\notin p_{1}\Bbb{N}+p_{2}\Bbb{N},$ then there is no tame
automorphism $F:\Bbb{C}^{3}\rightarrow \Bbb{C}^{3}$ with $\limfunc{mdeg}%
F=\left( p_{1},p_{2},d_{3}\right) .$ Thus up to the end of the proof we
assume that $d_{3}\notin p_{1}\Bbb{N}+p_{2}\Bbb{N}.$

Assume, to the contrary, that there are tame automorphisms $F$ of $\Bbb{C}%
^{3}$ such that $\limfunc{mdeg}F=(p_{1},p_{2},d_{3}).$ By Theorem \ref
{tw_reduc_type_4}, we only need to show that all such automorphisms do not
admit an elementary reduction and reduction of type III. Since $p_{2}>3$ is
a prime number, $2\nmid p_{2}.$ Hence by Remark \ref{remakrk_red_type_3}, no
automorphism $F$ of $\Bbb{C}^{3}$ with $\limfunc{mdeg}F=\left(
p_{1},p_{2},d_{3}\right) $ admits a reduction of type III.

Assume, to the contrary, that there is an an automorphism $F=\left(
F_{1},F_{2},F_{3}\right) $ of $\Bbb{C}^{3}$ with $\limfunc{mdeg}F=\left(
p_{1},p_{2},d_{3}\right) $ that admits an elemenatry reduction. Notice that,
by Theorem \ref{tw_sywester}, 
\begin{equation}
d_{3}<(p_{1}-1)(p_{2}-1).  \label{d_3 male}
\end{equation}

Assume that 
\begin{equation*}
(F_{1},F_{2},F_{3}-g(F_{1},F_{2})),
\end{equation*}
where $g\in \Bbb{C}\left[ X,Y\right] ,$ is an elementary reduction of $%
\left( F_{1},F_{2},F_{3}\right) .$ Then we have $\deg g(F_{1},F_{2})=\deg
F_{3}=d_{3}.$ But, by Proposition \ref{prop_deg_g_fg}, 
\begin{equation*}
\deg g(F_{1},F_{2})\geq q(p_{1}p_{2}-p_{1}-p_{2}+\deg [F_{1},F_{2}])+rp_{2},
\end{equation*}
where $\deg _{Y}g(X,Y)=qp_{1}+r$ with $0\leq r<p_{1}.\,$ Since $F_{1},F_{2}$
are algebraically independent, $\deg [F_{1},F_{2}]\geq 2$ and so 
\begin{equation*}
p_{1}p_{2}-p_{1}-p_{2}+\deg [F_{1},F_{2}]\geq
p_{1}p_{2}-p_{1}-p_{2}+2>(p_{1}-1)(p_{2}-1).
\end{equation*}
This and (\ref{d_3 male}) imply that $q=0,$ and that: 
\begin{equation*}
g(X,Y)=\sum_{i=0}^{p_{1}-1}g_{i}(X)Y^{i}.
\end{equation*}
Since $\func{lcm}(p_{1},p_{2})=p_{1}p_{2},$ the sets 
\begin{equation*}
p_{1}\Bbb{N},p_{2}+p_{1}\Bbb{N},\ldots ,(p_{1}-1)p_{2}+p_{1}\Bbb{N}
\end{equation*}
are pairwise disjoint. This yields: 
\begin{equation*}
\deg \left( \sum_{i=0}^{p_{1}-1}g_{i}(F_{1})F_{2}^{i}\right) =\underset{%
i=0,\ldots ,p_{1}-1}{\max }\left( \deg F_{1}\deg g_{i}+i\deg F_{2}\right) ,
\end{equation*}
and so 
\begin{equation*}
d_{3}=\deg g\left( F_{1},F_{2}\right) \in \bigcup_{r=0}^{p_{1}-1}\left(
rp_{2}+p_{1}\Bbb{N}\right) \subset p_{1}\Bbb{N}+p_{2}\Bbb{N},
\end{equation*}
a contradiction.

Now, assume that 
\begin{equation*}
(F_{1},F_{2}-g(F_{1},F_{3}),F_{3}),
\end{equation*}
where $g\in \Bbb{C}\left[ X,Y\right] ,$ is an elementary reduction of $%
F=\left( F_{1},F_{2},F_{3}\right) .$ Since $d_{3}\notin p_{1}\Bbb{N}+p_{2}%
\Bbb{N},p_{1}\nmid d_{3}$ and $\gcd (p_{1},d_{3})=1.$ This means, by
Proposition \ref{prop_deg_g_fg}, that 
\begin{equation*}
\deg g(F_{1},F_{3})\geq q(p_{1}d_{3}-d_{3}-p_{1}+\deg [F_{1},F_{3}])+rd_{3},
\end{equation*}
where $\deg _{Y}g(X,Y)=qp_{1}+r$ with $0\leq r<p_{1}.$ Since $%
p_{1}d_{3}-d_{3}-p_{1}+\deg [F_{1},F_{3}]\geq p_{1}d_{3}-2d_{3}\geq
d_{3}>p_{2}$ and since we want to have $\deg g(F_{1},F_{3})=p_{2},$ we
conclude that $q=r=0.$ This means that $g(X,Y)=g(X),$ and then $p_{2}=\deg
g(F_{1})\in p_{1}\Bbb{N},$ a contradiction.

Finally, if we assume that $(F_{1}-g(F_{2},F_{3}),F_{2},F_{3}),$ where $g\in 
\Bbb{C}\left[ X,Y\right] ,$ is an elementary reduction of $\left(
F_{1},F_{2},F_{3}\right) ,$ then in the same way as in the previous case we
obtain a contradiction.
\end{proof}

\begin{corollary}
The following equality holds true 
\begin{eqnarray*}
&&\left\{ \left( p_{1},p_{2},d_{3}\right) :3\leq p_{1}<p_{2}\leq d_{3},\text{
}p_{1},p_{2}\text{ prime numbers }\right\} \cap \limfunc{mdeg}\left( 
\limfunc{Tame}\left( \Bbb{C}^{3}\right) \right) \\
&=&\left\{ \left( p_{1},p_{2},d_{3}\right) :3\leq p_{1}<p_{2}\leq d_{3},%
\text{ }p_{1},p_{2}\text{ prime numbers, }d_{3}\in p_{1}\Bbb{N}+p_{2}\Bbb{N}%
\right\}
\end{eqnarray*}
\end{corollary}

\subsection{Some consequences}

\begin{theorem}
\textit{(\cite{Karas2}, Thm. 3.1) }Let $p_{2}>3$ be a prime number and $%
d_{3}\geq p_{2}$ be an integer. Then $(3,p_{2},d_{3})\in \limfunc{mdeg}(%
\limfunc{Tame}(\Bbb{C}^{3}))$ if and only if $d_{3}\notin \{2p_{2}-3k\ |\
k=1,\ldots ,\left[ \frac{p_{2}}{3}\right] \}.$
\end{theorem}

\begin{proof}
Since $p_{2}>3$ is a prime number, $p_{2}\equiv r\ (\func{mod}3)$ for some $%
r\in \{1,2\}.$ It is easy to see that if $d_{3}\geq p_{2}$ and $d_{3}\equiv
0\ (\func{mod}3)$ or $d_{3}\equiv r\ (\func{mod}3),$ then $d_{3}\in 3\Bbb{N}%
+p_{2}\Bbb{N}.$ Thus, by Theorem \ref{tw_sywester}, 
\begin{equation*}
2(p_{2}-1)-1\neq 0,r\ (\func{mod}3).
\end{equation*}
Take any $d_{3}$ such that $p_{2}\leq d_{3}\leq 2p_{2}-3$ and $d_{3}\neq
0,r\ (\func{mod}3).$ Since $d_{3}\leq 2p_{2}-3$ and $d_{3}\equiv 2p_{2}-3\ (%
\func{mod}3),$ we see that $d_{3}\notin 3\Bbb{N}+p_{2}\Bbb{N},$ because
otherwise we would have $2p_{2}-3\in 3\Bbb{N}+p_{2}\Bbb{N},$ contrary
toTheorem \ref{tw_sywester}. Thus 
\begin{eqnarray*}
\{d_{3} &\in &\Bbb{N\ }|\ d_{3}\geq p_{2},d_{3}\notin 3\Bbb{N}+p_{2}\Bbb{N}%
\}= \\
&=&\{d_{3}\in \Bbb{N\ }|\ p_{2}\leq d_{3}\leq 2p_{2}-3,d_{3}\equiv 2p_{2}-3%
\Bbb{(}\func{mod}3\Bbb{)}\} \\
&=&\{2p_{2}-3k\ |\ k=1,\ldots ,\left[ \tfrac{p_{2}}{3}\right] \}
\end{eqnarray*}
\end{proof}

\begin{theorem}
\textit{(\cite{Karas2}, Thm. 3.2) }(a) If $d_{3}\geq 7,$ then $%
(5,7,d_{3})\in \limfunc{mdeg}(\limfunc{Tame}(\Bbb{C}^{3}))$ if and only if 
\begin{equation*}
d_{3}\neq 8,9,11,13,16,18,23.
\end{equation*}
\newline
(b) If $d_{3}\geq 11,$ then $(5,11,d_{3})\in \limfunc{mdeg}(\limfunc{Tame}(%
\Bbb{C}^{3}))$ if and only if 
\begin{equation*}
d_{3}\neq 12,13,14,17,18,19,23,24,28,29,34,39.
\end{equation*}
\newline
(c) If $d_{3}\geq 13,$ then $(5,13,d_{3})\in \limfunc{mdeg}(\limfunc{Tame}(%
\Bbb{C}^{3}))$ if and only if 
\begin{equation*}
d_{3}\neq 14,16,17,19,21,22,24,27,29,32,34,37,42,47.
\end{equation*}
\newline
(d) If $d_{3}\geq 11,$ then $(7,11,d_{3})\in \limfunc{mdeg}(\limfunc{Tame}(%
\Bbb{C}^{3}))$ if and only if 
\begin{equation*}
d_{3}\neq 12,13,15,16,17,19,20,23,24,26,27,30,31,34,37,38,41,45,48,52,59.
\end{equation*}
\end{theorem}

\begin{proof}
This is a consequence of Theorems \ref{tw_sywester} and \ref{tw_p1_p2_d3}.
For example to prove (a), by Theorems \ref{tw_sywester} and \ref{tw_p1_p2_d3}
we only have to check which numbers among $7,8,\ldots ,23=(5-1)(7-1)-1$ are
elements of the set $5\Bbb{N}+7\Bbb{N}.$
\end{proof}

\subsection{Generalization}

Here we make a generalization of Theorem \ref{tw_p1_p2_d3}.

\begin{theorem}
\textit{(\cite{Karas Zygad}, Thm. 2.1) }\label{tw_gcd_d1,d2_1}\label%
{tw_odd_odd_gcd_1}Let $d_{3}\geq d_{2}>d_{1}\geq 3$ be positive integers. If 
$d_{1}$ and $d_{2}$ are odd numbers such that $\gcd \left(
d_{1},d_{2}\right) =1$, then $(d_{1},d_{2},d_{3})\in \limfunc{mdeg}(\limfunc{%
Tame}(\Bbb{C}^{3}))$ if and only if $d_{3}\in d_{1}\Bbb{N}+d_{2}\Bbb{N}.$
\end{theorem}

\begin{proof}
The proof is a modification of the proof of Theorem \ref{tw_p1_p2_d3}. As
before, if we assume that $d_{3}\in d_{1}\Bbb{N}+d_{2}\Bbb{N},$ then by
Proposition \ref{prop_sum_d_i}, there is a tame automorphism $F$ of $\Bbb{C}%
^{3}$ such that $\limfunc{mdeg}F=\left( d_{1},d_{2},d_{3}\right) .$

Moreover, as in the proof of Theorem \ref{tw_p1_p2_d3}, we only need to show
that no automorphism $F$ of $\Bbb{C}^{3}$ with $\limfunc{mdeg}F=\left(
d_{1},d_{2},d_{3}\right) $ admits an elementary reduction, when $d_{3}\notin
d_{1}\Bbb{N}+d_{2}\Bbb{N}.$ Assume, to the contrary, that $d_{3}\notin d_{1}%
\Bbb{N}+d_{2}\Bbb{N}$ and that $F=\left( F_{1},F_{2},F_{3}\right) $ is an
authomorphism of $\Bbb{C}^{3}$ with $\limfunc{mdeg}F=\left(
d_{1},d_{2},d_{3}\right) $ that admits an elementary reduction.

If we assume that $(F_{1},F_{2},F_{3}-g(F_{1},F_{2})),$where $g\in \Bbb{C}%
[X,Y],$ is an elementary reduction of $\left( F_{1},F_{2},F_{3}\right) ,$
then we can proceed exactly in the same way as in the proof of Theorem \ref
{tw_p1_p2_d3}.

Only in the case of an elementary reduction of the form $%
(F_{1},F_{2}-g(F_{1},F_{3}),F_{3})$ and $(F_{1}-g(F_{2},F_{3}),F_{2},F_{3})$
we must modify the arguments.

Assume that 
\begin{equation*}
(F_{1},F_{2}-g(F_{1},F_{3}),F_{3}),
\end{equation*}
where $g\in \Bbb{C}[X,X],$ is an elementary reduction of $\left(
F_{1},F_{2},F_{3}\right) .$ Since $d_{3}\notin d_{1}\Bbb{N}+d_{2}\Bbb{N},$
we have $d_{1}\nmid d_{3}.$ This follows that 
\begin{equation*}
p=\frac{d_{1}}{\gcd \left( d_{1},d_{3}\right) }>1.
\end{equation*}
Since $d_{1}\,$is odd, we also have $p\neq 2.$ Thus by Proposition \ref
{prop_deg_g_fg}, 
\begin{equation*}
\deg g(F_{1},F_{3})\geq q(pd_{3}-d_{3}-d_{1}+\deg [F_{1},F_{3}])+rd_{3},
\end{equation*}
where $\deg _{Y}g(X,Y)=qp+r$ with $0\leq r<p.$ Since $p\geq 3,$ we see that $%
pd_{3}-d_{3}-d_{1}+\deg [F_{1},F_{3}]\geq 2d_{3}-d_{1}+2>d_{3}.$ Since we
want to have $\deg g(F_{1},F_{3})=d_{2},$ it follows that $q=r=0,$ and then $%
g(X,Y)=g(X).$ This means that $d_{2}=\deg g\left( F_{1}\right) \in d_{1}\Bbb{%
N},$ contradicting to $\gcd \left( d_{1},d_{2}\right) =1$ and $1<d_{1}.$

Finally, if we assume that $(F_{1}-g(F_{2},F_{3}),F_{2},F_{3}),$ where $g\in 
\Bbb{C}\left[ X,Y\right] ,$is an elementary reduction of $\left(
F_{1},F_{2},F_{3}\right) ,$ then in the same way as in the previous case we
obtain a contradiction.
\end{proof}

\subsection{The set $\limfunc{mdeg}\left( \limfunc{Aut}\left( \Bbb{C}%
^{3}\right) \right) \backslash \limfunc{mdeg}\left( \limfunc{Tame}\left( 
\Bbb{C}^{3}\right) \right) $}

In this paragraph we say a few words about a relation between $\limfunc{mdeg}%
\left( \limfunc{Tame}\left( \Bbb{C}^{3}\right) \right) $ and $\limfunc{mdeg}%
\left( \limfunc{Aut}\left( \Bbb{C}^{3}\right) \right) .$ The obvious
relation is 
\begin{equation*}
\limfunc{mdeg}\left( \limfunc{Tame}\left( \Bbb{C}^{3}\right) \right) \subset 
\limfunc{mdeg}\left( \limfunc{Aut}\left( \Bbb{C}^{3}\right) \right)
\end{equation*}
and, more generally, 
\begin{equation*}
\limfunc{mdeg}\left( \limfunc{Tame}\left( \Bbb{C}^{n}\right) \right) \subset 
\limfunc{mdeg}\left( \limfunc{Aut}\left( \Bbb{C}^{n}\right) \right) .
\end{equation*}
The question is, whether the set $\limfunc{mdeg}\left( \limfunc{Tame}\left( 
\Bbb{C}^{n}\right) \right) $ is a proper subset of $\limfunc{mdeg}\left( 
\limfunc{Aut}\left( \Bbb{C}^{n}\right) \right) .$ In dimension two the
answer is negative due to Jung \cite{Jung} and van der Kulk \cite{Kulk}.
Namely we have 
\begin{equation*}
\limfunc{mdeg}\left( \limfunc{Tame}\left( \Bbb{C}^{2}\right) \right) =%
\limfunc{mdeg}\left( \limfunc{Aut}\left( \Bbb{C}^{2}\right) \right) =\left\{
\left( d_{1},d_{2}\right) :d_{1}|d_{2}\text{ or }d_{2}|d_{1}\right\} .
\end{equation*}
Let us notice that the result of Shestakov and Umirbaev \cite{sh umb2} about
wildness of the Nagata's example does not imply the positive answer in
dimension three. The problem is that the Nagata's example is of multidegree $%
\left( 5,3,1\right) \in \limfunc{mdeg}\left( \limfunc{Tame}\left( \Bbb{C}%
^{3}\right) \right) .$ In spite of that, the answer is positive. We will
show it in this subsection. Actually we show not only that the difference $%
\limfunc{mdeg}\left( \limfunc{Aut}\left( \Bbb{C}^{3}\right) \right)
\backslash \limfunc{mdeg}\left( \limfunc{Tame}\left( \Bbb{C}^{3}\right)
\right) $ is not empty, but also that this set has infinitely many elements.

Let 
\begin{equation*}
N:\Bbb{C}^{3}\ni (x,y,z)\mapsto
(x+2y(y^{2}+zx)-z(y^{2}+zx)^{2},y-z(y^{2}+zx),z)\in \Bbb{C}^{3}
\end{equation*}
be the Nagata's example and let 
\begin{equation*}
T:\Bbb{C}^{3}\ni (x,y,z)\mapsto (z,y,x)\in \Bbb{C}^{3}.
\end{equation*}
We start with the following lemma.

\begin{lemma}
\textit{(\cite{Karas Zygad}, Lem. 3.1) }For all $n\in \Bbb{N}$ we have $%
\limfunc{mdeg}(T\circ N)^{n}=(4n-3,4n-1,4n+1).$
\end{lemma}

\begin{proof}
We have $T\circ N(x,y,z)=(z,y-z(y^{2}+zx),x+2y(y^{2}+zx)-z(y^{2}+zx)^{2}),$
so the above equality is true for $n=1.$ Let $(f_{n},g_{n},h_{n})=\left(
T\circ N\right) ^{n}$ for $f_{n},g_{n},h_{n}\in \Bbb{C[}X,Y,Z].$ One can see
that $g_{1}^{2}+h_{1}f_{1}=Y^{2}+ZX,$ and by induction that $%
g_{n}^{2}+h_{n}f_{n}=Y^{2}+ZX$ for any $n\in \Bbb{N}^{*}.\Bbb{\,}$Thus 
\begin{gather*}
(f_{n+1},g_{n+1},h_{n+1})=\left( T\circ N\right) \circ \left(
f_{n},g_{n},h_{n}\right) \\
=\left( h_{n},g_{n}-h_{n}\left( g_{n}^{2}+h_{n}f_{n}\right)
,f_{n}+2h_{n}\left( g_{n}^{2}+h_{n}f_{n}\right) -h_{n}\left(
g_{n}^{2}+h_{n}f_{n}\right) ^{2}\right) \\
=\left( h_{n},g_{n}-h_{n}\left( Y^{2}+ZX\right) ,f_{n}+2h_{n}\left(
Y^{2}+ZX\right) -h_{n}\left( Y^{2}+ZX\right) ^{2}\right) .
\end{gather*}
So if we assume that $\limfunc{mdeg}(f_{n},g_{n},h_{n})=(4n-3,4n-1,4n+1),$
we obtain \allowbreak $\limfunc{mdeg}(f_{n+1},g_{n+1},h_{n+1})=(4n+1,\left(
4n+1\right) +2,\left( 4n+1\right) +2\cdot 2)=(4(n+1)-3,4(n+1)-1,4(n+1)+1).$
\end{proof}

By the above lemma and Theorem \ref{tw_gcd_d1,d2_1} we obtain the following

\begin{theorem}
\textit{(\cite{Karas Zygad}, Thm. 3.2) }For every $n\in \Bbb{N}$ the
automorphism $(T\circ N)^{n}$ is wild.
\end{theorem}

\begin{proof}
For $n=1$ this is the result of Shestakov and Umirbaev \cite{sh umb1,sh umb2}%
. So we can assume that $n\geq 2.$ The numbers $4n-3,4n-1$ are odd and $\gcd
(4n-3,4n-1)=\gcd (4n-3,2)=1.$ Since $4n-3>2,$ we see that $4n+1\notin (4n-3)%
\Bbb{N}+(4n-1)\Bbb{N}.$ Then, by Theorem \ref{tw_gcd_d1,d2_1}, $%
(4n-3,4n-1,4n+1)\notin \limfunc{mdeg}(\limfunc{Tame}(\Bbb{C}^{3}))$ for $%
n>1. $ This proves that $(T\circ N)^{n}$ is not a tame automorphism.
\end{proof}

Let us notice that in the proof of the above theorem we have also proved
that 
\begin{equation*}
\left\{ \left( 4n-3,4n-1,4n+1\right) :n\in \Bbb{N},n\geq 2\right\} \subset 
\limfunc{mdeg}(\limfunc{Aut}(\Bbb{C}^{3}))\backslash \limfunc{mdeg}(\limfunc{%
Tame}(\Bbb{C}^{3})).
\end{equation*}
This gives the following result.

\begin{theorem}
\textit{(\cite{Karas Zygad}, Thm. 1.1) }The set $\limfunc{mdeg}(\limfunc{Aut}%
(\Bbb{C}^{3}))\backslash \limfunc{mdeg}(\limfunc{Tame}(\Bbb{C}^{3}))$ is
infinite.
\end{theorem}

\section{The case $\left( 3,d_{2},d_{3}\right) \label{section_3dd}$}

In this section we give a complete description of the set 
\begin{equation*}
\left\{ \left( 3,d_{2},d_{3}\right) |3\leq d_{2}\leq d_{3}\right\} \cap 
\limfunc{mdeg}(\limfunc{Tame}(\Bbb{C}^{3})).
\end{equation*}
This description is given by the following

\begin{theorem}
\label{tw_3_d2_d3} \textit{(\cite{Karas3}, Thm. 1.1) }If $3\leq d_{2}\leq
d_{3},$ then $(3,d_{2},d_{3})\in \limfunc{mdeg}(\limfunc{Tame}(\Bbb{C}^{3}))$
if and only if $3|d_{2}$ or $d_{3}\in 3\Bbb{N}+d_{2}\Bbb{N}.$
\end{theorem}

\begin{proof}
By Corollary \ref{prop_sum_d_i}, if $3|d_{2}$ or $d_{3}\in 3\Bbb{N}+d_{2}%
\Bbb{N},$ there exists a tame automorphism $F:\Bbb{C}^{3}\rightarrow \Bbb{C}%
^{3}$ such that $\limfunc{mdeg}F=(3,d_{2},d_{3}).$ Thus in order to prove
Theorem \ref{tw_3_d2_d3} it is enough to show that if $3\nmid d_{2}$ and $%
d_{3}\notin 3\Bbb{N}+d_{2}\Bbb{N},$ then there is no tame automorphism of $%
\Bbb{C}^{3}$ with multidegree $(3,d_{2},d_{3}).$ So from now we will assume
that $3\nmid d_{2}$ and $d_{3}\notin 3\Bbb{N}+d_{2}\Bbb{N}.$

Since $3\nmid d_{2},$ we have $\gcd (3,d_{2})=1.$ Hence Theorem \ref
{tw_sywester} implies that for all $k\geq (3-1)(d_{2}-1)=2d_{2}-2$ we have $%
k\in 3\Bbb{N}+d_{2}\Bbb{N}.$ Thus, since $d_{3}\notin 3\Bbb{N}+d_{2}\Bbb{N},$
we have 
\begin{equation}
d_{3}<2d_{2}-2.  \label{row_sylwester}
\end{equation}

By Theorem \ref{tw_reduc_type_4} it is enough to show that all automorphisms 
$F$ of $\Bbb{C}^{3}$ with $\limfunc{mdeg}F=\left( 3,d_{2},d_{3}\right) $ do
not admit an elementary reduction and reduction of type III. Notice also
that, since $d_{1}=3$ and $d_{2}$ can be an even number, we can not use
Remark \ref{remakrk_red_type_3} to obtain that all automorphisms $F$ of $%
\Bbb{C}^{3}$ with $\limfunc{mdeg}F=\left( 3,d_{2},d_{3}\right) $ do not
admit reduction of type III.

Assume that an automorphism $F=\left( F_{1},F_{2},F_{3}\right) :\Bbb{C}%
^{3}\rightarrow \Bbb{C}^{3}$ with $\limfunc{mdeg}F=\left(
3,d_{2},d_{3}\right) $ admits a reduction of type III. Then by Definition 
\ref{def_type_III_IV} there is a permutation $\sigma \,$of $\{1,2,3\}$ and $%
n\in \Bbb{N}^{*}$ such that $\deg F_{\sigma (1)}=2n,$ and either: 
\begin{equation}
\deg F_{\sigma (2)}=3n,\qquad n<\deg F_{\sigma (3)}\leq \frac{3}{2}n
\label{row_type_III}
\end{equation}
or 
\begin{equation}
\frac{5}{2}n<\deg F_{\sigma (2)}\leq 3n,\qquad \deg F_{\sigma (3)}=\frac{3}{2%
}n.  \label{row_type_IV}
\end{equation}
Since $\frac{3}{2}n<2n<\min \{\frac{5}{2}n,3n\},$ we have $d_{2}=2n$ and
either: 
\begin{equation*}
d_{3}=3n,\qquad n<3\leq \frac{3}{2}n
\end{equation*}
or 
\begin{equation*}
\frac{5}{2}n<d_{3}\leq 3n,\qquad 3=\frac{3}{2}n.
\end{equation*}
Thus $n=2$ and then $5<d_{3}\leq 6.\,$From the last inequalities we obtain $%
d_{3}=6.$ This contradicts $d_{3}\notin 3\Bbb{N}+d_{2}\Bbb{N}.$

Now, assume that $(F_{1},F_{2},F_{3}-g(F_{1},F_{2})),$ where $g\in \Bbb{C}%
[X,Y],$ is an elementary reduction of $(F_{1},F_{2},F_{3}).$ Hence we have $%
\deg g(F_{1},F_{2})=\deg F_{3}=d_{3}.$ Since $\gcd (3,d_{2})=1,$ by
Proposition \ref{prop_deg_g_fg} we have 
\begin{equation*}
\deg g(F_{1},F_{2})\geq q(3d_{2}-d_{2}-3+\deg [F_{1},F_{2}])+rd_{2},
\end{equation*}
where $\deg _{Y}g(X,Y)=3q+r$ with $0\leq r<3.\,$ Since $F_{1},F_{2}$ are
algebraically independent, $\deg [F_{1},F_{2}]\geq 2$ and so $%
3d_{2}-d_{2}-3+\deg [F_{1},F_{2}]\geq 2d_{2}-1.$ Then (\ref{row_sylwester})
implies $q=0.$ Also by (\ref{row_sylwester}) we must have $r<2.$ Thus $%
g(X,Y)=g_{0}(X)+g_{1}(X)Y.$ Since $3\Bbb{N\cap (}d_{2}+3\Bbb{N})=\emptyset ,$
we deduce that $\deg g(F_{1},F_{2})\in 3\Bbb{N\cup (}d_{2}+3\Bbb{N})\subset 3%
\Bbb{N}+d_{2}\Bbb{N},$ contrary to assumption.

Now, assume that $(F_{1},F_{2}-g(F_{1},F_{3}),F_{3}),$ where $g\in \Bbb{C}%
[X,Y],$ there is an elementary reduction of $(F_{1},F_{2},F_{3}).$ Then $%
\deg g(F_{1},F_{3})=d_{2}.$ Since $d_{3}\notin 3\Bbb{N}+d_{2}\Bbb{N},$ it
follows that $\gcd (3,d_{3})=1.$ Then by Proposition \ref{prop_deg_g_fg} we
have 
\begin{equation*}
\deg g(F_{1},F_{3})\geq q(3d_{3}-d_{3}-3+\deg [F_{1},F_{3}])+rd_{3},
\end{equation*}
where $\deg _{Y}g(X,Y)=3q+r$ with $0\leq r<3.$ Since $3d_{3}-d_{3}-3+\deg
[F_{1},F_{3}]\geq 2d_{3}-1>d_{2},$ we infer that $q=0.$ Since also $%
d_{3}>d_{2}$ (because $d_{3}\geq d_{2}$ and $d_{3}\notin 3\Bbb{N}+d_{2}\Bbb{%
N)},$ we see that $r=0.$ Thus $g(X,Y)=g(X),$ and $\deg g(F_{1},F_{3})=\deg
g(F_{1})\in 3\Bbb{N},$ a contradiction.

Finally, assume that $(F_{1}-g(F_{2},F_{3}),F_{2},F_{3}),$ where $g\in \Bbb{C%
}[x,y],$ there is an elementary reduction of $(F_{1},F_{2},F_{3}).$ Then $%
\deg g(F_{2},F_{3})=3.$ Let 
\begin{equation*}
p=\frac{d_{2}}{\gcd (d_{2},d_{3})}.
\end{equation*}
Since $d_{3}\notin 3\Bbb{N}+d_{2}\Bbb{N},$ we obtain $d_{2}\nmid d_{3},$ and
hence $p>1.$ By Proposition \ref{prop_deg_g_fg} we have 
\begin{equation*}
\deg g(F_{2},F_{3})\geq q(pd_{3}-d_{2}-d_{3}+\deg [F_{1},F_{3}])+rd_{3},
\end{equation*}
where $\deg _{Y}g(X,Y)=qp+r$ with $0\leq r<p.$ Since $d_{3}>3,$ it follows
that $r=0.$ Consider the case $p\geq 3.$ Then $pd_{3}-d_{2}-d_{3}+\deg
[F_{1},F_{3}]\geq d_{3}+\deg [F_{1},F_{3}]>3.$ Thus we must have $q=0.$
Hence $g(X,Y)=g(X),$ and $3=\deg g(F_{2},F_{3})=\deg g(F_{2})\in d_{2}\Bbb{N}%
.$ This contradicts $d_{2}\neq 3$ (we have assumed that $3\nmid d_{2}$). 

Consider now the case $p=2.$ Since $p=2,$ we have,  for some $n\in \Bbb{N},$ 
$d_{2}=2n$ and $d_{3}=ns,$ where  $s\geq 3$ is odd. Since also $d_{2}>3,$ it
follows that $n\geq 2.$ This means that $d_{3}-d_{2}\geq 2,$ and $%
2d_{3}-d_{3}-d_{2}+\deg [F_{1},F_{3}]=d_{3}-d_{2}+\deg [F_{1},F_{3}]\geq 4>3.
$ Thus, also in this case we have $q=0.$ As before this leads to a
contradiction.
\end{proof}

\begin{corollary}
The following equality holds true 
\begin{gather*}
\left\{ \left( 3,d_{2},d_{3}\right) |3\leq d_{2}\leq d_{3}\right\} \cap 
\limfunc{mdeg}(\limfunc{Tame}(\Bbb{C}^{3}))= \\
=\left\{ \left( 3,d_{2},d_{3}\right) |3\leq d_{2}\leq d_{3},3|d_{2}\text{ or 
}d_{3}\in 3\Bbb{N}+d_{2}\Bbb{N}\right\} .
\end{gather*}
\end{corollary}

\section{The case $\left( 4,d_{2},d_{3}\right) \label{section_4dd}$}

In this section we give partial description of the set 
\begin{equation*}
\left\{ \left( 4,d_{2},d_{3}\right) |4\leq d_{2}\leq d_{3}\right\} \cap 
\limfunc{mdeg}(\limfunc{Tame}(\Bbb{C}^{3})).
\end{equation*}
This description will be given separately for four cases: (I) $d_{2},d_{3}$
are both even numbers, (II) $d_{2},d_{3}$ are both odd numbers, (III) $d_{2}$
is even and $d_{3}$ is odd, (IV) $d_{2}$ is odd and $d_{3}$ is even.

\subsection{The case $\left( 4,\limfunc{even},\limfunc{even}\right) $}

This is the easiest case, summarised as follows.

\begin{theorem}
\label{tw_4pp}For all even numbers $d_{3}\geq d_{2}\geq 4$, $\left(
4,d_{2},d_{3}\right) \in \limfunc{mdeg}\left( \limfunc{Tame}\left( \Bbb{C}%
^{3}\right) \right) .$
\end{theorem}

\begin{proof}
Since all numbers $4,d_{2},d_{3}$ are even, we have $\gcd \left(
4,d_{2},d_{3}\right) \in \left\{ 2,4\right\} .$ Thus $\frac{4}{\gcd \left(
4,d_{2},d_{3}\right) }\leq 2$ and we can use Theorem \ref{tw_gcd_male}.
\end{proof}

\subsection{The case $\left( 4,\limfunc{odd},\limfunc{odd}\right) $}

In this section we give entire description of the set 
\begin{equation*}
\left\{ \left( 4,d_{2},d_{3}\right) \ |\ 4\leq d_{2}\leq d_{3},\quad
d_{2},d_{3}\in 2\Bbb{N}+1\right\} \cap \limfunc{mdeg}\left( \limfunc{Tame}%
\left( \Bbb{C}^{3}\right) \right) .
\end{equation*}
We will show the following

\begin{theorem}
\label{tw_4_odd_odd}Let $d_{3}\geq d_{2}\geq 4$ be odd numbers. Then $\left(
4,d_{2},d_{3}\right) \in \limfunc{mdeg}\left( \limfunc{Tame}\left( \Bbb{C}%
^{3}\right) \right) $ if and only if $d_{3}\in 4\Bbb{N}+d_{2}\Bbb{N}.$
\end{theorem}

\begin{proof}
Because of Proposition \ref{prop_sum_d_i} it is enough to show that if $%
d_{3}\notin 4\Bbb{N}+d_{2}\Bbb{N},$ then $\left( 4,d_{2},d_{3}\right) \notin 
\limfunc{mdeg}\left( \limfunc{Tame}\left( \Bbb{C}^{3}\right) \right) .$
Thus, up to the end of the proof we assume that $d_{3}\notin 4\Bbb{N}+d_{2}%
\Bbb{N}.$ Since $d_{2}$ is odd, we have $\gcd \left( 4,d_{2}\right) =1,$ and
so, by Theorem \ref{tw_sywester}, 
\begin{equation}
d_{3}<\left( 4-1\right) \left( d_{2}-1\right) =3d_{2}-3.
\label{row_4_odd_odd_0}
\end{equation}

By Remark \ref{remakrk_red_type_3} and Theorem \ref{tw_reduc_type_4}, it is
enough to show that no automorphism $F$ of $\Bbb{C}^{3}$ with $\limfunc{mdeg}%
F=\left( 4,d_{2},d_{3}\right) $ admits an elementary reduction.

Assume, to the contrary, that $\left( F_{1},F_{2},F_{3}-g\left(
F_{1},F_{2}\right) \right) ,$ where $g\in \Bbb{C}\left[ X,Y\right] ,$ is an
elementary reduction of an automorphism $F=\left( F_{1},F_{2},F_{3}\right) :%
\Bbb{C}^{3}\rightarrow \Bbb{C}^{3}$ with $\limfunc{mdeg}F=\left(
4,d_{2},d_{3}\right) .$ Then
\begin{equation}
\deg g\left( F_{1},F_{2}\right) =d_{3}.  \label{row_4_odd_odd_1}
\end{equation}
By Proposition \ref{prop_deg_g_fg}, 
\begin{equation}
\deg g\left( F_{1},F_{2}\right) \geq q\left( pd_{2}-d_{2}-4+\deg \left[
F_{1},F_{2}\right] \right) +rd_{2},  \label{row_4_odd_odd_2}
\end{equation}
where $\deg _{Y}g\left( X,Y\right) =pq+r,0\leq r<p$ and $p=\frac{4}{\gcd
\left( 4,d_{2}\right) }=4.$ Since $pd_{2}-d_{2}-4+\deg \left[
F_{1},F_{2}\right] =3d_{2}-4+\deg \left[ F_{1},F_{2}\right] \geq 3d_{2}-2,$
by (\ref{row_4_odd_odd_0}), (\ref{row_4_odd_odd_1}) and (\ref
{row_4_odd_odd_2}) we have $q=0$ and $r\leq 2.$ This means that $g\left(
X,Y\right) $ is of the form 
\begin{equation*}
g\left( X,Y\right) =g_{0}\left( X\right) +g_{1}\left( X\right) Y+g_{2}\left(
X\right) Y^{2}.
\end{equation*}
Since the sets $4\Bbb{N},d_{2}+4\Bbb{N}$ and $2d_{2}+4\Bbb{N}$ are pairwise
disjoint (because $\func{lcm}\left( 4,d_{2}\right) =4d_{2}$), it follows
that 
\begin{equation*}
d_{3}=\deg g\left( F_{1},F_{2}\right) \in 4\Bbb{N}\cup \left( d_{2}+4\Bbb{N}%
\right) \cup \left( 2d_{2}+4\Bbb{N}\right) .
\end{equation*}
This contradicts $d_{3}\notin 4\Bbb{N}+d_{2}\Bbb{N}.$

Now, assume that $\left( F_{1},F_{2}-g\left( F_{1},F_{3}\right)
,F_{3}\right) ,$ where $g\in \Bbb{C}\left[ x,y\right] ,$ is an elementary
reduction of an automorphism $F=\left( F_{1},F_{2},F_{3}\right) :\Bbb{C}%
^{3}\rightarrow \Bbb{C}^{3}$ with $\limfunc{mdeg}F=\left(
4,d_{2},d_{3}\right) .$ Then 
\begin{equation}
\deg g\left( F_{1},F_{3}\right) =d_{2}.  \label{row_4_odd_odd_3}
\end{equation}
But, by Proposition \ref{prop_deg_g_fg} we have 
\begin{equation}
\deg g\left( F_{1},F_{3}\right) \geq q\left( pd_{3}-d_{3}-4+\deg \left[
F_{1},F_{3}\right] \right) +rd_{3},  \label{row_4_odd_odd_4}
\end{equation}
where $\deg _{Y}g\left( X,Y\right) =pq+r,0\leq r<p$ and $p=\frac{4}{\gcd
\left( 4,d_{2}\right) }=4.$ Since $d_{3}>d_{2}>4,$ we see that $%
pd_{3}-d_{3}-4+\deg \left[ F_{1},F_{3}\right] >2d_{3}>d_{2}.$ Hence by (\ref
{row_4_odd_odd_3}) and (\ref{row_4_odd_odd_4}), $q=r=0.$ This means that $%
g\left( X,Y\right) =g\left( X\right) $ and so $d_{2}=\deg g\left(
F_{1},F_{3}\right) =\deg g\left( F_{1}\right) \in 4\Bbb{N}.$ This
contradicts the assumption that $d_{2}$ is an odd number.

Finally, assume that $\left( F_{1}-g\left( F_{2},F_{3}\right)
,F_{2},F_{3}\right) ,$ where $g\in \Bbb{C}\left[ x,y\right] ,$ is an
elementary reduction of an automorphism $F=\left( F_{1},F_{2},F_{3}\right) :%
\Bbb{C}^{3}\rightarrow \Bbb{C}^{3}$ with $\limfunc{mdeg}F=\left(
4,d_{2},d_{3}\right) .$ Then 
\begin{equation}
\deg g\left( F_{2},F_{3}\right) =4.  \label{row_4_odd_odd_5}
\end{equation}
By Proposition \ref{prop_deg_g_fg}, 
\begin{equation}
\deg g\left( F_{1},F_{3}\right) \geq q\left( pd_{3}-d_{3}-d_{2}+\deg \left[
F_{2},F_{3}\right] \right) +rd_{3},  \label{row_4_odd_odd_6}
\end{equation}
where $\deg _{Y}g\left( X,Y\right) =pq+r,0\leq r<p$ and $p=\frac{d_{2}}{\gcd
\left( d_{2},d_{3}\right) }.$ Since $d_{3}>4,$ by (\ref{row_4_odd_odd_5})
and (\ref{row_4_odd_odd_6}) we have $r=0.$ Since also $2\nmid d_{2}$ and $%
d_{2}\nmid d_{3}$ (because $d_{3}\notin 4\Bbb{N}+d_{2}\Bbb{N}$), we conclude
that $p=\frac{d_{2}}{\gcd \left( d_{2},d_{3}\right) }\geq 3$ and $%
pd_{3}-d_{3}-d_{2}+\deg \left[ F_{2},F_{3}\right] >d_{3}>4.$ Thus $q=0.$
Then we obtain a contradiction as in the previous case.
\end{proof}

\begin{corollary}
The following equality holds true: 
\begin{gather*}
\left\{ \left( 4,d_{2},d_{3}\right) :4\leq d_{2}\leq d_{3},\ d_{2},d_{3}\in 2%
\Bbb{N}+1\right\} \cap \limfunc{mdeg}\left( \limfunc{Tame}\left( \Bbb{C}%
^{3}\right) \right) = \\
=\left\{ \left( 4,d_{2},d_{3}\right) :4\leq d_{2}\leq d_{3},\ d_{2},d_{3}\in
2\Bbb{N}+1,\ d_{3}\in 4\Bbb{N}+d_{2}\Bbb{N}\right\} .
\end{gather*}
\end{corollary}

\subsection{The case $\left( 4,\limfunc{even},\limfunc{odd}\right) $}

We start this subsection with the following two examples (or rather two
series of examples).

\begin{example}
Since 
\begin{eqnarray*}
\left( X+Z^{4}\right) ^{3} &=&Z^{12}+3XZ^{8}+3X^{2}Z^{4}+X^{3}, \\
\left( Y+Z^{6}\right) ^{2} &=&Z^{12}+2YZ^{6}+Y^{2},
\end{eqnarray*}
we see that 
\begin{equation*}
\deg \left[ \left( Y+Z^{6}\right) ^{2}-\left( X+Z^{4}\right) ^{3}\right] =9.
\end{equation*}
Thus, for any $k\in \Bbb{N},$%
\begin{equation*}
\deg \left[ \left( Y+Z^{6}\right) ^{2}-\left( X+Z^{4}\right) ^{3}\right]
\left( X+Z^{4}\right) ^{k}=9+4k.
\end{equation*}
This means that 
\begin{equation*}
\limfunc{mdeg}\left( F_{2}\circ F_{1}\right) =\left( 4,6,9+4k\right) ,
\end{equation*}
where 
\begin{eqnarray*}
F_{1}\left( x,y,z\right)  &=&\left( x+z^{4},y+z^{6},z\right) , \\
F_{2}\left( u,v,w\right)  &=&\left( u,v,w+\left( v^{2}-u^{3}\right)
u^{k}\right) .
\end{eqnarray*}
\end{example}

\begin{example}
Since 
\begin{eqnarray*}
\left( X+Z^{4}\right) ^{3} &=&Z^{12}+3XZ^{8}+3X^{2}Z^{4}+X^{3}, \\
\left( Y+\tfrac{3}{2}XZ^{2}+Z^{6}\right) ^{2} &=&Z^{12}+3XZ^{8}+2YZ^{6}+%
\tfrac{9}{4}X^{2}Z^{4}+3YXZ^{2}+Y^{2},
\end{eqnarray*}
it follows that 
\begin{equation*}
\deg \left[ \left( Y+\tfrac{3}{2}XZ^{2}+Z^{6}\right) ^{2}-\left(
X+Z^{4}\right) ^{3}\right] =7,
\end{equation*}
and 
\begin{equation*}
\deg \left[ \left( Y+\tfrac{3}{2}XZ^{2}+Z^{6}\right) ^{2}-\left(
X+Z^{4}\right) ^{3}\right] \left( X+Z^{4}\right) ^{k}=7+4k.
\end{equation*}
Thus we have 
\begin{equation*}
\limfunc{mdeg}\left( F_{2}\circ F_{1}\right) =\left( 4,6,7+4k\right) ,
\end{equation*}
where 
\begin{eqnarray*}
F_{1}\left( x,y,z\right)  &=&\left( x+z^{4},y+\tfrac{3}{2}%
xz^{2}+z^{6},z\right) , \\
F_{2}\left( u,v,w\right)  &=&\left( u,v,w+\left( v^{2}-u^{3}\right)
u^{k}\right) .
\end{eqnarray*}
\end{example}

Combining above examples and Theorem \ref{tw_4pp} we obtain the following

\begin{proposition}
For any integer $d_{3}\geq 6$ we have $\left( 4,6,d_{3}\right) \in \limfunc{%
mdeg}\left( \limfunc{Tame}\left( \Bbb{C}^{3}\right) \right) .$
\end{proposition}

In the same manner one can prove the following

\begin{proposition}
For any integer $d_{3}\geq 10$ we have $\left( 4,10,d_{3}\right) \in 
\limfunc{mdeg}\left( \limfunc{Tame}\left( \Bbb{C}^{3}\right) \right) .$
\end{proposition}

Using Corollary \ref{cor_male_d1} we obtain

\begin{proposition}
For $k=1,2,\ldots $ and an integer $d_{3}\geq 4k$ we have $\left(
4,4k,d_{3}\right) \in \limfunc{mdeg}\left( \limfunc{Tame}\left( \Bbb{C}%
^{3}\right) \right) .$
\end{proposition}

The next proposition gives partial information about multidegrees of the
form $\left( 4,4k+2,d_{3}\right) ,$ where $k=3,4,\ldots $ and $d_{3}\geq
4k+2.$

\begin{proposition}
\label{prop_4_4k+2_d3}For integers $k\geq 3$ and $d_{3}\geq 5k+1$ we have $%
\left( 4,4k+2,d_{3}\right) \in \limfunc{mdeg}\left( \limfunc{Tame}\left( 
\Bbb{C}^{3}\right) \right) .$
\end{proposition}

\begin{proof}
Let us notice that 
\begin{equation*}
\left( X+Z^{4}\right) ^{2k+1}=\sum_{l=0}^{2k+1}\binom{2k+1}{l}%
X^{l}Z^{8k+4-4l}
\end{equation*}
and 
\begin{eqnarray*}
\left( Y+Z^{r}+\sum_{l=0}^{k}a_{l}X^{l}Z^{4k+2-4l}\right) ^{2}
&=&Y^{2}+2YZ^{r}+Z^{2r}+2Y\sum_{l=0}^{k}a_{l}X^{l}Z^{4k+2-4l} \\
&&+2Z^{r}\sum_{l=0}^{k}a_{l}X^{l}Z^{4k+2-4l} \\
&&+\sum_{s=0}^{2k}\left( \sum_{l+m=s,\ l,m\in \left\{ 0,\ldots ,k\right\}
}a_{l}a_{m}\right) X^{s}Z^{8k+4-4s}.
\end{eqnarray*}

We will consider the cases $r=k-1,k,k+1$ and $k+2.$ Thus we have:
\begin{eqnarray*}
\deg 2YZ^{r} &\leq &k+3<5k+1, \\
\deg Z^{2r} &\leq &2k+4<5k+1, \\
\deg 2Y\sum_{l=0}^{k}a_{l}X^{l}Z^{4k+2-4l} &\leq &4k+3<5k+1, \\
\deg 2Z^{r}\sum_{l=2}^{k}a_{l}X^{l}Z^{4k+2-4l} &\leq &5k-2<5k+1.
\end{eqnarray*}
This means that the only summands of the polynomial 
\begin{equation}
\left( X+Z^{4}\right) ^{2k+1}-\left(
Y+Z^{r}+\sum_{l=0}^{k}a_{l}X^{l}Z^{4k+2-4l}\right) ^{2}
\label{row_prop_4_4k+2}
\end{equation}
of degree greater or equal to $5k+1$ are: 
\begin{eqnarray*}
&&\left( 1-a_{0}^{2}\right) Z^{8k+4}, \\
&&\left[ \binom{2k+1}{1}-2a_{0}a_{1}\right] XZ^{8k}, \\
&&\left[ \binom{2k+1}{2}-\left( 2a_{0}a_{2}+a_{1}^{2}\right) \right]
X^{2}Z^{8k-4}, \\
&&\vdots  \\
&&\left[ \binom{2k+1}{k}-\left( a_{0}a_{k}+a_{1}a_{k-1}+\cdots
+a_{k-1}a_{1}+a_{k}a_{0}\right) \right] X^{k}Z^{4k+4}, \\
&&2a_{0}z^{4k+2+r}
\end{eqnarray*}
and (only in the case $r=k+2$) 
\begin{equation*}
2a_{1}XZ^{4k-2+r}.
\end{equation*}
Since we can recursively solve the following system of equations (notice
that we can take $a_{0}=1$) 
\begin{eqnarray*}
1-a_{0}^{2} &=&0, \\
\binom{2k+1}{1}-2a_{0}a_{1} &=&0, \\
\binom{2k+1}{2}-\left( 2a_{0}a_{2}+a_{1}^{2}\right)  &=&0, \\
&&\vdots  \\
\binom{2k+1}{k}-\left( a_{0}a_{k}+a_{1}a_{k-1}+\cdots
+a_{k-1}a_{1}+a_{k}a_{0}\right)  &=&0,
\end{eqnarray*}
it follows that we can choose coefficients $a_{0},a_{1},\ldots ,a_{k}$ such
that the degree of the polynomial (\ref{row_prop_4_4k+2}) is equal to 
\begin{equation*}
\deg \left( 2a_{0}Z^{4k+2+r}\right) =4k+2+r.
\end{equation*}
Taking $r=k-1,k,k+1$ and $k+2$ we obtain polynomials of degree equal to $%
5k+1,5k+2,5k+3$ and $5k+4,$ respectively.

Now, it is easy to see that taking 
\begin{equation*}
F\left( x,y,z\right) =\left(
x+z^{4},y+z^{r}+\sum_{l=0}^{k}a_{l}x^{l}z^{4k+2-4l},z\right)
\end{equation*}
and 
\begin{equation*}
G\left( u,v,w\right) =\left( u,v,w+\left( u^{4k+1}-v^{2}\right) u^{q}\right)
,
\end{equation*}
where $q=0,1,\ldots ,$ we obtain that 
\begin{equation*}
\limfunc{mdeg}\left( G\circ F\right) =\left( 4,4k+2,4k+2+r+4q\right) .
\end{equation*}
Since for any $d_{3}\geq 5k+1$ we can find $r\in \left\{
k-1,k,k+1,k+2\right\} $ and $q\in \Bbb{N}$ such that $4k+2+r+4q=d_{3},$ the
result follows.
\end{proof}

\subsection{The case $\left( 4,\limfunc{odd},\limfunc{even}\right) $}

In this subsection we give almost complete description of the set 
\begin{equation*}
\left\{ \left( 4,d_{2},d_{3}\right) \ |\ 4\leq d_{2}\leq d_{3},\quad
d_{2}\in 2\Bbb{N}+1,d_{3}\in 2\Bbb{N}\right\} \cap \limfunc{mdeg}\left( 
\limfunc{Tame}\left( \Bbb{C}^{3}\right) \right) .
\end{equation*}
Namely we have the following result.

\begin{theorem}
\label{tw_4_odd_even}If $d_{2}\geq 5$ is an odd number and $d_{3}\geq d_{2}$
is an even number such that $d_{3}-d_{2}\neq 1,$ then $\left(
4,d_{2},d_{3}\right) \in \limfunc{mdeg}\left( \limfunc{Tame}\left( \Bbb{C}%
^{3}\right) \right) $ if and only if $d_{3}\in 4\Bbb{N}+d_{2}\Bbb{N}.$
\end{theorem}

\begin{proof}
If $d_{3}\in 4\Bbb{N}+d_{2}\Bbb{N},$ then by Proposition \ref{prop_sum_d_i}, 
$\left( 4,d_{2},d_{3}\right) \in \limfunc{mdeg}\left( \limfunc{Tame}\left( 
\Bbb{C}^{3}\right) \right) .$ Thus up to the end of the proof we will assume
that $d_{3}\notin 4\Bbb{N}+d_{2}\Bbb{N}.$ Our purpose is to show that there
is no tame automorphism $F$ of $\Bbb{C}^{3}$ with $\limfunc{mdeg}F=\left(
4,d_{2},d_{3}\right) .$ Since $d_{2}$ is an odd number, by Remark \ref
{remakrk_red_type_3} and Theorem \ref{tw_reduc_type_4} it is enough to show
that all automorphisms $F:\Bbb{C}^{3}\rightarrow \Bbb{C}^{3}$ with $\limfunc{%
mdeg}F=\left( 4,d_{2},d_{3}\right) $ do not admit an elementary reduction.
Thus the rest of the proof is the inspection of the three cases of
elementary reducibility.

Assume that $\left( F_{1},F_{2},F_{3}-g\left( F_{1},F_{2}\right) \right) ,$
where $g\in \Bbb{C}\left[ X,Y\right] ,$ is an elementary reduction of an
automorphism $F=\left( F_{1},F_{2},F_{3}\right) :\Bbb{C}^{3}\rightarrow \Bbb{%
C}^{3}$ with $\limfunc{mdeg}F=\left( 4,d_{2},d_{3}\right) .$ Thus 
\begin{equation*}
\deg g\left( F_{1},F_{2}\right) =d_{3},
\end{equation*}
and by Proposition \ref{prop_deg_g_fg}, 
\begin{equation*}
\deg g\left( F_{1},F_{2}\right) \geq q\left( pd_{2}-d_{2}-4+\deg \left[
F_{1},F_{2}\right] \right) +rd_{2},
\end{equation*}
where $\deg _{Y}g\left( X,Y\right) =pq+r,0\leq r<p$ and $p=\frac{4}{\gcd
\left( 4,d_{2}\right) }=4.$ Since $d_{3}\notin 4\Bbb{N}+d_{2}\Bbb{N}$ and $%
\gcd \left( 4,d_{2}\right) =1,$ we have (similarly as in the proof of
Theorem \ref{tw_4_odd_odd}) 
\begin{equation}
d_{3}<3d_{2}-3.  \label{row_4_odd_even_0}
\end{equation}
Thus we can repeat the arguments from the corresponding case in the proof of
Theorem \ref{tw_4_odd_odd} to obtain a contradiction.

Now, assume that $\left( F_{1},F_{2}-g\left( F_{1},F_{3}\right)
,F_{3}\right) ,$ for some $g\in \Bbb{C}\left[ X,Y\right] ,$ is an elementary
reduction of an authomorphism $F=\left( F_{1},F_{2},F_{3}\right) $ of $\Bbb{C%
}^{3}$ with $\limfunc{mdeg}F=\left( 4,d_{2},d_{3}\right) .$ Then
\begin{equation}
\deg g\left( F_{1},F_{3}\right) =d_{2},  \label{row_4_odd_even_1}
\end{equation}
and by Proposition \ref{prop_deg_g_fg}, 
\begin{equation}
\deg g\left( F_{1},F_{3}\right) \geq q\left( pd_{3}-d_{3}-4+\deg \left[
F_{1},F_{2}\right] \right) +rd_{3},  \label{row_4_odd_even_2}
\end{equation}
where $\deg _{Y}g\left( X,Y\right) =pq+r,0\leq r<p$ and $p=\frac{4}{\gcd
\left( 4,d_{3}\right) }=2$ (because $d_{3}$ is an even number and $%
d_{3}\notin 4\Bbb{N}+d_{2}\Bbb{N}$). Thus $pd_{3}-d_{3}-4+\deg \left[
F_{1},F_{2}\right] \geq d_{3}-2.$ But by the assumptions $d_{3}-d_{2}\geq 0$
is an odd number different from $1.$ So $d_{2}\leq d_{3}-3,$ and then $%
pd_{2}-d_{2}-4+\deg \left[ F_{1},F_{2}\right] >d_{2}.$ Consequently, by (\ref
{row_4_odd_even_1}) and (\ref{row_4_odd_even_2}), $q=0.$ Since also $r=0$
(because $d_{3}>d_{2}$), we see that $g\left( X,Y\right) =g\left( X\right) ,$
and so 
\begin{equation*}
d_{2}=\deg g\left( F_{1},F_{3}\right) =\deg g\left( F_{1}\right) \in 4\Bbb{N}%
.
\end{equation*}
This contradicts the assumption that $d_{2}$ is an odd number.

In the last case we can repeat the arguments from the corresponding case in
the proof of Theorem \ref{tw_4_odd_odd}
\end{proof}

\begin{corollary}
If $d_{2}\geq 5$ is an odd number such that $d_{2}\equiv 3\left( \func{mod}%
4\right) ,$ and $d_{3}\geq d_{2}$ is an even number, then $\left(
4,d_{2},d_{3}\right) \in \limfunc{mdeg}\left( \limfunc{Tame}\left( \Bbb{C}%
^{3}\right) \right) $ if and only if $d_{3}\in 4\Bbb{N}+d_{2}\Bbb{N}.$
\end{corollary}

\begin{proof}
Let us notice, that if $d_{3}-d_{2}=1,$ then $4|d_{3}.$ Thus $d_{3}\in 4\Bbb{%
N}+d_{2}\Bbb{N}$ and by Proposition \ref{prop_sum_d_i} $\left(
4,d_{2},d_{3}\right) \in \limfunc{mdeg}\left( \limfunc{Tame}\left( \Bbb{C}%
^{3}\right) \right) .$ In the case $d_{3}-d_{2}>1,$ we can use Theorem \ref
{tw_4_odd_even}.
\end{proof}

By the above corollary, we know that to complete the picture of the set 
\begin{equation*}
\left\{ \left( 4,d_{2},d_{3}\right) \ |\ 4\leq d_{2}\leq d_{3},\quad
d_{2}\in 2\Bbb{N}+1,d_{3}\in 2\Bbb{N}\right\} \cap \limfunc{mdeg}\left( 
\limfunc{Tame}\left( \Bbb{C}^{3}\right) \right) 
\end{equation*}
it is enough to consider the triples of the form 
\begin{equation*}
\left( 4,4k+1,4k+2\right) \qquad \text{for }k=1,2,\ldots 
\end{equation*}
Moreover, using the arguments from the proof of Theorem \ref{tw_4_odd_even},
one can show the following

\begin{proposition}
Let $k\in \Bbb{N}^{*}.$ If there exists a tame authomorphism $\widetilde{F}$
of $\Bbb{C}^{3}$ with $\limfunc{mdeg}\widetilde{F}=\left( 4,4k+1,4k+2\right)
,$ then there is also a tame automorphism $F=\left( F_{1},F_{2},F_{3}\right) 
$ of $\Bbb{C}^{3}$ with $\limfunc{mdeg}F=\left( 4,4k+1,4k+2\right) $ that
admits an elementary reduction $\left( F_{1},F_{2}-g\left(
F_{1},F_{3}\right) ,F_{3}\right) ,$ for some $g\in \Bbb{C}\left[ X,Y\right] .
$ Moreover, for such automorphism $F$ we have $\deg \left[
F_{1},F_{3}\right] \leq 3.$
\end{proposition}

Using arguments from the proof of Theorem \ref{lem_5_6_9} one can also show
that $\deg \left[ F_{1},F_{3}\right] =3$ when $k<25.$

\section{The case $\left( p,d_{2},d_{3}\right) $ and $\left(
5,d_{2},d_{3}\right) \label{section_pdd_5dd}$}

\subsection{The general case.}

Now we make, in some sense, a generalization of the results of section 'The
case $\left( 3,d_{2},d_{3}\right) $'. This generalization is not complete,
because in the presented picture there are some holes. The first, general
result is the following

\begin{theorem}
\label{tw_p_d2_d3}Let $2\leq p\leq d_{2}\leq d_{3}$ be a sequence of
positive integers, and let $p$ be a prime number. If\newline
(1) $\frac{d_{3}}{d_{2}}\neq \frac{3}{2}$ or\newline
(2) $\frac{d_{3}}{d_{2}}=\frac{3}{2}$ and $\frac{d_{2}}{2}>p-2,$\newline
then $\left( p,d_{2},d_{3}\right) \in \limfunc{mdeg}\left( \limfunc{Tame}%
\left( \Bbb{C}^{3}\right) \right) $ if and only if $p|d_{2}$ or $d_{3}\in p%
\Bbb{N}+d_{2}\Bbb{N}.$
\end{theorem}

\begin{proof}
By Corollary \ref{prop_sum_d_i}, if $p|d_{2}$ or $d_{3}\in p\Bbb{N}+d_{2}%
\Bbb{N},$ then there exists a tame automorphism $F:\Bbb{C}^{3}\rightarrow 
\Bbb{C}^{3}$ such that $\limfunc{mdeg}F=\left( p,d_{2},d_{3}\right) .$ Thus
in order to prove Theorem \ref{tw_p_d2_d3} it is enough to show that if $%
p\nmid d_{2}$ and $d_{3}\notin p\Bbb{N}+d_{2}\Bbb{N}$ and (1) or (2) holds,
then $\left( p,d_{2},d_{3}\right) \notin \limfunc{mdeg}\left( \limfunc{Tame}%
\left( \Bbb{C}^{3}\right) \right) .$

So let us assume that $p\nmid d_{2},d_{3}\notin p\Bbb{N}+d_{2}\Bbb{N}$ and
(1) or (2) hold. In particular $p<d_{2}<d_{3}.$ By Theorems \ref{tw_3_d2_d3}
and \ref{cor_male_d1}, we can assume that $p>3.$ Indeed, for $p=2,$ by
Corollary \ref{cor_male_d1} we have $\left( 2,d_{2},d_{3}\right) \in 
\limfunc{mdeg}\left( \limfunc{Tame}\left( \Bbb{C}^{3}\right) \right) $ for
all integers $2\leq d_{2}\leq d_{3}.$ Also the condition $2|d_{2}$ or $%
d_{3}\in 2\Bbb{N}+d_{2}\Bbb{N}$ is satisfied for all integers $2\leq
d_{2}\leq d_{3}.$ For $p=3$ we simple use Theorem \ref{tw_3_d2_d3}. So up to
the end of the proof we will assume that $p>3.$ Thus by Theorem \ref
{tw_reduc_type_4} it is enough to show that no automorphism $F:\Bbb{C}%
^{3}\rightarrow \Bbb{C}^{3}$ with $\limfunc{mdeg}F=\left(
p,d_{2},d_{3}\right) $ admits an elementary reduction (notice that $3\nmid p$%
).

Assume, to the contrary, that there exists an automorphism $%
F=(F_{1},F_{2},F_{3})$ with $\limfunc{mdeg}F=\left( p,d_{2},d_{3}\right) $
that addmits an elementary reduction. Since $p\nmid d_{2},$ we have $\gcd
(p,d_{2})=1.$ So by Theorem \ref{tw_sywester} we have $k\in p\Bbb{N}+d_{2}%
\Bbb{N}$ for all $k\geq (p-1)(d_{2}-1)=pd_{2}-d_{2}-p+1.$ Thus 
\begin{equation}
d_{3}<pd_{2}-d_{2}-p+1,  \label{row_sylwester_pdd}
\end{equation}
since $d_{3}\notin p\Bbb{N}+d_{2}\Bbb{N}.$

Assume that
\begin{equation*}
(F_{1},F_{2},F_{3}-g(F_{1},F_{2})),
\end{equation*}
where $g\in \Bbb{C}[X,Y],$ is an elementary reduction of $%
(F_{1},F_{2},F_{3}).$ Hence we have $\deg g(F_{1},F_{2})=\deg F_{3}=d_{3}.$
Since $\gcd (p,d_{2})=1,$ we see that $\frac{p}{\gcd \left( p,d_{2}\right) }%
=p,$ and so by Proposition \ref{prop_deg_g_fg}, 
\begin{equation*}
\deg g(F_{1},F_{2})\geq q(pd_{2}-d_{2}-p+\deg [F_{1},F_{2}])+rd_{2},
\end{equation*}
where $\deg _{Y}g(X,Y)=pq+r$ with $0\leq r<p.\,$ Since $F_{1},F_{2}$ are
algebraically independent, $\deg [F_{1},F_{2}]\geq 2$ and $%
pd_{2}-d_{2}-p+\deg [F_{1},F_{2}]\geq pd_{2}-d_{2}-p+2.$ Then by (\ref
{row_sylwester_pdd}) follows that $q=0.$ Thus 
\begin{equation*}
g(X,Y)=\sum_{i=0}^{p-1}g_{i}(X)Y^{i}.
\end{equation*}
Since $\func{lcm}(p,d_{2})=pd_{2},$ the sets 
\begin{equation*}
p\Bbb{N},d_{2}+p\Bbb{N},\ldots ,(p-1)d_{2}+p\Bbb{N}
\end{equation*}
are pairwise disjoint. So 
\begin{equation*}
\deg \left( \sum_{i=0}^{p-1}g_{i}(F_{1})F_{2}^{i}\right) =\underset{%
i=0,\ldots ,p-1}{\max }\left( \deg F_{1}\deg g_{i}+i\deg F_{2}\right) 
\end{equation*}
and 
\begin{eqnarray*}
d_{3} &=&\deg g\left( F_{1},F_{2}\right)  \\
&=&\deg \left( \sum_{i=0}^{p-1}g_{i}(F_{1})F_{2}^{i}\right) \in
\bigcup_{r=0}^{p-1}\left( rd_{2}+p\Bbb{N}\right) \subset p\Bbb{N}+d_{2}\Bbb{N%
},
\end{eqnarray*}
a contradiction.

Now assume that
\begin{equation*}
(F_{1},F_{2}-g(F_{1},F_{3}),F_{3}),
\end{equation*}
where $g\in \Bbb{C}[X,Y],$ is an elementary reduction of $%
(F_{1},F_{2},F_{3}).$ Since $d_{3}\notin p\Bbb{N}+d_{2}\Bbb{N},$ we have $%
p\nmid d_{3}$ and $\gcd (p,d_{3})=1.$ Hence by Proposition \ref
{prop_deg_g_fg}, 
\begin{equation*}
\deg g(F_{1},F_{3})\geq q(pd_{3}-d_{3}-p+\deg [F_{1},F_{3}])+rd_{3},
\end{equation*}
where $\deg _{Y}g(X,Y)=qp+r$ with $0\leq r<p.$ Since $pd_{3}-d_{3}-p+\deg
[F_{1},F_{3}]\geq pd_{3}-2d_{3}\geq 3d_{3}>d_{2}$ and since we want to have $%
\deg g(F_{1},F_{3})=p_{2},$ we conclude that $q=r=0.$ This means that $%
g(X,Y)=g(X),$ and so 
\begin{equation*}
d_{2}=\deg g\left( F_{1},F_{2}\right) =\deg g\left( F_{1}\right) \in p\Bbb{%
N\subset }p\Bbb{N}+d_{2}\Bbb{N},
\end{equation*}
a contradiction.

Finally, assume that 
\begin{equation*}
(F_{1}-g(F_{2},F_{3}),F_{2},F_{3}),
\end{equation*}
where $g\in \Bbb{C}[X,Y],$ is an elementary reduction of $%
(F_{1},F_{2},F_{3}).$ Thus we have $\deg g(F_{2},F_{3})=p.$ Let 
\begin{equation*}
\widetilde{p}=\frac{d_{2}}{\gcd (d_{2},d_{3})}.
\end{equation*}
Since $d_{3}\notin p\Bbb{N}+d_{2}\Bbb{N},$ we see that $d_{2}\nmid d_{3},$
and so $\widetilde{p}>1.$ By Proposition \ref{prop_deg_g_fg}, 
\begin{equation*}
\deg g(F_{2},F_{3})\geq q(\widetilde{p}d_{3}-d_{2}-d_{3}+\deg
[F_{1},F_{3}])+rd_{3},
\end{equation*}
where $\deg _{Y}g(X,Y)=q\widetilde{p}+r$ with $0\leq r<\widetilde{p}.$ Since 
$d_{3}>p$ (because $d_{3}>d_{2}>p$), we see that $r=0.$ Consider the case $%
\widetilde{p}\geq 3.$ Then $\widetilde{p}d_{3}-d_{2}-d_{3}+\deg
[F_{1},F_{3}]\geq d_{3}+\deg [F_{1},F_{3}]>p.$ Thus we must have $q=0.$
Hence $g(X,Y)=g(X)$ and 
\begin{equation*}
p=\deg g(F_{2},F_{3})=\deg g(F_{2})\in d_{2}\Bbb{N}.
\end{equation*}
This contradicts $d_{2}\neq p$ (we have assummed that $p\nmid d_{2}$).

Now, consider the case $\widetilde{p}=2.$ Since $\widetilde{p}=2,$ we have,
for some $n\in \Bbb{N}^{*},$ $d_{2}=2n$ and $d_{3}=ns,$ where $s\geq 3$ is
odd. Consider first the case $s>3.$ Then 
\begin{eqnarray*}
2d_{3}-d_{3}-d_{2}+\deg [F_{1},F_{3}] &=&d_{3}-d_{2}+\deg [F_{1},F_{3}] \\
&=&\left( s-2\right) n+\deg [F_{1},F_{3}]>d_{2}>p.
\end{eqnarray*}
Thus we have $q=0.\,$As before this leads to a contradiction.

Now, consider the case $s=3.$ This is the case when we use the second
statement of the assumption (2). Since $d_{2}=2n$ and $d_{3}=3n,$ we see
that $\frac{d_{3}}{d_{2}}=\frac{3}{2}.$ Hence (1) is not satisfied. Thus,
the assumption (2) is satisfied and so $n=\frac{d_{2}}{2}>p-2.$ Hence 
\begin{eqnarray*}
2d_{3}-d_{3}-d_{2}+\deg [F_{1},F_{3}] &=&d_{3}-d_{2}+\deg [F_{1},F_{3}]\geq 
\\
&\geq &n+2>p.
\end{eqnarray*}
So, also in this case we have $q=0.$ As before this leads to a contradiction.
\end{proof}

For small prime numbers $p$ the above theorem gives, for example, the
following results.

\begin{corollary}
\label{cor_p_d2_d3}(a) If $\left( 5,d_{2},d_{3}\right) \neq \left(
5,6,9\right) $ and $5\leq d_{2}\leq d_{3},$ then $\left(
5,d_{2},d_{3}\right) \in \limfunc{mdeg}\left( \limfunc{Tame}\left( \Bbb{C}%
^{3}\right) \right) $ if and only if $5|d_{2}$ or $d_{3}\in 5\Bbb{N}+d_{2}%
\Bbb{N}.$\newline
(b) If $\left( 7,d_{2},d_{3}\right) \notin \left\{ \left( 7,8,12\right)
,\left( 7,10,15\right) \right\} $ and $7\leq d_{2}\leq d_{3},$ then $\left(
7,d_{2},d_{3}\right) \in \limfunc{mdeg}\left( \limfunc{Tame}\left( \Bbb{C}%
^{3}\right) \right) $ if and only if $7|d_{2}$ or $d_{3}\in 7\Bbb{N}+d_{2}%
\Bbb{N}.$\newline
(c) If $\left( 11,d_{2},d_{3}\right) \notin \left\{ \left( 11,12,18\right)
,\left( 11,14,21\right) ,\left( 11,16,24\right) ,\left( 11,18,27\right)
\right\} $ and $11\leq d_{2}\leq d_{3},$ then $\left( 11,d_{2},d_{3}\right)
\in \limfunc{mdeg}\left( \limfunc{Tame}\left( \Bbb{C}^{3}\right) \right) $
if and only if $1|d_{2}$ or $d_{3}\in 11\Bbb{N}+d_{2}\Bbb{N}.$\newline
(d) If $\left( 13,d_{2},d_{3}\right) \notin \left\{ \left( 13,14,21\right)
,\left( 13,16,24\right) ,\left( 13,18,27\right) ,\left( 13,20,30\right)
,\left( 13,22,33\right) \right\} $ and $13\leq d_{2}\leq d_{3},$ then $%
\left( 13,d_{2},d_{3}\right) \in \limfunc{mdeg}\left( \limfunc{Tame}\left( 
\Bbb{C}^{3}\right) \right) $ if and only if $13|d_{2}$ or $d_{3}\in 13\Bbb{N}%
+d_{2}\Bbb{N}.$
\end{corollary}

\begin{proof}
One can easily check that, for example, for $p=11$ the only triples of the
form $\left( 11,d_{2},d_{3}\right) $, with $11\leq d_{2}\leq d_{3},$ that
does not satisfy neither condition (1) nor condition (2) of the above
theorem are $\left( 11,12,18\right) ,\left( 11,14,21\right) ,\left(
11,16,24\right) $ and $\left( 11,18,27\right) .$
\end{proof}

The point (a) of the above corollary says that we have almost complete
description of the set
\begin{equation}
\left\{ \left( 5,d_{2},d_{3}\right) :5\leq d_{2}\leq d_{3}\right\} \cap 
\limfunc{mdeg}\left( \limfunc{Tame}\left( \Bbb{C}^{3}\right) \right) .
\label{row_pdd_p5_1}
\end{equation}

The only one thing that we do not know is whetear $\left( 5,6,9\right) $ is
an element of this set. One can, of course, notice that $9\notin 5\Bbb{N}+6%
\Bbb{N}.$ In the next section we show that $\left( 5,6,9\right) \notin 
\limfunc{mdeg}\left( \limfunc{Tame}\left( \Bbb{C}^{3}\right) \right) ,$ and
so we obtain the complete description of the set (\ref{row_pdd_p5_1}).

\subsection{Tame automorphism of $\Bbb{C}^{3}$ with multidegree equal $%
\left( 5,6,9\right) $ and the Jacobian Conjecture}

Our main purpose in this section is to prove the following result.

\begin{theorem}
\label{lem_5_6_9}There is no tame automorphism of $\Bbb{C}^{3}$ with
multidegree equal to $\left( 5,6,9\right) .$
\end{theorem}

Before we give the proof of the above theorem we recall some positive
results about the Jacobian Conjecture in dimension two. In the proof of the
theorem we can use one of such results but for the completeness we recall a
little bit more.

The first one is the following result due to Magnus \cite{Magnus}.

\begin{theorem}
\textit{(Magnus, see also \cite{van den Essen}, Thm. 10.2.24)} Let $F=\left(
P,Q\right) $ be a Keller map (i.e. such that $JacF=1$). If $\gcd \left( \deg
P,\deg Q\right) =1$ then $F$ is invertible and $\deg P=1$ or $\deg Q=1.$
\end{theorem}

The next, also due to Magnus, is the following corollary of the above
theorem.

\begin{corollary}
\textit{(Magnus, see e.g. \cite{van den Essen})} If $F=\left( P,Q\right) $
is a Keller map and $\deg P$ or $\deg Q$ is a prime number, then $F$ is
invertible.
\end{corollary}

Later Applegate, Onishi and Nagata improved the result of Magnus.

\begin{theorem}
\textit{(Applegate, Onishi, Nagata, see e.g. \cite{Applegate, BabaNakai} or 
\cite{van den Essen}) Let }$F=\left( P,Q\right) $ be a Keller map and $%
d=\gcd \left( \deg P,\deg Q\right) .$ If $d\leq 8$ or $d$ is a prime number,
then $F$ is invertible.
\end{theorem}

The last result that we recall here is the following one due to Moh \cite
{Moh}.

\begin{theorem}
\textit{(see also \cite{van den Essen})} \label{tw_JC_dim2_Moh}Let $F:\Bbb{C}%
^{2}\rightarrow \Bbb{C}^{2}$ be a Keller map with $\deg F\leq 101.$ Then $F$
is invertible.
\end{theorem}

Now we can give the proof of Theorem \ref{lem_5_6_9}.

\begin{proof}
By Theorem \ref{tw_reduc_type_4}, it is enough to show that no
(hypothetical) automorphism $F$ of $\Bbb{C}^{3}$ with $\limfunc{mdeg}%
F=\left( 5,6,9\right) $ admits an elementary reduction. Moreover, it is
enough to show this for automorphisms $F:\Bbb{C}^{3}\rightarrow \Bbb{C}^{3}$
such that $F\left( 0,0,0\right) =\left( 0,0,0\right) .$

Aassume, to the contrary, that there is an automorphism $F=\left(
F_{1},F_{2},F_{3}\right) :\Bbb{C}^{3}\rightarrow \Bbb{C}^{3}$ with $\limfunc{%
mdeg}F=\left( 5,6,9\right) $ that admits an elementary reduction.

Assume that 
\begin{equation*}
\left( F_{1},F_{2},F_{3}-g\left( F_{1},F_{2}\right) \right) ,
\end{equation*}
where $g\in \Bbb{C}\left[ X,Y\right] ,$ is an elementary reduction of $%
\left( F_{1},F_{2},F_{3}\right) .$ Then 
\begin{equation}
\deg g\left( F_{1},F_{2}\right) =\deg F_{3}=9.  \label{row_tw_569_1}
\end{equation}
By Proposition \ref{prop_deg_g_fg}, 
\begin{equation}
\deg g\left( F_{1},F_{2}\right) \geq q\left( 5\cdot 6-6-5+\deg \left[
F_{1},F_{2}\right] \right) +6r,  \label{row_tw_569_2}
\end{equation}
where $\deg _{Y}g\left( X,Y\right) =5q+r,$ with $0\leq r<5.$ Since $5\cdot
6-6-5+\deg \left[ F_{1},F_{2}\right] \geq 19+\deg \left[ F_{1},F_{2}\right]
>9,$ by (\ref{row_tw_569_1}) and (\ref{row_tw_569_2}) we have $q=0.$ Also by
(\ref{row_tw_569_1}) and (\ref{row_tw_569_2}) we have $r<2.$ Thus $g\left(
X,Y\right) =g_{0}\left( X\right) +Yg_{0}\left( X\right) ,$ and since $5\Bbb{N%
}\cap \left( 6+5\Bbb{N}\right) =\emptyset ,$ it follows that 
\begin{equation*}
9=\deg g\left( F_{1},F_{2}\right) \in 5\Bbb{N}\cup \left( 6+5\Bbb{N}\right) ,
\end{equation*}
a contradiction.

Now, assume that 
\begin{equation*}
\left( F_{1},F_{2}-g\left( F_{1},F_{3}\right) ,F_{3}\right) ,
\end{equation*}
where $g\in \Bbb{C}\left[ X,Y\right] ,$ is an elementary reduction of $%
\left( F_{1},F_{2},F_{3}\right) .$ Then 
\begin{equation}
\deg g\left( F_{1},F_{3}\right) =\deg F_{2}=6.  \label{row_tw_569_3}
\end{equation}
By Proposition \ref{prop_deg_g_fg} 
\begin{equation}
\deg g\left( F_{1},F_{3}\right) \geq q\left( 5\cdot 9-9-5+\deg \left[
F_{1},F_{3}\right] \right) +9r,  \label{row_tw_569_4}
\end{equation}
where $\deg _{Y}g\left( X,Y\right) =5q+r,$ with $0\leq r<5.$ Since $5\cdot
9-9-5+\deg \left[ F_{1},F_{3}\right] \geq 31+\deg \left[ F_{1},F_{3}\right]
>6,$ we have $q=r=0.$ This means that $g\left( X,Y\right) =g\left( X\right) ,
$ and so 
\begin{equation*}
\deg g\left( F_{1},F_{2}\right) =\deg g\left( F_{1}\right) \in 5\Bbb{N},
\end{equation*}
a contradiction.

Finally, assume that 
\begin{equation*}
\left( F_{1}-g\left( F_{2},F_{3}\right) ,F_{2},F_{3}\right) ,
\end{equation*}
where $g\in \Bbb{C}\left[ X,Y\right] ,$ is an elementary reduction of $%
\left( F_{1},F_{2},F_{3}\right) .$ By Theorem \ref{tw_reduc_type_4}, we can
also assume that $F\left( 0,0,0\right) =\left( 0,0,0\right) .$ We have 
\begin{equation}
\deg g\left( F_{2},F_{3}\right) =\deg F_{1}=5  \label{row_tw_569_5}
\end{equation}
and by Proposition \ref{prop_deg_g_fg}, 
\begin{equation}
\deg g\left( F_{2},F_{3}\right) \geq q\left( p\cdot 9-9-6+\deg \left[
F_{2},F_{3}\right] \right) +9r,  \label{row_tw_569_6}
\end{equation}
where $\deg _{Y}g\left( X,Y\right) =qp+r,$ with $0\leq r<p$ and $p=\frac{6}{%
\gcd \left( 6,9\right) }=2.$ By (\ref{row_tw_569_5}) and (\ref{row_tw_569_6}%
), $r=0.$

Consider the case $\deg \left[ F_{2},F_{3}\right] >2.$ Then $p\cdot
9-9-6+\deg \left[ F_{2},F_{3}\right] =3+\deg \left[ F_{2},F_{3}\right] >5,$
and then by (\ref{row_tw_569_5}) and (\ref{row_tw_569_6}) we see that $q=0.$
Thus in this case, we have $g\left( X,Y\right) =g\left( X\right) ,$ and so $%
\deg g\left( F_{2},F_{3}\right) =\deg g\left( F_{2}\right) \in 6\Bbb{N}.$
This contradicts (\ref{row_tw_569_5}).

Now, consider the case $\deg \left[ F_{2},F_{3}\right] =2$ (since $%
F_{2},F_{3}$ are algebraically independent, we have $\deg \left[
F_{2},F_{3}\right] \geq 2$). Let $L$ be the linear part of the automorphism $%
F.$ Since $F\left( 0,0,0\right) =(0,0,0),$ the linear part of $F\circ L^{-1}$
is the identity map $\limfunc{id}_{\Bbb{C}^{3}}.$ Thus 
\begin{eqnarray}
F_{2}\circ L^{-1} &=&X_{2}+\text{higher degree summands,}
\label{row_tw_569_7} \\
F_{3}\circ L^{-1} &=&X_{3}+\text{higher degree summands.}  \notag
\end{eqnarray}
Since, by Lemma \ref{lem_degree_linear_change}, 
\begin{equation*}
\deg \left[ F_{2}\circ L^{-1},F_{3}\circ L^{-1}\right] =\deg \left[
F_{2},F_{3}\right] =2,
\end{equation*}
it follows, by Lemma \ref{lem_deg_poiss_2}, that 
\begin{equation*}
F_{2}\circ L^{-1},F_{3}\circ L^{-1}\in \Bbb{C}\left[ X_{2},X_{3}\right] .
\end{equation*}
But $\deg \left[ F_{2}\circ L^{-1},F_{3}\circ L^{-1}\right] =2$ means that 
\begin{equation*}
Jac\left( F_{2}\circ L^{-1},F_{3}\circ L^{-1}\right) \in \Bbb{C}^{*}
\end{equation*}
(of course we consider here $F_{2}\circ L^{-1},F_{3}\circ L^{-1}$ as
functions of two variables $X_{2},X_{3}$). By Lemma \ref
{lem_mdeg_linear_change} we have $\deg \left( F_{2}\circ L^{-1}\right)
=6,\deg \left( F_{3}\circ L^{-1}\right) =9.$ Then, by Theorem \ref
{tw_JC_dim2_Moh}, the map $\left( F_{2}\circ L^{-1},F_{3}\circ L^{-1}\right)
:\Bbb{C}^{3}\rightarrow \Bbb{C}^{3}$ is an automorphism. But $6\nmid 9$
contradict Jung - van der Kulk theorem (see Thoerem \ref
{tw_jung_van_der_Kulk} and Corollary \ref{cor_dim2_tame_d1_d2}).
\end{proof}

By Theorem \ref{lem_5_6_9} and Corollary \ref{cor_p_d2_d3} (a) we obtain the
following result.

\begin{corollary}
The following equality holds true 
\begin{gather*}
\left\{ \left( 5,d_{2},d_{3}\right) :5\leq d_{2}\leq d_{3}\right\} \cap 
\limfunc{mdeg}\left( \limfunc{Tame}\left( \Bbb{C}^{3}\right) \right) = \\
=\left\{ \left( 5,d_{2},d_{3}\right) :5\leq d_{2}\leq d_{3},5|d_{2}\text{ or 
}d_{3}\in 5\Bbb{N}+d_{2}\Bbb{N}\right\} .
\end{gather*}
\end{corollary}

\subsection{The case $\left( p,2\left( p-2\right) ,3\left( p-2\right)
\right) $}

In the same manner as we proved Theorem \ref{lem_5_6_9} one can show the
following

\begin{theorem}
\label{tw_p_od_5_do_35}Let $p\geq 5$ be a prime number such that $p\leq 35.$
Then $\left( p,2\left( p-2\right) ,3\left( p-3\right) \right) \notin 
\limfunc{mdeg}\left( \limfunc{Tame}\left( \Bbb{C}^{3}\right) \right) .$
\end{theorem}

\begin{proof}
Since $3\left( p-2\right) \leq 101,$ it follows that one can use Theorem \ref
{tw_JC_dim2_Moh} and repeat the arguments from the proof of Theorem \ref
{lem_5_6_9}.
\end{proof}

By the above theorem and Corollary \ref{cor_p_d2_d3} we obtain the following
result.

\begin{corollary}
The following equality holds true 
\begin{gather*}
\left[ \left\{ \left( 7,d_{2},d_{3}\right) :7\leq d_{2}\leq d_{3}\right\}
\cap \limfunc{mdeg}\left( \limfunc{Tame}\left( \Bbb{C}^{3}\right) \right)
\right] \backslash \left\{ \left( 7,8,12\right) \right\} = \\
=\left\{ \left( 7,d_{2},d_{3}\right) :7\leq d_{2}\leq d_{3},7|d_{2}\text{ or 
}d_{3}\in 7\Bbb{N}+d_{2}\Bbb{N}\right\} .
\end{gather*}
\end{corollary}

The above corollary means that to obtain the complete description of the set 
$\left\{ \left( 7,d_{2},d_{3}\right) :7\leq d_{2}\leq d_{3}\right\} \cap 
\limfunc{mdeg}\left( \limfunc{Tame}\left( \Bbb{C}^{3}\right) \right) $ we
''only'' need to know whether $\left( 7,8,12\right) \in \limfunc{mdeg}\left( 
\limfunc{Tame}\left( \Bbb{C}^{3}\right) \right) .$

At the end of this subsection notice the following result.

\begin{theorem}
\label{tw_p_ponad_35}The Jacobian Conjecture for dimension two iimplies that
for a prime numbers $p\geq 5$ we have $\left( p,2\left( p-2\right) ,3\left(
p-2\right) \right) \notin \limfunc{mdeg}\left( \limfunc{Tame}\left( \Bbb{C}%
^{3}\right) \right) .$
\end{theorem}

\begin{proof}
If we assume that the Jacobian Conjecture for dimension two holds true, then
one can repeat the arguments from the proof of Theorem \ref{lem_5_6_9}.
\end{proof}

\begin{corollary}
If there is a tame automorphism $F$ of $\Bbb{C}^{3}$ with $\limfunc{mdeg}%
F=\left( p,2\left( p-2\right) ,3\left( p-2\right) \right) ,$ where $p>35$ is
a prime number, then the Jacobian Conjecture for dimension two is false.
\end{corollary}

\begin{proof}
This is a consequence of Theorem \ref{tw_p_od_5_do_35} and Theorem \ref
{tw_p_ponad_35}.
\end{proof}

In particular we have the following

\begin{theorem}
If there is a tame automorphism $F$ of $\Bbb{C}^{3}$ with $\limfunc{mdeg}%
F=\left( 37,70,105\right) ,$ then the two-dimensional Jacobian Conjecture is
false.
\end{theorem}

\newpage 

\section{Finiteness results\label{section_ab_skonczone}}

Let us consider the set 
\begin{equation*}
T_{a,b}^{(n)}=\left\{ \left( d_{1},\ldots ,d_{n}\right) \in \left( \Bbb{N}%
^{*}\right) ^{n}:d_{1}\leq \ldots \leq d_{n},d_{1}=a,d_{2}=b\right\}
\backslash \limfunc{mdeg}\left( \limfunc{Tame}\left( \Bbb{C}^{n}\right)
\right) .
\end{equation*}
Of course, by Jung-Van der Kulk result, $T_{a,b}^{(2)}=\left\{ \left(
a,b\right) \right\} $ if $a\nmid b,$ and $T_{a,b}^{(2)}=\emptyset $ if $a|b.$
Thus $\#T_{a,b}^{(2)}\leq 1<+\infty $ for all $1\leq a\leq b.$ We will show
that also for $n\geq 3$ the set $T_{a,b}^{(n)}$ is finite. For $n=3$ this
result is due to Zygad\l o \cite{Zygadlo}.

\begin{theorem}
\label{tw_finite_result_1}For all integers $1\leq a\leq b$ the set $%
T_{a,b}^{(3)}$ is finite. Moreover the following inclusion is true 
\begin{equation*}
T_{a,b}^{(3)}\subset \left\{ \left( a,b,d_{3}\right) :d_{3}<\func{lcm}\left(
a,b\right) -r\right\} ,
\end{equation*}
where $r=\min \left\{ b-1,\left( a-1\right) \left( \left\lfloor \frac{b}{a}%
\right\rfloor +1\right) \right\} .$
\end{theorem}

The original proof of the above theorem due to Zygad\l o can be found in 
\cite{Zygadlo}, but we give here another, simpler proof. It is based on the
proof of Proposition \ref{prop_4_4k+2_d3}, but there are also similarities
to the proof in \cite{Zygadlo}.

\begin{proof}
First of all notice that without loss of generality we can assume that $%
1<a<b.\,$Indeed, if $a=1\,$or $a=b,\,$then by Proposition \ref{prop_sum_d_i}
we have $T_{a,b}^{(3)}=\emptyset .$ Thus up to the end of the proof we
assume that $1<a<b.$

Let $d=\gcd \left( a,b\right) .$ Then $a=d\widetilde{a},$ $b=d\widetilde{b},$
where $\widetilde{a},\widetilde{b}\in \Bbb{N}^{*}$ are coprime numbers. We
have $\func{lcm}\left( a,b\right) =\frac{ab}{\gcd \left( a,b\right) }=a%
\widetilde{b}=b\widetilde{a}.$ Let us notice that 
\begin{equation}
\left( X+Z^{a}\right) ^{\widetilde{b}}=\sum_{l=0}^{\widetilde{b}}\binom{%
\widetilde{b}}{l}X^{l}Z^{a\widetilde{b}-la}  \label{tw_T_ab_Zygad_1}
\end{equation}
and 
\begin{gather}
\left( Y+Z^{p}+\sum_{l=0}^{\left\lfloor \frac{b}{a}\right\rfloor
}a_{l}X^{l}Z^{b-la}\right) ^{\widetilde{a}}=  \label{tw_T_ab_Zygad_2} \\
=\sum_{s_{1}+s_{2}=\widetilde{a},s_{1}>0}\left( Y+Z^{p}\right)
^{s_{1}}\left( \sum_{l=0}^{\left\lfloor \frac{b}{a}\right\rfloor
}a_{l}X^{l}Z^{b-la}\right) ^{s_{2}}+\left( \sum_{l=0}^{\left\lfloor \frac{b}{%
a}\right\rfloor }a_{l}X^{l}Z^{b-la}\right) ^{\widetilde{a}}.  \notag
\end{gather}
If we take $p<b,$ then 
\begin{equation*}
\deg \left[ \sum_{s_{1}+s_{2}=\widetilde{a},s_{1}>0}\left( Y+Z^{p}\right)
^{s_{1}}\left( \sum_{l=0}^{\left\lfloor \frac{b}{a}\right\rfloor
}a_{l}X^{l}Z^{b-la}\right) ^{s_{2}}\right] \le p+b\left( \widetilde{a}%
-1\right) ,
\end{equation*}
and since $Z^{p+b\left( \widetilde{a}-1\right) }$ can be obtained in the
above polynomial only in one way, that we actually have (provided that $%
a_{0}\neq 0$): 
\begin{equation}
\deg \left[ \sum_{s_{1}+s_{2}=\widetilde{a},s_{1}>0}\left( Y+Z^{p}\right)
^{s_{1}}\left( \sum_{l=0}^{\left\lfloor \frac{b}{a}\right\rfloor
}a_{l}X^{l}Z^{b-la}\right) ^{s_{2}}\right] =p+b\left( \widetilde{a}-1\right)
.  \label{tw_T_ab_Zygad_3}
\end{equation}

In the sequel, we will take $p\in \left\{ 1,\ldots ,b-1\right\} $ such that $%
p+b\left( \widetilde{a}-1\right) =\func{lcm}\left( a,b\right) -r,\ldots ,%
\func{lcm}\left( a,b\right) -r+\left( a-1\right) .$ This is possible,
because $b\left( \widetilde{a}-1\right) +1\leq \func{lcm}\left( a,b\right)
-r $ and $\func{lcm}\left( a,b\right) -r+\left( a-1\right) <\func{lcm}\left(
a,b\right) =b\widetilde{a}.$

Now, using (\ref{tw_T_ab_Zygad_1}), (\ref{tw_T_ab_Zygad_2}) and (\ref
{tw_T_ab_Zygad_3}) we obtain that the summands of degree greather than $%
p+b\left( \widetilde{a}-1\right) \,$in the polynomial 
\begin{equation*}
\left( X+Z^{a}\right) ^{\widetilde{b}}-\left(
Y+Z^{p}+\sum_{l=0}^{\left\lfloor \frac{b}{a}\right\rfloor
}a_{l}X^{l}Z^{b-la}\right) ^{\widetilde{a}}
\end{equation*}
are 
\begin{eqnarray*}
&&\left( 1-a_{0}^{\widetilde{a}}\right) Z^{a\widetilde{b}}, \\
&&\left[ \binom{\widetilde{b}}{1}-\binom{\widetilde{a}}{1}a_{0}^{\widetilde{a%
}-1}a_{1}\right] XZ^{a\left( \widetilde{b}-1\right) }, \\
&&\left[ \binom{\widetilde{b}}{2}-\binom{\widetilde{a}}{2}a_{0}^{\widetilde{a%
}-2}a_{1}^{2}-\binom{\widetilde{a}}{1}a_{0}^{\widetilde{a}-1}a_{2}\right]
X^{2}Z^{a\left( \widetilde{b}-2\right) },
\end{eqnarray*}
and for $k=3,\ldots ,\left\lfloor \frac{b}{a}\right\rfloor $%
\begin{equation*}
\left[ \binom{\widetilde{b}}{k}-\left( \sum_{l_{1}+\cdots +l_{\widetilde{a}%
}=k,l_{i}<k}a_{l_{1}}\cdots a_{l_{\widetilde{a}}}\right) -\binom{\widetilde{a%
}}{1}a_{0}^{\widetilde{a}-1}a_{k}\right] X^{k}Z^{a\left( \widetilde{b}%
-k\right) }.
\end{equation*}
Thus we can recursively choose coefficients $a_{0},\ldots ,a_{\left\lfloor 
\frac{b}{a}\right\rfloor }$ such that all expressions in the brackets above
are equal to zero. Since also in the polynomial 
\begin{equation*}
\left( X+Z^{a}\right) ^{\widetilde{b}}-\left( \sum_{l=0}^{\left\lfloor \frac{%
b}{a}\right\rfloor }a_{l}X^{l}Z^{b-la}\right) ^{\widetilde{a}}
\end{equation*}
there is no summands belonging to $\Bbb{C}\left[ Z\right] \backslash \Bbb{C}$
(provided that $a_{0}=1$), then 
\begin{equation*}
\deg \left[ \left( X+Z^{a}\right) ^{\widetilde{b}}-\left(
Y+Z^{p}+\sum_{l=0}^{\left\lfloor \frac{b}{a}\right\rfloor
}a_{l}X^{l}Z^{b-la}\right) ^{\widetilde{a}}\right] =p+b\left( \widetilde{a}%
-1\right) .
\end{equation*}

Now, let $d_{3}\geq \func{lcm}\left( a,b\right) -r$ be arbitrary. Then there
are $p\in \{1,\ldots ,b-1\}$ and $q\in \Bbb{N}$ such that $p+b\left( 
\widetilde{a}-1\right) \in \left\{ \func{lcm}\left( a,b\right) -r,\ldots ,%
\func{lcm}\left( a,b\right) -r+\left( a-1\right) \right\} $ and $%
d_{3}=p+b\left( \widetilde{a}-1\right) +qa.$ By the above considerations we
obtain that 
\begin{equation*}
\limfunc{mdeg}\left( G\circ F\right) =\left( a,b,d_{3}\right) ,
\end{equation*}
where 
\begin{equation*}
F\left( x,y,z\right) =\left( x+z^{a},y+z^{p}+\sum_{l=0}^{\left\lfloor \frac{b%
}{a}\right\rfloor }a_{l}x^{l}z^{b-la},z\right)
\end{equation*}
and 
\begin{equation*}
G\left( u,v,w\right) =\left( u,v,w+\left( u^{\widetilde{b}}-v^{\widetilde{a}%
}\right) u^{q}\right) .
\end{equation*}
\end{proof}

\begin{corollary}
For $n\in \Bbb{N},n\geq 3,$ and all integers $1\leq a\leq b$ the set $%
T_{a,b}^{(n)}$ is finite. Moreover the following inclusion is true 
\begin{equation*}
T_{a,b}^{(3)}\subset \left\{ \left( a,b,d_{3},\ldots ,d_{n}\right) \in
\left( \Bbb{N}^{*}\right) ^{n}:d_{3},\ldots ,d_{n}<\func{lcm}\left(
a,b\right) -r\right\} ,
\end{equation*}
where $r$ is defined as in Theorem \ref{tw_finite_result_1}.
\end{corollary}

\begin{proof}
If for some $i\in \left\{ 3,\ldots ,n\right\} $ we have $d_{i}\geq \func{lcm}%
\left( a,b\right) -r$ (actually we can think that $i=n,$ because of the
inequalities $d_{3}\leq \ldots \leq d_{n}$) then by Theorem \ref
{tw_finite_result_1}, there exists a tame automorphism $F:\Bbb{C}%
^{3}\rightarrow \Bbb{C}^{3}$ such that $\limfunc{mdeg}F=\left(
a,b,d_{i}\right) .$ Now it is enough to use Proposition \ref{prop_sum_d_i}.
\end{proof}

\section{Multidegree of the inverse of polynomial automorphisms of $\Bbb{C}%
^{2}$}

In \cite{RuskWiniar} Rusek and Winiarski proved that for all automorphisms $%
F $ of $\Bbb{C}^{n}$ the equality $\deg F^{-1}\leq \left( \deg F\right)
^{n-1}$ holds true and hence $\deg F^{-1}=\deg F$ for $n=2.$ Here we give
complete information about $\limfunc{mdeg}F^{-1}$ for $F\in \limfunc{Aut}%
\left( \Bbb{C}^{2}\right) .$

\subsection{Multidegree and the length of the automorphism of $\Bbb{C}^{2}$}

Here we establish the relations between multidegree of a given automorphism
of $\Bbb{C}^{2}$ and the length of it. We start with the following

\begin{lemma}
\label{lem_degP_degQ}If $\left( P,Q\right) \in \limfunc{Aut}\left( \Bbb{C}%
^{2}\right) $ is such that $\deg P<\deg Q,$ then there is a polynomial $f\in 
\Bbb{C}\left[ T\right] $ with $\deg f>1$ such that:\newline
(1) $\deg \left( Q-f\left( P\right) \right) <\deg P$ if $\deg P>1$ or\newline
(2) $\deg \left( Q-f\left( P\right) \right) =1$ if $\deg P=1.$
\end{lemma}

\begin{proof}
Since $\deg Q>\deg P\geq 1,$ we have $\deg Q+\deg P>2$ and $Jac\left( 
\overline{P},\overline{Q}\right) =0$ (because $Jac\left( P,Q\right) \in \Bbb{%
C}^{*}$). By Lemma \ref{lem_f_g_alg_zalez}, 
\begin{equation*}
\overline{P}=\alpha h^{n_{1}}\qquad \overline{Q}=\beta h^{n_{2}}
\end{equation*}
for some $\alpha ,\beta \in \Bbb{C}^{*},n_{1},n_{2}\in \Bbb{N}^{*}$ and some
homogeneous polynomial $h\in \Bbb{C}\left[ X,Y\right] .$ Since $\deg 
\overline{P}|\deg \overline{Q},$ we have $n_{1}|n_{2}$ and so $\overline{Q}%
=c_{1}\overline{P}^{k_{1}}$ for some $c_{1}\in \Bbb{C}^{*}$ and $k_{1}=\frac{%
n_{2}}{n_{1}}.$ Now $\deg \left( Q-c_{1}P^{k_{1}}\right) <\deg Q,$ and if $%
\deg \left( Q-c_{1}P^{k_{1}}\right) <\deg P$ or $\deg \left(
Q-c_{1}P^{k_{1}}\right) =\deg P=1,$ then we are done. And, if $\deg \left(
Q-c_{1}P^{k_{1}}\right) >\deg P$ or $\deg \left( Q-c_{1}P^{k_{1}}\right)
=\deg P>1,$ then we can repeat the above arguments for $\overline{%
Q-c_{1}P^{k_{1}}}$ and $\overline{P}$ to obtain $c_{2}\in \Bbb{C}^{*}$ and $%
k_{2}<k_{1}$ such that $\overline{Q-c_{1}P^{k_{1}}}=c_{2}\overline{P}%
^{k_{2}}.$ Then, 
\begin{equation*}
\deg \left( Q-c_{1}P^{k_{1}}-c_{2}P^{k_{2}}\right) <\deg \left(
Q-c_{1}P^{k_{1}}\right) 
\end{equation*}
and we can proceed inductively.
\end{proof}

Now we can prove the following

\begin{proposition}
\label{Prop_trojk_minim}If $F\in \limfunc{Aut}\left( \Bbb{C}^{2}\right) ,$
then there is a number $l\in \Bbb{N}$ (including zero), affine automorphisms 
$L_{1},L_{2}$ of $\Bbb{C}^{2}$ and triangular automorphisms $T_{1},\ldots
,T_{l}$ of the forms 
\begin{eqnarray}
T_{i} &:&\Bbb{C}^{2}\ni \left( x,y\right) \mapsto \left( x,y+f_{i}(x)\right)
\in \Bbb{C}^{2}\qquad \text{for }i=1,3,\ldots   \label{Row_Prop_amalg_1} \\
T_{i} &:&\Bbb{C}^{2}\ni \left( x,y\right) \mapsto \left( x+f_{i}(y),y\right)
\in \Bbb{C}^{2}\qquad \text{for }i=2,4,\ldots   \label{Row_Prop_amalg_2}
\end{eqnarray}
with $\deg f_{i}>1,$ such that 
\begin{equation*}
F=L_{2}\circ T_{l}\circ \cdots \circ T_{1}\circ L_{1}.
\end{equation*}
Moreover, the number $l$ is unique, and one can require that $T_{i},$ $%
i=1,\ldots ,l$ are of the form (\ref{Row_Prop_amalg_1}) for even $i$ and of
the form (\ref{Row_Prop_amalg_2}) for odd $i.$
\end{proposition}

\begin{proof}
Let $F=\left( F_{1},F_{2}\right) .$ If $\deg F_{1}=\deg F_{2}=1,$ then $F$
is an affine mapping and we have $F=L_{2}\circ L_{1}$ for $L_{2}=\limfunc{id}%
_{\Bbb{C}^{2}}$ and $L_{1}=F.$

If $\deg F_{1}=\deg F_{2}>1,$ then $Jac\left( \overline{F_{1}},\overline{%
F_{2}}\right) =0$ (because $Jac\left( F_{1},F_{2}\right) \in \Bbb{C}^{*}$).
Thus, by Lemma \ref{lem_f_g_alg_zalez} 
\begin{equation*}
\overline{F_{1}}=\alpha h^{n}\qquad \overline{F_{2}}=\beta h^{n}
\end{equation*}
for some $\alpha ,\beta \in \Bbb{C}^{*},n\in \Bbb{N}^{*}$ and some
homogeneous polynomial $h\in \Bbb{C}\left[ X,Y\right] .$ Let $L_{2}\left(
x,y\right) =\left( x+\frac{\alpha }{\beta }y,y\right) $ and 
\begin{equation*}
\left( G_{1},G_{2}\right) =L_{2}^{-1}\circ F.
\end{equation*}
Then $\deg G_{2}=\deg F_{2}$ (actually $G_{2}=F_{2}$) and $\deg G_{1}<\deg
G_{2}.$ Hence we can assume that $\deg F_{1}\neq \deg F_{2},$ and without
loss of generality that $\deg F_{1}<\deg F_{2}$ (if $\deg F_{1}>\deg F_{2},$
then for $\left( G_{1},G_{2}\right) =L_{2}^{-1}\circ F,$ where $L_{2}\left(
x,y\right) =\left( y,x\right) ,$ we have $\deg G_{1}<\deg G_{2}$).

By Lemma \ref{lem_degP_degQ}, we obtain a polynomial $f\in \Bbb{C}\left[
T\right] ,$ $\deg f>1,$ such that for $T_{1}\left( x,y\right) =\left(
x,y+f\left( x\right) \right) $ and $\left( G_{1},G_{2}\right)
=T_{1}^{-1}\circ F$ we have $\deg G_{2}<\deg G_{1}$ or $\deg G_{2}=\deg
G_{1}=1.$ In the second case $\left( G_{1},G_{2}\right) $ is an affine map
and for $L_{1}=\left( G_{1},G_{2}\right) $ we have $F=T_{1}\circ L_{1},$ so
we are done. And in the first case we can next time use Lemma \ref
{lem_degP_degQ} and proceed inductively.

Thus we can assume that $F=\widetilde{L}_{2}\circ \widetilde{T}_{1}\circ
\cdots \circ \widetilde{T}_{l}\circ \widetilde{L}_{1},$ where $\widetilde{L}%
_{1},\widetilde{L}_{2}\in \limfunc{Aff}\left( \Bbb{C}^{2}\right) $ and $%
\widetilde{T}_{i}$ are of the forms (\ref{Row_Prop_amalg_1}), (\ref
{Row_Prop_amalg_2}). Let us set 
\begin{eqnarray*}
T_{i} &=&\left\{ 
\begin{array}{ll}
\widetilde{T}_{l+1-i}, & \text{for odd }l, \\ 
L\circ \widetilde{T}_{l+1-i}\circ L, & \text{for even }l,
\end{array}
\right. \\
L_{1} &=&\left\{ 
\begin{array}{ll}
\widetilde{L}_{1}, & \text{for odd }l, \\ 
L\circ \widetilde{L}_{1}, & \text{for even }l,
\end{array}
\right. \qquad L_{2}=\left\{ 
\begin{array}{ll}
\widetilde{L}_{2}, & \text{for odd }l, \\ 
\widetilde{L}_{2}\circ L, & \text{for even }l,
\end{array}
\right.
\end{eqnarray*}
where $L(x,y)=(y,x).$ Then one can check that $F=L_{2}\circ T_{l}\circ
\cdots \circ T_{1}\circ L_{1}.$

To see that $l$ is unique it is enough to notice that $L\circ T_{j}\circ
L\in J\left( \Bbb{C}^{2}\right) \backslash \limfunc{Aff}\left( \Bbb{C}%
^{2}\right) ,$ $j=1,3,\ldots $ and $T_{j}\in J\left( \Bbb{C}^{2}\right)
\backslash \limfunc{Aff}\left( \Bbb{C}^{2}\right) ,$ $j=2,4,\ldots ,$ and so 
\begin{equation*}
F=\widehat{L}_{2}\circ \cdots \circ L\circ \left( L\circ T_{3}\circ L\right)
\circ L\circ T_{2}\circ L\circ \left( L\circ T_{1}\circ L\right) \circ
\left( L\circ L_{1}\right) ,
\end{equation*}
is the amalgamated representation of $F$ for a suitable choosen sets $\Phi $
and $\Psi $ (see Definition and Proposition ), where 
\begin{equation*}
\widehat{L}_{2}=\left\{ 
\begin{array}{ll}
\widetilde{L}_{2}, & \text{for even }l, \\ 
\widetilde{L}_{2}\circ L, & \text{for odd }l.
\end{array}
\right.
\end{equation*}

To see that the last statement holds true, one can write 
\begin{equation*}
F=\left( L_{2}\circ L\right) \circ \left( L\circ T_{l}\circ L\right) \circ
\cdots \circ \left( L\circ T_{1}\circ L\right) \circ \left( L\circ
L_{1}\right) .
\end{equation*}
\end{proof}

\begin{definition}
Let $F\in \limfunc{Aut}\left( \Bbb{C}^{2}\right) $ be a polynomial
automorphism. The number $l$ from Proposition \ref{Prop_trojk_minim} is
called length of $F$ and denoted $\limfunc{length}F.$
\end{definition}

In what follows we will use the following numerical object.

\begin{definition}
Let $k\in \Bbb{N}^{*}$ and let $k=p_{1}^{\alpha _{1}}\cdots p_{r}^{\alpha
_{r}}$ be its prime decomposition. Then by $l\left( k\right) $ we denote the
number $\alpha _{1}+\cdots +\alpha _{r}.$
\end{definition}

Obviously, we have $l\left( k_{1}k_{2}\right) =l\left( k_{1}\right) +l\left(
k_{2}\right) ,$ for all $k_{1},k_{2}\in \Bbb{N}^{*},$ and $l\left( k\right)
\geq 1$ for $k>1.$

\begin{theorem}
\label{tw_length_ndeg}Let $F\in \limfunc{Aut}\left( \Bbb{C}^{2}\right) .$
Then:\newline
(1) if $\limfunc{length}F=1,$ then $\limfunc{mdeg}F\in \left\{ \left(
1,d\right) ,\left( d,1\right) ,\left( d,d\right) \right\} ,$ where $1<d,$%
\newline
(2) if $\limfunc{length}F=2,$ then either $\limfunc{mdeg}F\in \left\{ \left(
d_{1},d_{2}\right) ,\left( d_{2},d_{1}\right) \right\} $ with $%
1<d_{1}<d_{2},d_{1}|d_{2}$ or $\limfunc{mdeg}F=\left( d,d\right) $ with $%
l\left( d\right) \geq 2$ (in particular $d>1\,$is a composite number),%
\newline
(3) if $\limfunc{length}F\geq 3,$ then either $\limfunc{mdeg}F\in \left\{
\left( d_{1},d_{2}\right) ,\left( d_{2},d_{1}\right) \right\} $ with $%
1<d_{1}<d_{2},d_{1}|d_{2},l\left( d_{1}\right) \geq \limfunc{length}F-1$ or $%
\limfunc{mdeg}F=\left( d,d\right) $ with $l\left( d\right) \geq \limfunc{%
length}F.$
\end{theorem}

\begin{proof}
(1) Since $\limfunc{length}F=1,$ $F=L_{2}\circ T\circ L_{1},$ where $%
L_{1},L_{2}\in \limfunc{Aff}\left( \Bbb{C}^{2}\right) $ and $T$ is of the
form $T:\Bbb{C}^{2}\ni \left( x,y\right) \mapsto \left( x,y+f(x)\right) \in 
\Bbb{C}^{2},$ with $\deg f>1.$ Thus $\limfunc{mdeg}\left( T\circ
L_{1}\right) =\left( 1,d\right) ,\,$where $d=\deg f,$ and then one can easy
check that $\limfunc{mdeg}\left( L_{2}\circ T\circ L_{1}\right) \in \left\{
\left( 1,d\right) ,\left( d,1\right) ,\left( d,d\right) \right\} .$

(2) Since $\limfunc{length}F=2,$ $F=L_{2}\circ T_{2}\circ T_{1}\circ L_{1},$
where $L_{1},L_{2}\in \limfunc{Aff}\left( \Bbb{C}^{2}\right) $ and $%
T_{1},T_{2}$ are of the following forms 
\begin{eqnarray*}
T_{1} &:&\Bbb{C}^{2}\ni \left( x,y\right) \mapsto \left( x,y+f_{1}(x)\right)
\in \Bbb{C}^{2}, \\
T_{2} &:&\Bbb{C}^{2}\ni \left( x,y\right) \mapsto \left( x+f_{2}(y),y\right)
\in \Bbb{C}^{2},
\end{eqnarray*}
with $\deg f_{1},\deg f_{2}>1.$ Thus $\limfunc{mdeg}\left( T_{1}\circ
L_{1}\right) =\left( 1,\deg f_{1}\right) ,$ and then $\limfunc{mdeg}\left(
T_{2}\circ T_{1}\circ L_{1}\right) =\left( d_{2},d_{1}\right) ,$ where $%
d_{1}=\deg f_{1},$ $d_{2}=\deg f_{2}\cdot \deg f_{1}.$ Since $\deg
f_{1},\deg f_{2}>1,$ it follows that $l\left( d_{2}\right) =l\left( \deg
f_{1}\right) +l\left( \deg f_{2}\right) \geq 2.$ Now, one can easy see that $%
\limfunc{mdeg}\left( L_{2}\circ T_{2}\circ T_{1}\circ L_{1}\right) \in
\left\{ \left( d_{1},d_{2}\right) ,\left( d_{2},d_{1}\right) ,\left(
d_{2},d_{2}\right) \right\} .$

(3) Since $l=\limfunc{length}F\geq 3,$ $F=L_{2}\circ T_{l}\circ \cdots \circ
T_{1}\circ L_{1},$ where $L_{1},L_{2}\in \limfunc{Aff}\left( \Bbb{C}%
^{2}\right) $ and $T_{1},\ldots ,T_{l}$ are of the following forms 
\begin{equation*}
T_{i}:\Bbb{C}^{2}\ni \left( x,y\right) \mapsto \left( x+f_{i}(y),y\right)
\in \Bbb{C}^{2},
\end{equation*}
for even $i,$ and 
\begin{equation*}
T_{i}:\Bbb{C}^{2}\ni \left( x,y\right) \mapsto \left( x,y+f_{i}(x)\right)
\in \Bbb{C}^{2},
\end{equation*}
for odd $i,$ with $\deg f_{i}>1$ for $i=1,\ldots ,l.$ Now, one can easy
check that 
\begin{equation*}
\limfunc{mdeg}\left( T_{l}\circ \cdots \circ T_{1}\circ L_{1}\right)
=\left\{ 
\begin{array}{ll}
\left( \prod_{j=1}^{l}\deg f_{j},\prod_{j=1}^{l-1}\deg f_{j}\right) , & 
\text{for even }l, \\ 
\left( \prod_{j=1}^{l-1}\deg f_{j},\prod_{j=1}^{l}\deg f_{j}\right) , & 
\text{for odd }l.
\end{array}
\right.
\end{equation*}
Let 
\begin{equation*}
d_{2}=\prod_{j=1}^{l}\deg f_{j}\qquad \text{and\qquad }d_{1}=%
\prod_{j=1}^{l-1}\deg f_{j}.
\end{equation*}
Then, $\limfunc{mdeg}\left( T_{l}\circ \cdots \circ T_{1}\circ L_{1}\right)
=\left( d_{1},d_{2}\right) $ for odd $l,$ and $\limfunc{mdeg}\left(
T_{l}\circ \cdots \circ T_{1}\circ L_{1}\right) =\left( d_{2},d_{1}\right) $
for even $l.$

Since $\deg f_{i}>1,$ for $i=1,\ldots ,l,$ we have 
\begin{equation*}
l\left( d_{1}\right) \geq l\left( \deg f_{1}\right) +\cdots +l\left( \deg
f_{l-1}\right) \geq l-1
\end{equation*}
and 
\begin{equation*}
l\left( d_{2}\right) \geq l\left( \deg f_{1}\right) +\cdots +l\left( \deg
f_{l}\right) \geq l.
\end{equation*}
Of course, as in the previous case, we have 
\begin{equation*}
\limfunc{mdeg}\left( L_{2}\circ T_{l}\circ \cdots \circ T_{1}\circ
L_{1}\right) \in \left\{ \left( d_{1},d_{2}\right) ,\left(
d_{2},d_{1}\right) ,\left( d_{2},d_{2}\right) \right\} .
\end{equation*}
\end{proof}

\begin{theorem}
\label{tw_mdeg_length}Let $F\in \limfunc{Aut}\left( \Bbb{C}^{2}\right) $ be
arbitrary polynomial automorphism with $\limfunc{mdeg}F=\left(
d_{1},d_{2}\right) ,d_{1}\leq d_{2}.$ Then $\limfunc{length}F\leq \min
\left\{ l\left( d_{2}\right) ,l\left( d_{1}\right) +1\right\} .$
\end{theorem}

\begin{proof}
This is a consequence of Theorem \ref{tw_length_ndeg}.
\end{proof}

\subsection{The case of length 1.}

Here we consider the situation with $\limfunc{length}F=1.$ Because of
Theorem \ref{tw_length_ndeg}, this simple situation is described by the
following result.

\begin{theorem}
\label{tw_mdeg_inv_length1}Let $F\in \limfunc{Aut}\left( \Bbb{C}^{2}\right) ,
$ $\limfunc{length}F=1$ and $\limfunc{mdeg}F\in \left\{ \left( 1,d\right)
,\left( d,d\right) \right\} ,$ with $1<d.$ Then 
\begin{equation*}
\limfunc{mdeg}F^{-1}\in \left\{ \left( 1,d\right) ,\left( d,1\right) ,\left(
d,d\right) \right\} .
\end{equation*}
\end{theorem}

\begin{proof}
Since $\limfunc{length}F=1,$ $F=L_{2}\circ T\circ L_{1},$ where $T$ is a
triangular automorphism of the form $T:\Bbb{C}^{2}\ni \left( x,y\right)
\mapsto \left( x,y+f(x)\right) \in \Bbb{C}^{2},$ with $\deg f>1,$ and $%
L_{1},L_{2}\in \limfunc{Aff}\left( \Bbb{C}^{2}\right) .$ Let us notice that $%
\deg f=\deg T=\deg F=d.$ Thus $\limfunc{mdeg}\left( T^{-1}\circ
L_{2}^{-1}\right) =\left( 1,d\right) .$ Now, it is easy to see that 
\begin{equation*}
\limfunc{mdeg}F^{-1}=\limfunc{mdeg}\left( L_{1}^{-1}\,\circ T^{-1}\circ
L_{2}^{-1}\right) \in \left\{ \left( 1,d\right) ,\left( d,1\right) ,\left(
d,d\right) \right\} .
\end{equation*}
\end{proof}

The following two examples show that all possibilities described in the
above theorem are realized.

\begin{example}
Let $d\in \Bbb{N}\backslash \{0,1\}.$ Let us put 
\begin{equation*}
F_{a}=T,\qquad F_{b}=T\circ L_{b},\qquad \text{and\qquad }F_{c}=T\circ L_{c},
\end{equation*}
where $T\left( x,y\right) =\left( x,y+x^{d}\right) ,L_{b}\left( x,y\right)
=\left( y,x\right) $ and $L_{c}\left( x,y\right) =\left( x+y,y\right) .$ One
can check that 
\begin{equation*}
\limfunc{mdeg}F_{a}=\limfunc{mdeg}F_{b}=\limfunc{mdeg}F_{c}=\left(
1,d\right) 
\end{equation*}
and 
\begin{equation*}
\limfunc{mdeg}F_{a}^{-1}=\left( 1,d\right) ,\qquad \limfunc{mdeg}%
F_{b}^{-1}=\left( d,1\right) ,\qquad \limfunc{mdeg}F_{c}^{-1}=\left(
d,d\right) .
\end{equation*}
\end{example}

\begin{example}
Let $d\in \Bbb{N}\backslash \{0,1\}$ and put 
\begin{equation*}
F_{a}=L_{c}\circ T,\qquad F_{b}=L_{c}\circ T\circ L_{b},\qquad \text{%
and\qquad }F_{c}=L_{c}\circ T\circ L_{c},
\end{equation*}
where $T,L_{b}$ and $L_{c}$ are defined as in the previous example. One can
check that 
\begin{equation*}
\limfunc{mdeg}F_{a}=\limfunc{mdeg}F_{b}=\limfunc{mdeg}F_{c}=\left(
d,d\right) 
\end{equation*}
and 
\begin{equation*}
\limfunc{mdeg}F_{a}^{-1}=\left( 1,d\right) ,\qquad \limfunc{mdeg}%
F_{b}^{-1}=\left( d,1\right) ,\qquad \limfunc{mdeg}F_{c}^{-1}=\left(
d,d\right) .
\end{equation*}
\end{example}

\subsection{The case $\left( d_{1},d_{2}\right) $}

Here we investigate the situation with $\limfunc{mdeg}F=\left(
d_{1},d_{2}\right) ,$ $d_{1}\neq d_{2}$ and $\limfunc{length}F>1.$ Of
course, without loss of generality, we can assume that $d_{1}<d_{2}.$
Because of Theorem \ref{tw_length_ndeg}, the situation is described by the
following two theorems.

\begin{theorem}
\label{tw_mdeg_inv_length2}Let $F\in \limfunc{Aut}\left( \Bbb{C}^{2}\right)
, $ $\limfunc{length}F=2$ and $\limfunc{mdeg}F=\left( d_{1},d_{2}\right) ,$
with $1<d_{1}<d_{2},$ $d_{1}|d_{2}.$ Then 
\begin{equation*}
\limfunc{mdeg}F^{-1}\in \left\{ \left( d_{2},\frac{d_{2}}{d_{1}}\right)
,\left( \frac{d_{2}}{d_{1}},d_{2}\right) ,\left( d_{2},d_{2}\right) \right\}
.
\end{equation*}
\end{theorem}

\begin{proof}
Since $\limfunc{length}F=2,$ $F=L_{2}\circ T_{2}\circ T_{1}\circ L_{1},$
where $T_{1},T_{2}$ are triangular (and not affine) automorphisms and $%
L_{1},L_{2}\in \limfunc{Aff}\left( \Bbb{C}^{2}\right) .$ We can assume that $%
T_{1}$ and $T_{2}$ are of the following form: 
\begin{eqnarray*}
T_{1} &:&\Bbb{C}^{2}\ni \left( x,y\right) \mapsto \left( x+f_{1}\left(
y\right) ,y\right) \in \Bbb{C}^{2}, \\
T_{2} &:&\Bbb{C}^{2}\ni \left( x,y\right) \mapsto \left( x,y+f_{2}\left(
x\right) \right) \in \Bbb{C}^{2}.
\end{eqnarray*}
Then, $\limfunc{mdeg}\left( T_{1}\circ L_{1}\right) =\left( \deg
f_{1},1\right) $ and $\limfunc{mdeg}\left( T_{2}\circ T_{1}\circ
L_{1}\right) =\left( \deg f_{1},\deg f_{2}\cdot \deg f_{1}\right) .$ Thus,
we have $\deg f_{1}=d_{1}$ and $\deg f_{2}=\frac{d_{2}}{d_{1}}.$ Now one can
easy check that 
\begin{equation*}
\limfunc{mdeg}\left( T_{2}^{-1}\circ L_{2}^{-1}\right) =\left( 1,\deg
f_{2}\right) =\left( 1,\frac{d_{2}}{d_{1}}\right)
\end{equation*}
and 
\begin{equation*}
\limfunc{mdeg}\left( T_{1}^{-1}\circ T_{2}^{-1}\circ L_{2}^{-1}\right)
=\left( \deg f_{2}\cdot \deg f_{1},\deg f_{2}\right) =\left( d_{2},\frac{%
d_{2}}{d_{1}}\right) .
\end{equation*}

Since $F^{-1}=L_{1}^{-1}\circ T_{1}^{-1}\circ T_{2}^{-1}\circ L_{2}^{-1},$
the result follows.
\end{proof}

The following example shows that all possibilities described in the above
theorem are realized.

\begin{example}
Let $d_{1},d_{2}\in \Bbb{N}$ be such that $1<d_{1}<d_{2},d_{1}|d_{2}.$ Put 
\begin{eqnarray*}
T_{1} &:&\Bbb{C}^{2}\ni \left( x,y\right) \mapsto \left(
x+y^{d_{1}},y\right) \in \Bbb{C}^{2}, \\
T_{2} &:&\Bbb{C}^{2}\ni \left( x,y\right) \mapsto \left( x,y+x^{\delta
}\right) \in \Bbb{C}^{2},
\end{eqnarray*}
where $\delta =\frac{d_{2}}{d_{1}},$ and 
\begin{equation*}
F_{a}=T_{2}\circ T_{1},\qquad F_{b}=T_{2}\circ T_{1}\circ L_{b},\qquad
F_{c}=T_{2}\circ T_{1}\circ L_{c},
\end{equation*}
where $L_{b}(x,y)=(y,x)$ and $L_{c}(x,y)=(x,y+x).$ One can check that 
\begin{equation*}
\limfunc{mdeg}F_{a}=\limfunc{mdeg}F_{b}=\limfunc{mdeg}F_{c}=\left(
d_{1},d_{2}\right)
\end{equation*}
and 
\begin{equation*}
\limfunc{mdeg}F_{a}^{-1}=\left( d_{2},\frac{d_{2}}{d_{1}}\right) ,\qquad 
\limfunc{mdeg}F_{b}^{-1}=\left( \frac{d_{2}}{d_{1}},d_{2}\right) ,\qquad 
\limfunc{mdeg}F_{c}^{-1}=\left( d_{2},d_{2}\right) .
\end{equation*}
\end{example}

\begin{theorem}
\label{tw_mdeg_inv_length3}Let $F\in \limfunc{Aut}\left( \Bbb{C}^{2}\right)
, $ $\limfunc{length}F\geq 3$ and $\limfunc{mdeg}F=\left( d_{1},d_{2}\right)
,$ with $1<d_{1}<d_{2},$ $d_{1}|d_{2}.$ Then 
\begin{equation*}
\limfunc{mdeg}F^{-1}\in \left\{ \left( d_{2},\frac{d_{2}}{a}\right) ,\left( 
\frac{d_{2}}{a},d_{2}\right) ,\left( d_{2},d_{2}\right) :a\in \mathcal{A}%
_{F}\right\} ,
\end{equation*}
where $\mathcal{A}_{F}=\left\{ a:1<a<d_{1},\text{ }a|d_{1},\text{ }l\left( 
\frac{d_{1}}{a}\right) \geq \limfunc{length}F-2\right\} .$
\end{theorem}

\begin{proof}
Let $l=\limfunc{length}F.$ Then $F$ can be written in the following form 
\begin{equation*}
F=L_{2}\circ T_{l}\circ \cdots \circ T_{1}\circ L_{1},
\end{equation*}
where $T_{1},\ldots ,T_{l}$ are triangular (and not affine) automorphisms
and $L_{1},L_{2}\in \limfunc{Aff}\left( \Bbb{C}^{2}\right) .$ We can assume
that $T_{i}$ are of the following forms 
\begin{equation*}
T_{i}:\Bbb{C}^{2}\ni \left( x,y\right) \mapsto \left( x+f_{i}\left( y\right)
,y\right) \in \Bbb{C}^{2}
\end{equation*}
for odd $i,$ and 
\begin{equation*}
T_{i}:\Bbb{C}^{2}\ni \left( x,y\right) \mapsto \left( x,y+f_{i}\left(
x\right) \right) \in \Bbb{C}^{2}
\end{equation*}
for even $i.$ Now, one can check that: 
\begin{equation*}
\limfunc{mdeg}\left( T_{l}\circ \cdots \circ T_{1}\circ L_{1}\right)
=\left\{ 
\begin{array}{ll}
\left( \prod_{j=1}^{l}\deg f_{j},\prod_{j=1}^{l-1}\deg f_{j}\right) , & 
\text{for odd }l, \\ 
\left( \prod_{j=1}^{l-1}\deg f_{j},\prod_{j=1}^{l}\deg f_{j}\right) , & 
\text{for even }l.
\end{array}
\right.
\end{equation*}
In both cases we have 
\begin{equation*}
\prod_{j=1}^{l}\deg f_{j}=d_{2}\qquad \text{and\qquad }\prod_{j=1}^{l-1}\deg
f_{j}=d_{1}.
\end{equation*}

Let $a=\deg f_{1}.$ Since $T_{i}$ are not affine, $\deg f_{i}>1.$ Since also 
$l\geq 3$ (in other words, $l-1>1$), $a$ is a proper divisor of $d_{1}$ and $%
l\left( \frac{d_{1}}{a}\right) =l\left( \deg f_{2}\cdots \deg f_{l-1}\right)
\geq l-2.$

Now, one can check that 
\begin{equation*}
\limfunc{mdeg}\left( T_{1}^{-1}\circ \cdots \circ T_{l}^{-1}\circ
L_{2}^{-1}\right) =\left( \prod_{j=1}^{l}\deg f_{j},\prod_{j=2}^{l}\deg
f_{j}\right) =\left( d_{2},\frac{d_{2}}{a}\right) .
\end{equation*}
Since $F^{-1}=L_{1}^{-1}\circ T_{1}^{-1}\circ \cdots \circ T_{l}^{-1}\circ
L_{2}^{-1},$ the result follows.
\end{proof}

Also in this case all possibilities are realized, as the following example
shows.

\begin{example}
\label{exmpl_lengh>3_d1_d2}Let $d_{1},d_{2}\in \Bbb{N}$ be such that $%
1<d_{1}<d_{2},$ $d_{1}|d_{2},$ and let $l\leq l\left( d_{1}\right) +1$ be an
even number. Assume also that $a$ is a proper divisor of $d_{1}$ such that $%
l\left( \frac{d_{1}}{a}\right) \geq l-2.$ Take positive integers $%
a_{2},\ldots ,a_{l-1}$ such that 
\begin{equation*}
d_{1}=a\cdot a_{2}\cdots a_{l-1}.
\end{equation*}
Such integers exist, because $l\left( \frac{d_{1}}{a}\right) \geq l-2.$ Now
put: 
\begin{eqnarray*}
T_{1} &:&\Bbb{C}^{2}\ni \left( x,y\right) \mapsto \left( x+y^{a},y\right)
\in \Bbb{C}^{2}, \\
T_{2} &:&\Bbb{C}^{2}\ni \left( x,y\right) \mapsto \left(
x,y+x^{a_{2}}\right) \in \Bbb{C}^{2}, \\
T_{3} &:&\Bbb{C}^{2}\ni \left( x,y\right) \mapsto \left(
x+y^{a_{3}},y\right) \in \Bbb{C}^{2}, \\
&&\vdots \\
T_{l-1} &:&\Bbb{C}^{2}\ni \left( x,y\right) \mapsto \left(
x+y^{a_{l-1}},y\right) \in \Bbb{C}^{2}, \\
T_{l} &:&\Bbb{C}^{2}\ni \left( x,y\right) \mapsto \left( x,y+x^{\delta
}\right) \in \Bbb{C}^{2},
\end{eqnarray*}
where $\delta =\frac{d_{2}}{d_{1}}.$ Let us, also, set: 
\begin{equation*}
F_{a}=T_{l}\circ \ldots \circ T_{1},\qquad F_{b}=T_{l}\circ \ldots \circ
T_{1}\circ L_{b},
\end{equation*}
and 
\begin{equation*}
F_{c}=T_{l}\circ \ldots \circ T_{1}\circ L_{c},
\end{equation*}
where $L_{b}$ and $L_{c}$ are defined as in the rpevious example. One can
check that 
\begin{equation*}
\limfunc{mdeg}F_{a}=\limfunc{mdeg}F_{b}=\limfunc{mdeg}F_{c}=\left(
d_{1},d_{2}\right)
\end{equation*}
and 
\begin{equation*}
\limfunc{length}F=l.
\end{equation*}
It is also easy to see that: 
\begin{equation*}
\limfunc{mdeg}F_{a}^{-1}=\left( d_{2},\frac{d_{2}}{a}\right) ,\qquad 
\limfunc{mdeg}F_{b}^{-1}=\left( \frac{d_{2}}{a},d_{2}\right) ,\qquad 
\limfunc{mdeg}F_{c}^{-1}=\left( d_{2},d_{2}\right) .
\end{equation*}
\end{example}

In a similar way one can obtain an example when $l$ is odd.

The following example shows an application of Theorem \ref
{tw_mdeg_inv_length3}.

\begin{example}
Let $F\in \limfunc{Aut}\left( \Bbb{C}^{2}\right) $ be such that $\limfunc{%
mdeg}F=\left( 60,120\right) .$ Since $l\left( 60\right) =l\left( 2^{2}\cdot
3\cdot 5\right) =4,$ then $\limfunc{length}F\leq 5.$

If $\limfunc{length}F=3,$ then 
\begin{equation*}
\mathcal{A}_{F}=\left\{ 2,3,5,4,6,10,15,12,20,30\right\} ,
\end{equation*}
and so, by Theorem \ref{tw_mdeg_inv_length3}, 
\begin{eqnarray*}
\limfunc{mdeg}F^{-1} &\in &\{\left( 120,60\right) ,\left( 120,40\right)
,\left( 120,24\right) ,\left( 120,30\right) ,\left( 120,20\right) , \\
&&\left( 120,12\right) ,\left( 120,8\right) ,\left( 120,10\right) ,\left(
120,6\right) ,\left( 120,4\right) ,\left( 60,120\right) , \\
&&\left( 40,120\right) ,\left( 24,120\right) ,\left( 30,120\right) ,\left(
20,120\right) ,\left( 12,120\right) . \\
&&\left( 8,120\right) ,\left( 10,120\right) ,\left( 6,120\right) ,\left(
4,120\right) ,\left( 120,120\right) \}.
\end{eqnarray*}
If $\limfunc{length}F=4,$ then 
\begin{equation*}
\mathcal{A}_{F}=\left\{ 2,3,5,4,6,10,15\right\} ,
\end{equation*}
and so, by Theorem \ref{tw_mdeg_inv_length3}, 
\begin{eqnarray*}
\limfunc{mdeg}F^{-1} &\in &\{\left( 120,60\right) ,\left( 120,40\right)
,\left( 120,24\right) ,\left( 120,30\right) ,\left( 120,20\right) , \\
&&\left( 120,12\right) ,\left( 120,8\right) ,\left( 60,120\right) ,\left(
40,120\right) ,\left( 24,120\right) , \\
&&\left( 30,120\right) ,\left( 20,120\right) ,\left( 12,120\right) ,\left(
8,120\right) ,\left( 120,120\right) \}.
\end{eqnarray*}
If $\limfunc{length}F=5,$ then 
\begin{equation*}
\mathcal{A}_{F}=\left\{ 2,3,5\right\} ,
\end{equation*}
and so, by Theorem \ref{tw_mdeg_inv_length3}, 
\begin{eqnarray*}
\limfunc{mdeg}F^{-1} &\in &\{\left( 120,60\right) ,\left( 120,40\right)
,\left( 120,24\right) , \\
&&\left( 60,120\right) ,\left( 40,120\right) ,\left( 24,120\right) ,\left(
120,120\right) \}.
\end{eqnarray*}
Moreover, by the previous example, all above listed possibilities are
realized.
\end{example}

\subsection{The case $\left( d,d\right) $}

Using a similar arguments as in the proof of Theorem \ref
{tw_mdeg_inv_length3} one can prove the following

\begin{theorem}
Let $F\in \limfunc{Aut}\left( \Bbb{C}^{2}\right) ,$ $\limfunc{length}F\geq 2$
and $\limfunc{mdeg}F=\left( d,d\right) ,$ with $1<d.$ Then:\newline
\begin{equation*}
\limfunc{mdeg}F^{-1}\in \left\{ \left( d,\frac{d}{a}\right) ,\left( \frac{d}{%
a},d\right) ,\left( d,d\right) :a\in \mathcal{A}_{F}\right\} ,
\end{equation*}
where $\mathcal{A}_{F}=\left\{ a:1<a<d,\text{ }a|d,\text{ }l\left( \frac{d}{a%
}\right) \geq \limfunc{length}F-1\right\} .$
\end{theorem}

Also in this case all described possibilities are realized, as the following
example shows (this example is a modification of the example given after
Theorem \ref{tw_mdeg_inv_length3}).

\begin{example}
Let $d\in \Bbb{N}$ and $l\geq 2$ be a even number such that let $l\leq
l\left( d\right) .$ Assume, also, that $a$ is a proper divisor of $d$ such
that $l\left( \frac{d}{a}\right) \geq l-1.$ Take positive integers $%
a_{2},\ldots ,a_{l}$ such that 
\begin{equation*}
d=a\cdot a_{2}\cdots a_{l}.
\end{equation*}
Such integers exist, because $l\left( \frac{d}{a}\right) \geq l-1.$ Let $%
T_{1},\ldots ,T_{l-1}$ be defined as in Example \ref{exmpl_lengh>3_d1_d2}
and put 
\begin{equation*}
T_{l}:\Bbb{C}^{2}\ni \left( x,y\right) \mapsto \left( x,y+x^{a_{l}}\right)
\in \Bbb{C}^{2}.
\end{equation*}
Let us also set: 
\begin{equation*}
F_{a}=L\circ T_{l}\circ \ldots \circ T_{1},\qquad F_{b}=L\circ T_{l}\circ
\ldots \circ T_{1}\circ L_{b},
\end{equation*}
and 
\begin{equation*}
F_{c}=L\circ T_{l}\circ \ldots \circ T_{1}\circ L_{c},
\end{equation*}
where $L_{b}\left( x,y\right) =\left( y,x\right) ,L_{c}\left( x,y\right)
=\left( x,y+x\right) $ and $L\left( x,y\right) =\left( x+y,y\right) .$ Then
one can check that 
\begin{equation*}
\limfunc{mdeg}F_{a}=\limfunc{mdeg}F_{b}=\limfunc{mdeg}F_{c}=\left(
d,d\right) ,\qquad \limfunc{length}F=l,
\end{equation*}
and 
\begin{equation*}
\limfunc{mdeg}F_{a}^{-1}=\left( d,\frac{d}{a}\right) ,\qquad \limfunc{mdeg}%
F_{b}^{-1}=\left( \frac{d}{a},d\right) ,\qquad \limfunc{mdeg}%
F_{c}^{-1}=\left( d,d\right) .
\end{equation*}
\end{example}

\vspace{1cm}

\textsc{Marek Kara\'{s}\newline
Instytut Matematyki\newline
Uniwersytetu Jagiello\'{n}skiego\newline
ul. \L ojasiewicza 6}\newline
\textsc{30-348 Krak\'{o}w\newline
Poland\newline
} e-mail: Marek.Karas@im.uj.edu.pl

\end{document}